\documentclass[a4paper,reqno]{amsart}
\usepackage{amsmath}
\usepackage{amsfonts}
\usepackage{amssymb}
\usepackage{amscd}
\usepackage{url}
\usepackage{enumerate}
\usepackage{graphicx}
\usepackage[pdftex,bookmarks=true]{hyperref}

\newcommand{\indicator}[1]{\ensuremath{\mathbf{1}_{\{#1\}}}}
\newcommand{\oindicator}[1]{\ensuremath{\mathbf{1}_{#1}}}

\renewcommand{\Re}{\mathop{\rm Re}}
\renewcommand{\Im}{\mathop{\rm Im}}

\newcommand\eps{{\varepsilon}}
\newcommand\R{{\mathbb{R}}}
\newcommand\C{{\mathbb{C}}}
\newcommand\Z{{\mathbb{Z}}}
\newcommand\HQ{{\mathbb{H}}}
\newcommand\SQ{{\mathbb{S}}}

\newcommand\row{{\mathbf{r}}}

\DeclareMathOperator{\Vol}{Vol}
\DeclareMathOperator{\rank}{rank}
\DeclareMathOperator{\dist}{dist}

\DeclareMathOperator{\var}{Var}
\DeclareMathOperator{\tr}{tr}

\newcommand\Prob{{ \mathbb{P}}}
\renewcommand\P{{ \mathbb{P}}}
\newcommand\E{{ \mathbb{E} }}

\newcommand\Ba{{\mathbf a}}

\newcommand\Bc{{\mathbf c}}

\newcommand\Bi{{\mathbf i}}
\newcommand\Bj{{\mathbf j}}
\newcommand\Bk{{\mathbf k}}

\newcommand\Br{{\mathbf r}}
\newcommand\Bs{{\mathbf s}}
\newcommand\Bt{{\mathbf t}}
\newcommand\Bu{{\mathbf u}}
\newcommand\Bv{{\mathbf v}}
\newcommand\Bw{{\mathbf w}}
\newcommand\Bx{{\mathbf x}}
\newcommand\By{{\mathbf y}}

\newcommand\BA{{\mathbf A}}
\newcommand\BB{{\mathbf B}}
\newcommand\BC{{\mathbf C}}
\newcommand\BD{{\mathbf D}}

\newcommand\BG{{\mathbf G}}
\newcommand\BH{{\mathbf H}}
\newcommand\BI{{\mathbf I}}

\newcommand\BM{{\mathbf M}}
\newcommand\BN{{\mathbf N}}

\newcommand\BQ{{\mathbf Q}}
\newcommand\BR{{\mathbf R}}

\newcommand\BX{{\mathbf X}}

\newcommand\Bzero{{\mathbf 0}}

\newcommand\ep{\varepsilon}

\numberwithin{equation}{section}

\theoremstyle{plain}
  \newtheorem{theorem}[subsection]{Theorem}
  \newtheorem{conjecture}[subsection]{Conjecture}
  
  \newtheorem{assumption}[subsection]{Assumption}

  \newtheorem{fact}[subsection]{Fact}
  \newtheorem{lemma}[subsection]{Lemma}
  \newtheorem{corollary}[subsection]{Corollary}
  
  \newtheorem{example}[subsection]{Example}

  \newtheorem{claim}[subsection]{Claim}

\theoremstyle{definition}
  \newtheorem{definition}[subsection]{Definition}

\begin{document}

\title[Anti-concentration and universality]{On the concentration of random multilinear forms and the universality of random block matrices}

\author{Hoi H. Nguyen}
\thanks{H. Nguyen is partly supported by research grant DMS-1358648}
\address{Department of Mathematics, The Ohio State University, Columbus, OH 43210, USA}
\email{nguyen.1261@math.osu.edu}

\author{Sean O'Rourke}
\thanks{S. O'Rourke is supported by grant AFOSAR-FA-9550-12-1-0083.}
\address{Department of Mathematics, Yale University, New Haven , CT 06520, USA  }
\email{sean.orourke@yale.edu}


\begin{abstract}
The circular law asserts that if $\BX_n$ is a $n \times n$ matrix with iid complex entries of mean zero and unit variance, then the empirical spectral distribution of $\frac{1}{\sqrt{n}} \BX_n$ converges almost surely to the uniform distribution on the unit disk  as $n$ tends to infinity.  Answering a question of Tao, we prove the circular law for a general class of random block matrices with dependent entries.  The proof relies on an inverse-type result for the concentration of linear operators and multilinear forms.  
\end{abstract}

\maketitle

\section{Introduction}\label{section:introduction}

The \emph{eigenvalues} of a $n \times n$ matrix $\BM$ are the roots in $\mathbb{C}$ of the characteristic polynomial $\det(\BM-z\BI)$, where $\BI$ is the identity matrix.  We let $\lambda_1(\BM), \ldots, \lambda_n(\BM)$ denote the eigenvalues of $\BM$.  In this case, the \emph{empirical spectral measure} of $\BM$ is given by
$$ \mu_{\BM} := \frac{1}{n} \sum_{i = 1}^n \delta_{\lambda_i(\BM)}. $$
The corresponding \emph{empirical spectral distribution} (ESD) is given by
$$ F^{\BM}(x,y) := \frac{1}{n} \# \left\{1 \leq i \leq n: \Re(\lambda_i(\BM)) \leq x, \Im(\lambda_i(\BM)) \leq y \right\}. $$
Here $\# E$ denotes the cardinality of the set $E$.  If the matrix $\BM$ is Hermitian, then the eigenvalues $\lambda_1(\BM), \ldots, \lambda_n(\BM)$ are real.  In this case the ESD is given by 
$$ F^{\BM}(x) := \frac{1}{n} \# \left\{ 1 \leq i \leq n : \lambda_i(\BM) \leq x \right\}. $$

Given a random $n \times n$ matrix $\BX_n$, an important problem in random matrix theory is to study the limiting distribution of the empirical spectral measure as $n$ tends to infinity.  We consider one of the simplest random matrix ensembles, when the entries of $\BX_n$ are iid copies of the random variable $\xi$.  We refer to $\xi$ as the {\it atom variable} of $\BX_n$.  

When $\xi$ is a standard complex Gaussian random variable, $\BX_n$ can be viewed as a random matrix drawn from the probability distribution
$$ \Prob(d \BM) = \frac{1}{\pi^{n^2}} e^{- \tr (\BM \BM^\ast)} d \BM $$
on the set of complex $n \times n$ matrices.  Here $d \BM$ denotes the Lebesgue measure on the $2n^2$ real entries
$$ \{ \Re(m_{ij}) : 1 \leq i, j \leq n\} \cup \{ \Im(m_{ij}) : 1 \leq i, j \leq n\} $$
of $\BM=(m_{ij})_{i,j=1}^n$.  
This is known as the {\it complex Ginibre ensemble}.  The {\it real Ginibre ensemble} and {\it quaternionic Ginibre ensemble} are defined analogously.  

Following Ginibre \cite{Gi}, one may compute the joint density of the eigenvalues of a random matrix $\BX_n$ drawn from the complex Ginibre ensemble.  Mehta  \cite{M,M:B} used this joint density function to compute the limiting spectral measure of the complex Ginibre ensemble.  In particular, he showed that if $\BX_n$ is drawn from the complex Ginibre ensemble, then the ESD of $\frac{1}{\sqrt{n}} \BX_n$ converges to the {\it circular law} $F_{\mathrm{circ}}$, where
$$ F_{\mathrm{circ}}(x,y) := \mu_{\mathrm{circ}} \left( \left\{ z \in \mathbb{C} : \Re(z) \leq x, \Im(z) \leq y \right\} \right) $$
and $\mu_{\mathrm{circ}}$ is the uniform probability measure on the unit disk in the complex plane.  Edelman \cite{Ed-cir} verified the same limiting distribution for the real Ginibre ensemble.

For the general (non-Gaussian) case, there is no formula for the joint distribution of the eigenvalues and the problem appears much more difficult.  The universality phenomenon in random matrix theory asserts that the spectral behavior of a random matrix does not depend on the distribution of the atom variable $\xi$ in the limit $n \rightarrow \infty$.  In other words, one expects that the circular law describes the limiting ESD of a large class of random matrices (not just Gaussian matrices); Figure \ref{fig:universal} presents a numerical simulation depicting this universality phenomenon.     

\begin{figure}
	\begin{center}
	\includegraphics[width=6.25cm]{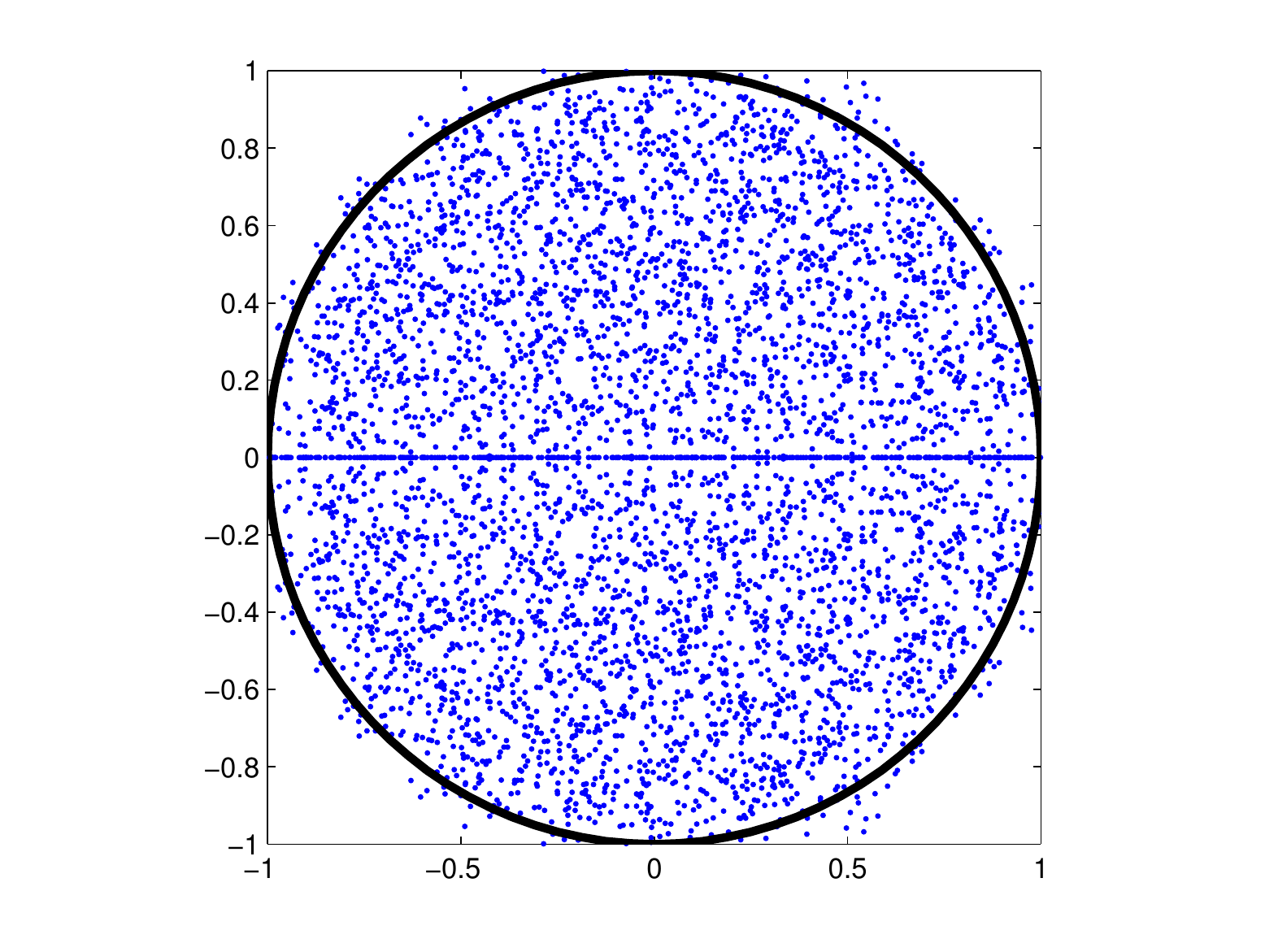}
	\includegraphics[width=6.25cm]{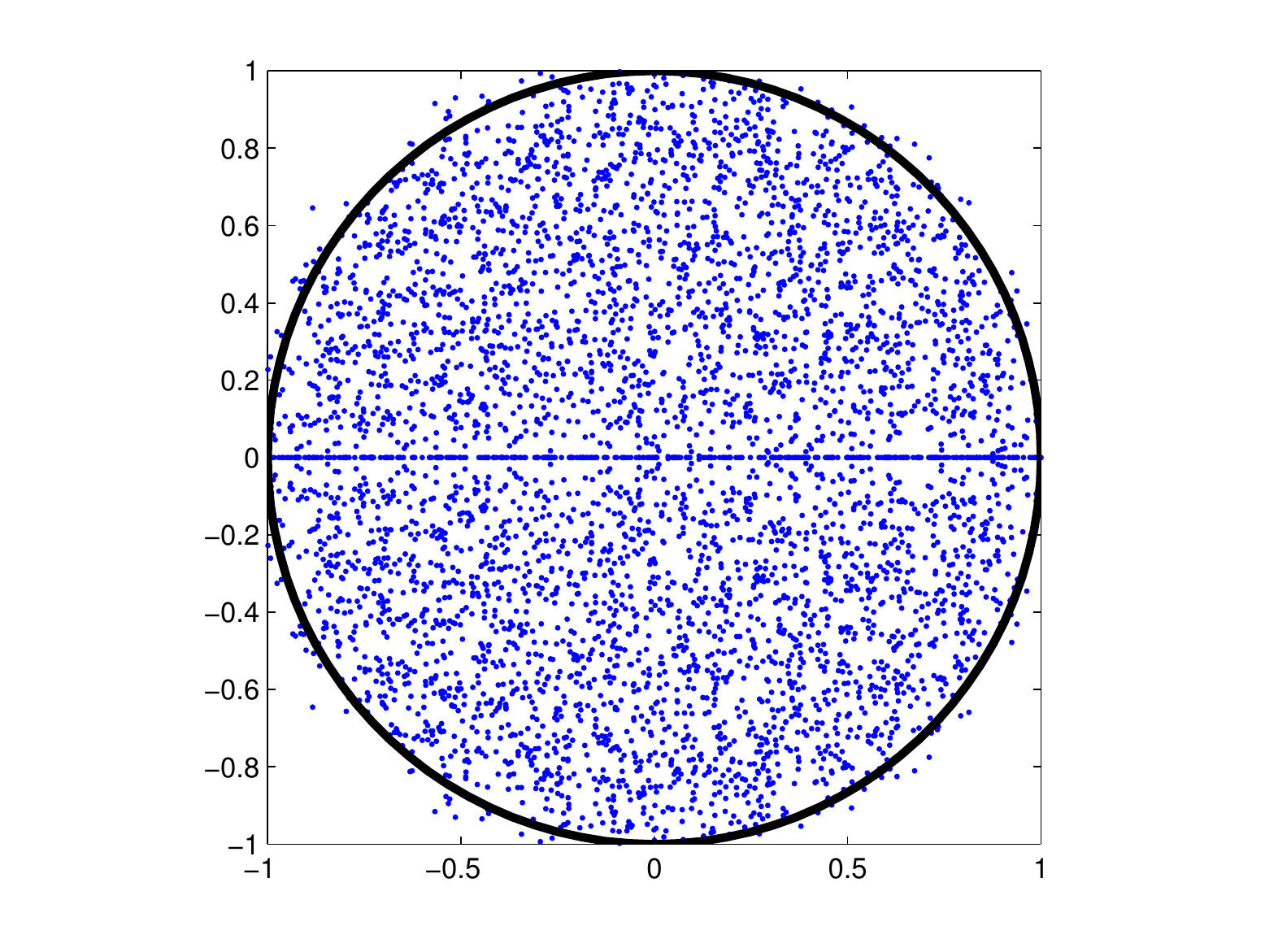}
	\caption{The eigenvalues of random matrices with iid entries.  The first plot contains the eigenvalues of $50$ samples of $100 \times 100$ random matrices drawn from the real Ginibre ensemble.  The second plot contains the eigenvalues of $50$ samples of $100 \times 100$ random matrices whose entries are Bernoulli random variables (i.e. each entry takes values $\pm 1$ with equal probability).  The black circle in each plot is the unit circle of radius one centered at the origin.}
   	\label{fig:universal}
	\end{center}
\end{figure}

In the 1950's, Wigner \cite{W} proved a version of the universality phenomenon for Hermitian random matrices, now known as Wigner matrices.  However, the random matrix ensemble described above is not Hermitian.  In fact, many of the techniques used to deal with Hermitian random matrices do not apply to non-Hermitian matrices \cite[Section 11.1]{BSbook}.  

An important result was obtained by Girko \cite{G1,G2} who related the empirical spectral measure of non-Hermitian matrices to that of Hermitian matrices.  Building upon this Hermitization technique, Bai \cite{B,BSbook} gave the first rigorous proof of the circular law for general (non-Gaussian) distributions under a number of moment and smoothness assumptions on the atom variable $\xi$.  Important results were obtained more recently by Pan and Zhou \cite{PZ} and G\"otze and Tikhomirov \cite{GTcirc}.  Tao and Vu \cite{TVcirc} were able to prove the circular law under the assumption that $\E|\xi|^{2+\eps} < \infty$, for some $\eps > 0$.  Recently, Tao and Vu \cite{TVbull,TVesd} established the law assuming only that $\xi$ has finite variance.

For any $m \times n$ matrix $\BM$, we denote the Hilbert-Schmidt norm $\|\BM\|_2$ by the formula
\begin{equation} \label{eq:def:hs}
	\|\BM\|_2 := \sqrt{ \tr (\BM \BM^\ast) } = \sqrt{ \tr (\BM^\ast \BM)}. 
\end{equation}

\begin{theorem}[Tao-Vu, \cite{TVesd}] \label{thm:tvcirc}
Let $\xi$ be a complex random variable with mean zero and unit variance.  For each $n \geq 1$, let $\BX_n$ be a $n \times n$ matrix whose entries are iid copies of $\xi$, and let $\BN_n$ be a $n \times n$ deterministic matrix.  If $\rank(\BN_n) = o(n)$ and $\sup_{n \geq 1} \frac{1}{n^2} \| \BN_n \|_2^2 < \infty$, then the ESD of $\frac{1}{\sqrt{n}} (\BX_n + \BN_n)$ converges almost surely to the circular law $F_{\mathrm{circ}}$ as $n \rightarrow \infty$\footnote{Here, and throughout the paper, we use asymptotic notation such as $O,o$ under the assumption that $n \to \infty$.  See Section \ref{sec:notation} for a complete description of our asymptotic notation.}.  
\end{theorem}

One of the key steps in proving Theorem \ref{thm:tvcirc} is controlling the largest and smallest singular values of $\BX_n + \BN_n$.  We recall that the \emph{singular values} of a $m \times n$ matrix $\BM$ are the eigenvalues of $|\BM| := \sqrt{ \BM^\ast \BM}$.  We let $\sigma_1(\BM) \geq \cdots \geq \sigma_n(\BM) \geq 0$ denote the singular values of $\BM$.  In particular, the largest and smallest singular values are given by 
$$ \sigma_1(\BM) := \sup_{\|x\| = 1} \| \BM x \|, \qquad \sigma_n(\BM) := \inf_{\|x \| = 1} \| \BM x \|, $$
where $\|v\|$ denotes the Euclidean norm of the vector $v$.  We let $\| \BM\| := \sigma_1(\BM)$ denote the spectral norm of the matrix $\BM$.  

While the behavior of the largest singular value is well studied (e.g. see \cite{BSY, Suniv}), bounds for the smallest singular value appear more difficult.  Using techniques from additive combinatorics, Tao and Vu established the following bound on the least singular value of $\BX_n + \BN_n$.  

\begin{theorem}[Tao-Vu, \cite{TVcirc}]  \label{theorem:least-sing:TV} Assume that $\BX_n$ is an $n \times n$ random matrix whose entries are iid copies of a random variable with mean zero and variance one. Assume that $\BN_n$ is a deterministic $n \times n$ matrix whose entries are bounded by $n^\alpha$ in absolute value. Then for any $B>0$, there exists $A>0$ (depending on $B$ and $\alpha$) such that 
$$\P(\sigma_{n}(\BX_n +\BN_n)\le n^{-A}) = O(n^{-B}).$$
\end{theorem}

\section{Universality of random block matrices}

The goal of this note is to study a class of random matrices that generalizes the random matrix ensemble discussed above.  In particular, we consider random block matrices whose entries are not necessarily independent.  We will show that, under some moment assumptions, the limiting ESD of these block matrices is also given by the circular law.

\subsection{Quaternions and matrices of quaternions}

One of the prototypical examples of a block matrix is that of a quaternionic matrix.  We now review some preliminary facts on quaternions and matrices of quaternions.  Most of the these results can be found in the detailed survey by Zhang \cite{Z}.  Let $\HQ$ denote the non-commutative field of quaternions.  As a real vectors space $\HQ$ admits a basis $\{1,\Bi,\Bj,\Bk\}$ with the usual multiplicative table: $1$ is the identity element and
$$ \Bi^2 = \Bj^2 = \Bk^2 = -1, \quad \Bi \Bj = -\Bj \Bi = \Bk, \quad \Bj \Bk = - \Bk \Bj = \Bi, \quad \Bk \Bi = - \Bi \Bk = \Bj. $$
For $q = q_0 + \Bi q_1 + \Bj q_2 + q_3 \Bk \in \HQ$, we have $q^\ast := q_0 - q_1 \Bi - q_2 \Bj - q_3 \Bk$, $\Re(q) := q_0$, and $\Im(q) := q_1 \Bi + q_2 \Bj + q_3 \Bk$.  Then
$$ q q^\ast = q_0^2 + q_1^2 + q_2^2 + q_3^2, $$
and thus any nonzero quaternion is invertible.  Define the norm $|q| : = \sqrt{q q^\ast}$.  It follows that for any $q,q' \in \HQ$, $|q q'| = |q| |q'|$.  Real numbers and complex numbers can be thought of as quaternions in the natural way, and one has $\R \subset \C \subset \HQ$.  Every quaternion $q = q_0 + q_1 \Bi + q_2 \Bj + q_3 \Bk$ can be written uniquely as $q = c_1 + c_2\Bj$ where $c_1 = q_0 + q_1\Bi$, $c_2 = q_2 + q_3 \Bi$ are complex numbers.  

We say that two quaternions $q, q'$ are {\it similar} if there exists a nonzero quaternion $x$ such that $q = x q' x^{-1}$.  We let $ \SQ(\HQ)$ denote the group of quaternions with norm one.  It follows that $q,q'$ are similar if and only if there exists $x \in \SQ(\HQ)$ with $q = x q' x^\ast$.  The following lemma shows that every quaternion is similar to a complex number.  

\begin{lemma}[\cite{Z}] \label{lemma:similarquaternion}
If $q = q_0 + q_1 \Bi + q_2 \Bj + q_3 \Bk \in \HQ$, then $q$ and $\Re(q) + |\Im(q)| \Bi$ are similar. 
\end{lemma}

Let $\BM$ be a $n \times n$ matrix with quaternion entries.  Then $\lambda \in \HQ$ is called a {\it right eigenvalue} of $\BM$ if there exists a nonzero vector $X \in \HQ^n$ such that $\BM X = X \lambda$.  If $\lambda$ is a right eigenvalue of $\BM$, one finds that $q \lambda q^{-1}$ is also a right eigenvalue of $\BM$ for any nonzero quaternion $q$.  Hence the right spectrum of $\BM$ is either infinite or contained in $\R$.  From Lemma \ref{lemma:similarquaternion}, we restrict our attention to complex right eigenvalues.  We consider the (unique) decomposition $\BM = \BM_1 + \BM_2 \Bj$.  Then for any $\lambda \in \C$ and $X = Y + Z \Bj$ with $Y,Z \in \C^n$, the following are equivalent:
\begin{enumerate}[(i)]
\item $\BM X = X \lambda$,
\item $\begin{bmatrix} \BM_1 & \BM_2 \\ -\overline{\BM}_2 & \overline{\BM}_1 \end{bmatrix} \begin{bmatrix} Y \\ - \overline{Z} \end{bmatrix} = \lambda \begin{bmatrix} Y \\ -\overline{Z} \end{bmatrix}$,
\item $\begin{bmatrix} \BM_1 & \BM_2 \\ -\overline{\BM}_2 & \overline{\BM}_1 \end{bmatrix} \begin{bmatrix} Z \\ \overline{Y} \end{bmatrix} = \bar{\lambda} \begin{bmatrix} Z \\ \overline{Y} \end{bmatrix}$.
\end{enumerate}
Thus, the right spectrum of $\BM$, when restricted to complex numbers, is given by the $2n$ eigenvalues of the complex matrix 
$$ \begin{bmatrix} \BM_1 & \BM_2 \\ -\overline{\BM}_2 & \overline{\BM}_1 \end{bmatrix}. $$
Moreover, the complex eigenvalues appear as conjugate pairs.  The whole set of right eigenvalues of $\BM$ is then the union of all similarity classes of the complex right eigenvalues of $\BM$.

\subsection{Random quaternionic matrices}

Let $\xi$ be a real random variable with mean zero and variance $1/4$.  We study the right eigenvalues of random quaternion matrices whose entries are iid copies of $\xi_0 + \xi_1 \Bi + \xi_2 \Bj + \xi_3 \Bk$, where $\xi_0, \xi_1, \xi_2, \xi_3$ are iid copies of $\xi$.  From the discussion above, we find that this is equivalent to studying the eigenvalues of random complex block matrices.  Indeed, the problem reduces to studying the eigenvalues of the $2n \times 2n$ matrix
\begin{equation} \label{eq:xnquat}
	\BX_n = \begin{bmatrix} \BA_n & \BB_n \\ - \overline{\BB}_n & \overline{\BA}_n \end{bmatrix}, 
\end{equation}
where $\BA_n, \BB_n$ are independent $n \times n$ complex matrices whose entries are iid copies of $\xi_0 + \xi_1 \Bi$.  We note, however, that the entries of $\BX_n$ are not independent.  Thus, Theorem \ref{thm:tvcirc} cannot be applied to the block matrix $\BX_n$.

In the case that $\xi$ is Gaussian (i.e. the quaternionic Ginibre ensemble), the circular law was established by Benaych-Georges and Chapon \cite{BCquat} using logarithmic potential theory.  We will verify the circular law for random quaternionic matrices when the atom variable $\xi$ is non-Gaussian.  

\begin{theorem}[Universality for quaternion random matrices] \label{thm:quaternion}
Let $\xi$ be a complex random variable with mean zero and variance $1/2$, and suppose $\E[\xi^2] = 0$ and $\E|\xi|^{2+\eta} < \infty$ for some $\eta > 0$.  For each $n \geq 1$, let $\BA_n, \BB_n$ be independent $n \times n$ matrices whose entries are iid copies of $\xi$, and let $\BX_n$ be the $2n \times 2n$ matrix defined in \eqref{eq:xnquat}.  For each $n \geq 1$, let $\BN_n$ be a deterministic $2n \times 2n$ matrix, and suppose the sequence $\{\BN_n\}_{n \geq 1}$ satisfies $\rank(\BN_n) = O(n^{1-\eps})$ and $\sup_{n \geq 1} \frac{1}{n^2} \|\BN_n\|_2^2 < \infty$, for some $\eps > 0$. Then the ESD of $\frac{1}{\sqrt{n}} (\BX_n + \BN_n)$ converges almost surely to the circular law $F_{\mathrm{circ}}$ as $n \rightarrow \infty$.  
\end{theorem}

\subsection{Random block matrices} More generally, we will study random block matrices of the form 
\begin{equation} \label{eq:2x2block}
	\BX_n = \begin{bmatrix} \BA_n & \BB_n \\ \BC_n & \BD_n \end{bmatrix}, 
\end{equation}
where $\BA_n = (a_{ij})_{i,j=1}^n, \BB_n = (b_{ij})_{i,j=1}^n, \BC_n = (c_{ij})_{i,j=1}^n, \BD_n = (d_{ij})_{i,j=1}^n$, and 
$$ \{(a_{ij}, b_{ij}, c_{ij}, d_{ij}) : 1 \leq i,j \leq n \} $$
is a collection of iid copies of the random vector $(\xi_1,\xi_2,\xi_3,\xi_4)$.  Here the random variables $\xi_1, \xi_2, \xi_3, \xi_4$ are not required to be independent.  

This ensemble of block matrices was proposed by Tao at the AIM Workshop on Random Matrices as a matrix model with dependent entries in which the circular law is still expected to hold\footnote{See \url{http://www.aimath.org/WWN/randommatrices/randommatrices.pdf}.}.  We will prove the circular law for this ensemble of random block matrices under some moment assumptions on the atom variables $\xi_1,\xi_2,\xi_3,\xi_4$.  

The matrix $\BX_n$ in \eqref{eq:2x2block} can be viewed as a $2 \times 2$ block matrix.  More generally, we will study $d \times d$ block matrices for any $d \geq 2$.  We begin with the following definition.  

\begin{definition} [Random block matrices with dependent entries; Condition {\bf C0}]  \label{def:C0} \label{cond:covariance}
Let $d \geq 2$.  Let $(\xi_{st})_{s,t=1}^d$ be a complex random matrix where each entry $\xi_{st}$ has mean zero and variance $1/d$.  For each $s,t \in \{1,\ldots,d\}$, let $\{ x_{st;ij}\}_{i,j \geq 1}$ be an infinite double array of complex random variables all defined on the same probability space.  For each $n \geq 1$ and all $s,t \in \{1,\ldots,d\}$, define the $n \times n$ random matrix $\BX_{n,st} := (x_{st;ij})_{i,j=1}^n$.  Define the $dn \times dn$ random block matrix 
$$ \BX_n := \begin{bmatrix} \BX_{n,11} & \ldots & \BX_{n,1d} \\ \vdots & \ddots & \vdots \\ \BX_{n,d1} & \ldots & \BX_{n,dd} \end{bmatrix} = \left(\BX_{n,st} \right)_{s,t=1}^d. $$
We say the sequence of matrices $\{\BX_n\}_{n \geq 1}$ satisfies condition {\bf C0} with parameter $d$ and atom variables $(\xi_{st})_{s,t=1}^d$ if the following conditions hold:
\begin{enumerate}[(i)]
\item $\{ (x_{st;ij})_{s,t=1}^d : 1 \leq i,j \}$ is a collection of iid copies of $(\xi_{st})_{s,t=1}^d$,  
\item We have $\E\left[ \xi_{st} \overline{\xi_{uv}}\right] = 0$ for all $(s,t) \neq (u,v)$.  \label{def:C0:uncorr}
\end{enumerate}
\end{definition}

In Theorem \ref{thm:main} below, we establish the circular law for a class of random block matrices that satisfy condition {\bf C0}.  In particular, Theorem \ref{thm:quaternion} is a corollary of the following theorem in the case that $d=2$.

\begin{theorem}[Universality for random block matrices] \label{thm:main} 
Let $\{\BX_n\}_{n \geq 1}$ be a sequence of random matrices that satisfies condition {\bf C0} with parameter $d \geq 2$ and atom variables $(\xi_{st})_{s,t=1}^d$, and assume that
$$ \max_{1 \leq s,t \leq d} \E|\xi_{st}|^{2+\eta} < \infty $$
for some $\eta > 0$.  For each $n \geq 1$, let $\BN_n$ be a deterministic $dn \times dn$ matrix, and suppose the sequence $\{\BN_n\}_{n \geq 1}$ satisfies $\rank(\BN_n) = O(n^{1-\eps})$ and $\sup_{n\geq 1} \frac{1}{n^2} \|\BN_n\|_2^2 < \infty$ for some $\eps > 0$.  Then the ESD of $\frac{1}{\sqrt{n}} (\BX_n + \BN_n)$ converges almost surely to the circular law $F_\mathrm{circ}$ as $n \rightarrow \infty$.
\end{theorem}

In Definition \ref{def:C0}, we require the atom variables $(\xi_{st})_{s,t=1}^d$ to be uncorrelated.  In this note, we will not deal with the correlated case.  However, when there is a correlation among the atom variables, we do not always expect the circular law to be the limiting distribution.  In Figure \ref{fig:unknown}, we plot the eigenvalues of $\frac{1}{\sqrt{2n}} \BX_n$ in the case that 
\begin{equation} \label{eq:unknown}
	\BX_n = \begin{bmatrix} \BA_n & \BA_n \\ \BA_n & \BB_n \end{bmatrix}, 
\end{equation}
where $\BA_n, \BB_n$ are independent $n \times n$ random matrices drawn from the real Ginibre ensemble.  In particular, $\BX_n$ does not satisfy condition \eqref{def:C0:uncorr} of Definition \ref{def:C0}.  Figure \ref{fig:unknown} predicts that more of the eigenvalues will concentrate near the origin, and so we do not believe the limiting distribution will be uniform on the unit disk.  

\begin{figure} 
	\begin{center}
  	\includegraphics[width=10cm]{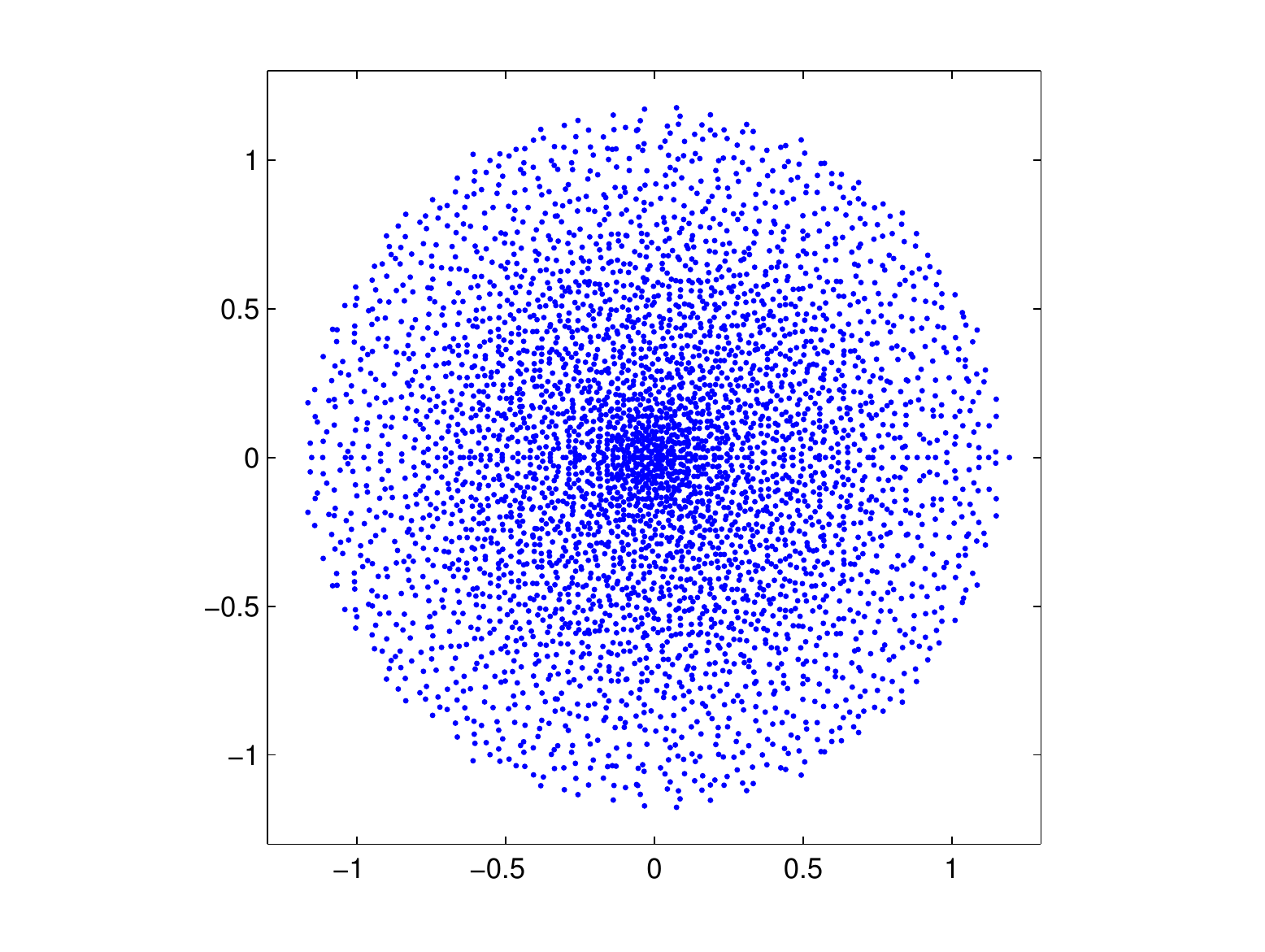}
   	\caption{The eigenvalue plot of $\frac{1}{\sqrt{2n}} \BX_n$, when $n=2000$, $\BX_n$ is defined in \eqref{eq:unknown}, and $\BA_n, \BB_n$ are independent $n \times n$ random matrices drawn from the real Ginibre ensemble.  The eigenvalues appear to concentrate near the origin.}
   	\label{fig:unknown}
 	\end{center}
\end{figure}

\subsection{Least singular value bound}
One of the key ingredients in the proof of Theorem \ref{thm:main} is a bound on the least singular value of random matrices $\{\BX_n\}_{n \geq 1}$ that satisfy condition {\bf C0}.  In particular, we establish the following result.

\begin{theorem}[Least singular value bound] \label{thm:least-sing-value}
Let $\{\BX_n\}_{n \geq 1}$ be a sequence of random matrices that satisfies condition {\bf C0} with parameter $d \geq 2$ and atom variables $(\xi_{st})_{s,t=1}^d$, and assume that
$$ \max_{1 \leq s,t \leq d} \E|\xi_{st}|^{2+\eta} < \infty $$
for some $\eta > 0$.  For each $n \geq 1$, let $\BN_n$ be a deterministic $dn \times dn$ matrix whose entries are bounded by $n^{\alpha}$ for some $\alpha > 0$.  Then, for every $B > 0$, there exist $A>0$  (depending only on $d, B, \alpha$) such that
$$ \Prob ( \sigma_{dn} (\BX_n + \BN_n) \leq n^{-A} )=O(n^{-B}) .$$  
\end{theorem}

\subsection{Overview}
The proof of Theorem \ref{thm:least-sing-value} requires studying an inverse Littlewood-Offord problem for random multilinear forms.  To this end, we introduce the Littlewood-Offord problem and random multilinear forms in Section \ref{section:anti-concentration}.  Sections \ref{section:singularity:approach}, \ref{section:step1}, and \ref{section:singularity:step2} contain the proof of Theorem \ref{thm:least-sing-value}.  Finally, we prove Theorem \ref{thm:main} in Section \ref{sec:universality}.  A number of auxiliary results are contained in the appendix.  

\subsection{Notation} \label{sec:notation}
We use asymptotic notation (such as $O,o,\Omega, \asymp$) under the assumption that $n \rightarrow \infty$.  We use $X \ll Y, Y \gg X, Y=\Omega(X)$, or $X = O(Y)$ to denote the bound $X \leq CY$ for all sufficiently large $n$ and for some constant $C$.  Notations such as $X \ll_k Y$ and $X=O_k(Y)$ mean that the hidden constant $C$ depends on another constant $k$.  $X=o(Y)$ or $Y=\omega(X)$ means that $X/Y \rightarrow 0$ as $n \rightarrow \infty$.  

We let $\|\BM\|_2$ denote the Hilbert-Schmidt norm of $\BM$ (defined in \eqref{eq:def:hs}), and let $\|\BM\|$ denote the spectral norm of $\BM$.  For a vector $\Bv$, we let $\|\Bv\| = \|\Bv\|_2$ denote the Euclidean norm of $\Bv$.  

We let $\BI_n$ denote the $n \times n$ identity matrix.  Often we will just write $\BI$ for the identity matrix when the size can be deduced from the context.  Similarly, we let $\Bzero$ denote the zero matrix.  

For an event $E$, we let $\oindicator{E}$ denote the indicator function of the event $E$.  We write a.s., a.a., and a.e. for almost surely, Lebesgue almost all, and Lebesgue almost everywhere respectively.  We use $\sqrt{-1}$ to denote the imaginary unit and reserve $i$ as an index.

\subsection*{Acknowledgments}
The authors are grateful to the anonymous referees for valuable comments and suggestions.

\section{The Littlewood-Offord problem and random multilinear forms}\label{section:anti-concentration}

In this section, we introduce the Littlewood-Offord problem and some anti-concentration results for random multilinear forms, which will be used to prove Theorem \ref{thm:least-sing-value}.  

\subsection{The Littlewood-Offord problem}
Consider $\xi$ a real random variable with mean zero and unit variance.  A large portion of classical probability theory is devoted to studying random sums $S_\xi(A) := \sum_{i=1}^n a_i x_i$,  where $A =\{a_1,\ldots, a_n\}$ is a multiset of complex vectors in $\C^d$ and $x_1,\ldots,x_n$ are iid copies of $\xi$. The {\it Littlewood-Offord problem} is to estimate the {\it small ball probability}
$$ \rho_{\beta,\xi}(A) := \sup_{z \in \C^d} \Prob ( \|S_\xi(A) - z\| \leq \beta). $$
In particular, if $\rho_{\beta,\xi}(A)$ is small, then the random sum $S_\xi(A)$ is well spread.  Conversely, if $\rho_{\beta,\xi}(A)$ is large, then the random sum concentrates near a point. 

A classical result of Littlewood and Offord \cite{LO}, which was strengthened by Erd\H{o}s \cite{E}, gives an estimate for the small ball probability when $\xi$ is a Bernoulli random variable (takes values $\pm 1$ each with probability $1/2$) and $d=1$.  

\begin{theorem}[Erd\H{o}s, \cite{E}] \label{thm:erdos}
Let $\xi$ be a Bernoulli random variable.  If the complex numbers $a_i$ satisfy $|a_i| \geq 1$ for all $i$, then 
$$ \rho_{1,\xi}(A) = O(n^{-1/2}). $$
\end{theorem}


The reader is invited to consult \cite{NgV-survey} and references therein for further extensions of this result. Motivated by inverse theorems from additive combinatorics, Tao and Vu \cite{TVinverse} consider the following phenomenon: 
\begin{quote}
	If $\rho_{\beta,\xi}(A)$ is large, then most of the elements of $A$ are additively correlated.
\end{quote}

In order to introduce the precise result, we recall the notion of a \emph{generalized arithmetic progression} (GAP).  A set $Q\subset \C^d$ is a \emph{GAP of rank $r$} if it can be expressed in the form
$$Q= \{g_0+ k_1g_1 + \dots +k_r g_r : K_i \le k_i \le K_i', k_i\in \Z \hbox{ for all } 1 \leq i \leq r\}$$ for some $g_0,\ldots,g_r\in \C^d$, and some integers $K_1,\ldots,K_r,K'_1,\ldots,K'_r$.


The vectors $g_i$ are the \emph{generators} of $Q$, the numbers $K_i'$ and $K_i$ are the \emph{dimensions} of $Q$, and $\Vol(Q) := |B|$ is the \emph{volume} of $Q$. We say that $Q$ is \emph{proper} if $|Q| = \Vol(Q)$. If $g_0=0$ and $-K_i=K_i'$ for all $i\ge 1$, we say that $Q$ is {\it symmetric}.

Consider a proper symmetric GAP $Q= \{\sum_{i=1}^r k_ig_i : -K_i \le k_i \le K_i\}$ of rank $r=O(1)$ and size $N=n^{O(1)}$ in $\C$.  Assume that $\xi$ has Bernoulli distribution and for each $a_i$ there exists $q_i\in Q$ such that $|a_i-q|\le \delta$. Then, because the random sum $\sum_i q_ix_i$ takes values in the GAP $nQ:=\{\sum_{i=1}^r k_ig_i : -nK_i \le k_i \le nK_i\}$, a GAP of size $|nQ| \le n^r N=n^{O(1)}$, the pigeon-hole principle implies that $\sum_i q_ix_i$ takes some value in $nQ$ with probability $n^{-O(1)}$. Thus we have
\begin{equation}\label{bound2} 
\rho_{n\delta, \xi}(A)   = n^{-O(1)}.
\end{equation}

This example shows that if $A$ is {\it close} to a $GAP$ of rank $O(1)$ and size $n^{O(1)}$, then $A$ has large small ball probability. It was shown  by Tao and Vu in \cite{TVbull,TVinverse,TVcirc,TVcomp} that this is essentially the only example which has large small ball probability. We recite here an explicit version from \cite{NgV} which will be used later on.



We say that a vector $a$ is {\it $\delta$-close} to a set $Q$ if there exists $q\in Q$ such that $\|a-q\|\le \delta$.

\begin{theorem}[Inverse Littlewood-Offord theorem for linear forms, \cite{NgV}]\label{theorem:linear} Let $0 <\ep < 1$ and $B>0$. Let 
$ \beta >0$ be a parameter that may depend on $n$. Suppose that $\sum_i \|a_i\|^2 =1$ and 
$$\rho:=\rho_{\beta,\xi}(A) \ge n^{-B},$$ 
where $x_1,\ldots,x_n$ are iid copies of a random variable $\xi$ having bounded $(2+\eta)$-moment. Then, for any number $n'$ between $n^\ep$ and $n$, there exists a proper symmetric GAP $Q=\{\sum_{i=1}^r k_ig_i : |k_i|\le K_i \}$ such that

\begin{itemize}

\item At least $n-n'$ elements of $a_i$ are $\beta$-close to $Q$.

\item $Q$ has small rank, $r=O_{B,\ep}(1)$, and small size

$$|Q| \le \max \{ O_{B,\ep}(\frac{\rho^{-1}}{\sqrt{n'}}),1\}.$$

\item There is a non-zero integer $p=O_{B,\ep}(\sqrt{n'})$ such that all
 generators $g_i$ of $Q$ have the form  $g_i=(g_{i1},\dots,g_{id})$, where $g_{ij}=\beta\frac{p_{ij}} {p} $ with $p_{ij} \in \Z$ and $|p_{ij}|=O_{B,\ep}(\beta^{-1} \sqrt{n'}).$

\end{itemize}
\end{theorem}

\subsection{Random multilinear forms}
One can view the sum  $S_\xi(A) =a_1 x_1+\dots+a_n x_n$ as a linear function of the random variables $x_1,\dots, x_n$. It is natural to study 
general polynomials of higher degree.

Let $D$ be a fixed positive integer. Let $x_{1i_1},x_{2i_2},\dots,x_{Di_D}$, $1\le i_1,\dots,i_D \le n$, be iid copies of a random variable $\xi$, and let $A=(a_{i_1i_2 \dots i_D})_{1\le i_1,\dots,i_D\le n}$ be an $n^D$-array of complex numbers. We define the {\it $D$-multilinear concentration probability} of $A$ by
$$\rho_{\beta,\xi}(A):= \sup_{a \in \C, L_{D-1}}\P \Big( \sum_{1\le i_1,\dots,i_D\le n} a_{i_1i_2\dots i_D}x_{1i_1}x_{2i_2} \dots x_{Di_D} + L_{D-1}(\Bx_1,\Bx_2,\dots,\Bx_D)\in B(a,\beta) \Big),$$
where $\Bx_i=(x_{i1},\dots,x_{in})$ and $L_{D-1}(\Bx_1,\dots, \Bx_D)$ is any $(D-1)$-multilinear form of $(\Bx_1,\dots,\Bx_D)$.

We would like to characterize $A$ with large $\rho_{\beta,\xi}(A)$. The following examples serve as good candidates. 


\begin{example}\label{example:dlinear} In what follows $\xi$ has Bernoulli distribution and for each $a_{i_1i_2\dots i_D}$ there exists $q_{i_1i_2\dots i_D}$ such that $|a_{i_1i_2\dots i_D}-q_{i_1i_2\dots i_D}|\le \delta.$
\begin{enumerate}
\item  Let $Q$ be a proper symmetric GAP of rank $r=O(1)$ and size $n^{O(1)}$. Assume that the approximated values $q_{i_1i_2\dots i_D}$ belong to $Q$. Then, the pigeon-hole principle implies that $\sum_{i_1,\dots,i_D}q_{i_1i_2\dots i_D}x_{1i_1}\dots x_{Di_D}$ takes some value in $n^2Q$ with probability $n^{-O(1)}$. Passing back to $a_{i_1i_2\dots i_D}$, we obtain $\rho_{n^2\delta,\xi}(A) =n^{-O(1)}$. 


\item Assume that $q_{i_1i_2\dots i_D}$ can be written as $q_{i_1i_2\dots i_D}=k_{i_1}b_{\bar{i}_1i_2\dots i_D}+l_{i_2}b_{i_1\bar{i}_2\dots i_D}+\dots + m_{i_D} b_{i_1i_2\dots \bar{i}_D}$, where $b_{\bar{i}_1i_2\dots i_D},\ldots,b_{i_1i_2\dots \bar{i}_D}$ are arbitrary sequences in $\R^d$ without indices $i_1,\dots,i_D$ respectively, and $k_{i_1},l_{i_2},\dots, m_{i_D}$ are integers bounded by $n^{O(1)}$ such that 
\begin{align*}
&\P_{\Bx_1}(\sum_{i_1} k_{i_1}x_{1i_1}= 0)=n^{-O(1)},\dots, \\ 
&\P_{\Bx_D}(\sum_{i_D} m_{i_D}x_{Di_D}= 0)=n^{-O(1)}.
\end{align*}

Then, as $\sum_{i_1,i_2,\dots,i_D}q_{i_1i_2\dots i_D}x_{1i_1}x_{2i_2}\dots x_{Di_D}$ factors out, we have 
$$\P(\sum_{i_1,i_2,\dots,i_D}q_{i_1i_2\dots i_D}x_{1i_1}x_{2i_2}\dots x_{di_D} =\mathbf{0}) =n^{-O(1)}.$$ 
\noindent Passing back to $a_{i_1i_2\dots i_D}$, we hence obtain $\rho_{n^2\delta,\xi}(A) =n^{-O(1)}.$ 


\item  Assume that $q_{i_1i_2\dots i_D}=q_{i_1i_2\dots i_D}' +q_{i_1i_2\dots i_D}''$, where $q_{i_1i_2\dots i_D}'\in Q$, a proper symmetric GAP of rank $O(1)$ and size $n^{O(1)}$, and $q_{i_1i_2\dots i_D}''$ is a sum of a few forms from (2) in such a way that the linear factors are zero with high probability. As such, we have 

$$\sup_{q\in n^2Q}\P_{\Bx_1,\dots,\Bx_D}(\sum_{i_1,\dots,i_D}q_{i_1i_2\dots i_D}x_{1i_1}x_{2i_2}\dots x_{di_D} =q)=n^{-O(1)}.$$ 

Hence we also have  $\rho_{n^2\delta,\xi}(A) =n^{-O(1)}$ in this case. 
\end{enumerate}
\end{example}

The above examples demonstrate that if the $a_{i_1i_2\dots i_D}$ can be decomposed into additive and algebraic structural parts, then $\rho_{\xi,\beta}(A)$ is large. We conjecture that these are essentially the only cases that have large  concentration probability. 

\begin{conjecture}\label{conjecture:dlinear} Assume that  $\rho_{\xi,\beta}(A) \ge n^{-B}$ for some generic $\xi$ and small $\beta$, then most of the elements of $A$ can be $\beta$-approximated by a set of $q_{i_1i_2\dots i_D}$ from (3) of Example \ref{example:dlinear}.
\end{conjecture}

Due to its nature, we believe that any justification of Conjecture \ref{conjecture:dlinear} would be highly technical. In this note we prove a weak version of it as follows.

\begin{theorem}[Weak inverse-type theorem for multilinear forms]\label{theorem:dlinear:weak}

Let $0 <\ep < 1$ and $C>0$. Let $ \beta >0$ be a parameter that may depend on $n$. Assume that  
$$\rho =\sup_{a,L_{D-1}} \P_{\Bx_1,\dots,\Bx_D}\Big(\sum_{i_1,i_2,\dots,i_D}a_{i_1i_2\dots i_D}x_{1i_1}x_{2i_2}\dots x_{Di_D} -L_{D-1}(\Bx_1,\dots,\Bx_D) \in B(a,\beta)\Big)\ge n^{-C},$$
where $\Bx_1=(x_{11},\dots,x_{1n}),\dots,\Bx_D=(x_{D1},\dots,x_{Dn})$, and $x_{1i_1},\dots,x_{Di_D}$ are iid copies of a random variable $\xi$ having bounded $(2+\eta)$-moment. Then there exist index sets $I_1,I_1^0$ with $|I_1|=n-n^\ep$ and $|I_1^0|=O_{C,\ep}(1)$ such that for any $i_1\in I_1$, there exist index sets $I_2,I_2^0$ depending on $i_1$ with $|I_2|=n-n^\ep$ and $|I_2^0|=O_{C,\ep}(1)$  such that \dots there exist index sets $I_{D-1},I_{D-1}^0$ depending on $i_1,\dots,i_{D-2}$ with  $|I_{D-1}|=n-n^\ep$ and $|I_{D-1}^0|=O_{C,\ep}(1)$  such that the following holds: for any $i_{D-1}\in I_{D-1}$, there exist integers $k_{j_1 \dots j_{D-1}}$, where each index $j_k$ with $1\le k\le D-1$ either takes value $i_k$ or belongs to the thin sets $I_k^0$, such that $k_{j_1 \dots j_{D-1}}=n^{O_{C,d,\ep}(1)}$ and $k_{i_1,\dots,i_{D-1}}\neq 0$, as well as 
$$\P_{\Bx_D}\Big(|\sum_{1\le i_D \le n}\sum_{j_1 \dots j_{D-1}}k_{j_1 \dots j_{D-1}}a_{j_1 \dots j_{D-1} i_D}x_{i_D}|\le \beta n^{O_{C,\ep}(1)}\Big)\ge n^{-O_{C,\ep}(1)}.$$
\end{theorem}

Notice that while in Example \ref{example:dlinear} most of the $D-1$ dimensional arrays of $A$ have structure, Theorem \ref{theorem:dlinear:weak} just asserts that most of the 1-dimension arrays $a_{j_1 \dots j_{D-1} i_D}, 1\le i_D\le n$, with fixed $j_1,\dots,j_{D-1}$, have structure.

For the rest of this section we give a proof of Theorem \ref{theorem:dlinear:weak}. Our argument heavily relies on the following simple fact about generalized arithmetic progressions of small rank.

\begin{fact}\label{fact:GAP}
Assume that $q_1,\dots,q_{r+1}$ are elements of a GAP of rank $r$ and of cardinality $n^C$, then there exists integer coefficients $\alpha_1,\dots,\alpha_r$ with $|\alpha_i|\le n^{rC}$ and such that
$$\sum_i \alpha_i q_i =0.$$
\end{fact}

\subsection{Proof of Theorem \ref{theorem:dlinear:weak}}

Without loss of generality, we assume that $\xi$ has discrete distribution. The continuous case can be easily extended by a standard limiting argument. We begin by applying Theorem \ref{theorem:linear}.

\begin{lemma}\label{lemma:roworthogonal} 
Let $\ep<1$ and $C$ be positive constants. Assume that $\rho_{\xi,\beta}(A)=\rho \ge n^{-C}$. Then the following holds with probability at least $\frac{3\rho}{4}$ with respect to $\Bx_2,\dots, \Bx_D$. There exist a proper symmetric GAP $Q_{\Bx_2\dots \Bx_D}$ of rank $O_{C,\ep}(1)$ and size $O_{C,\ep}(1/\rho)$ and a set $I_{\Bx_2,\dots, \Bx_D}$ of $n-n^\ep$ indices such that for each $i\in I_{\Bx_2,\dots,\Bx_D}$, there exists $q_i\in Q_{\Bx_2,\dots,\Bx_D}$ so that
$$|\sum_{i_2,\dots, i_D}a_{ii_2\dots i_D}x_{2i_2}\dots x_{Di_D}-q_i| \le \beta.$$
\end{lemma}

\begin{proof}(of Lemma \ref{lemma:roworthogonal})
For short we write
$$\sum_{i_1,i_2,\dots,i_D} a_{i_1i_2\dots i_D}x_{1i_1}x_{2i_2}\dots x_{Di_D} = \sum_{i=1}^n x_{1i} B_i(\Bx_2,\dots,\Bx_D),$$
where 
$$B_i(\Bx_2,\dots,\Bx_D):=\sum_{i_2,\dots,i_D}a_{ii_2\dots i_D}x_{2i_2}\dots x_{Di_D}.$$
We call a vector tuple $(\Bx_2,\dots,\Bx_D)$ {\it good} if
$$\P_{\Bx_1}(|\sum_{i=1}^n x_{1i}B_i(\Bx_2,\dots,\Bx_D)\in B(a,\beta)|)\ge \rho/4.$$
We call $\Bx_2,\dots,\Bx_D$ {\it bad} otherwise. Let $G$ be the collection of good tuples.

First, we estimate the probability $p$ of randomly chosen vectors $(\Bx_2,\dots,\Bx_D)$ being bad by an averaging method.
\begin{align*}
\P_{\Bx_2,\dots,\Bx_D} \P_{\Bx_1} (\sum_{i=1}^n x_{1i}B_i(\Bx_2,\dots,\Bx_D)\in B(a,\beta))&=\rho\\
p \rho/4 + 1-p &\ge \rho.\\
(1-\rho)/(1-\rho/4) &\ge p.
\end{align*}

Thus, the probability of randomly chosen vectors $(\Bx_2,\dots,\Bx_D)$ being good is at least 
$$1-p \ge (3\rho/4)/(1-\rho/4) \ge 3\rho/4.$$
Next, we consider good vectors $(\Bx_2,\dots,\Bx_D)\in G$. By definition, we have
$$\P_{\Bx_1} (\sum_{i=1}^n x_{1i} B_i(\Bx_2,\dots,\Bx_D) \in B(a,\beta)) \ge \rho/4 .$$
Observe that this is a high concentration of a linear form of $x_{1i}$. A direct application of Theorem \ref{theorem:linear} to the sequence $B_i(\Bx_2,\dots,\Bx_D)$, $i=1,\dots,n$ yields the desired result.
\end{proof}

By a useful property of GAP containment (see for instance \cite[Section 8]{TVsing} and \cite[Theorem 6.1]{Ng}), we may assume that the $q_i(\Bx_2,\dots,\Bx_D)$ span $Q_{\Bx_2,\dots,\Bx_D}$. From now on we fix such a $Q_{\Bx_2,\dots,\Bx_D}$ for each $\Bx_2,\dots,\Bx_D$. Let $G$ be the collection of good vectors $(\Bx_2,\dots,\Bx_D)$. Thus, 
\begin{equation}
\P_{\Bx_2,\dots,\Bx_D}((\Bx_2,\dots,\Bx_D)\in G)\ge 3\rho/4.
\end{equation}

Now we state our crucial lemma for the proof of Theorem \ref{theorem:dlinear:weak}.

\begin{lemma}\label{lemma:reduction} There exits an index set $I$ of size at least $n-2n^\ep$, an index set $I_0$ of size $O_{C,\ep}(1)$, and an integer $k\neq 0$ with $|k|\le n^{O_{C,\ep}(1)}$ such that for any index $i$ from $I$, there are numbers $k_{ii_0} \in \Z, i_0\in I_0$, all bounded by $n^{O_{C,\ep}(1)}$, such that 
$$\P_{\Bx_2,\dots,\Bx_D}\Big(k B_i(\Bx_2,\dots,\Bx_D)+ \sum_{i_0\in I_0} k_{ii_0} B_{i_0}(\Bx_2,\dots,\Bx_D)\in B(a,\beta)\Big)=\rho/n^{O_{C,\ep}(1)}.$$
\end{lemma}

\begin{proof}(of Lemma \ref{lemma:reduction}) 
For each $(\Bx_2,\dots,\Bx_D)\in G$, we choose from $I_{\Bx_2,\dots,\Bx_D}$ $s$ indices $i_{(1,\Bx_2,\dots,\Bx_D)},\dots,i_{{(s,\Bx_2,\dots,\Bx_D)}}$ such that $q_{i_{{(j,\Bx_2,\dots,\Bx_D)}}}(\Bx_2,\dots,\Bx_D)$ span $Q_{\Bx_2,\dots,\Bx_D}$, where $s$ is the rank of $Q_{\Bx_2,\dots,\Bx_D}$. We note that $s=O_{C,\ep}(1)$ for all $(\Bx_2,\dots,\Bx_D)\in G$. 

Consider the tuples $(i_{{(1,\Bx_2,\dots,\Bx_D)}},\dots,i_{{(s,\Bx_2,\dots,\Bx_D)}})$ for all $(\Bx_2,\dots,\Bx_D)\in G$. Because there are $\sum_{s} O_{C,\ep,\mu}(n^s) = n^{O_{C,\ep,\mu}(1)}$ possibilities these tuples can take, there exists a tuple, say $(1,\dots,r)$ (by rearranging the rows of $A$ if needed), such that $(i_{{(1,\Bx_2,\dots,\Bx_D)}},\dots,i_{{(s,\Bx_2,\dots,\Bx_D)}})=(1,\dots,r)$ for all $(\Bx_2,\dots,\Bx_D)\in G'$, where $G'$ is a subset of $G$ satisfying 
\begin{equation}
\P_{\Bx_2,\dots,\Bx_D}((\Bx_2,\dots,\Bx_D)\in G')\ge \P_{\Bx_2,\dots,\Bx_D}((\Bx_2,\dots,\Bx_D)\in G)/n^{O_{C,\ep}(1)} =\rho/n^{O_{C,\ep}(1)}.
\end{equation}

For each $1\le i\le r$, we express $q_i(\Bx_2,\dots,\Bx_D)$ in terms of the generators of $Q_{\Bx_2,\dots,\Bx_D}$ for each $(\Bx_2,\dots,\Bx_D)\in G'$, 
$$q_{i}(\Bx_2,\dots,\Bx_D) = c_{i1}(\Bx_2,\dots,\Bx_D)g_{1}(\Bx_2,\dots,\Bx_D)+\dots + c_{ir}(\Bx_2,\dots,\Bx_D)g_{r}(\Bx_2,\dots,\Bx_D),$$ 
where $c_{i1}(\Bx_2,\dots,\Bx_D),\dots c_{ir}(\Bx_2,\dots,\Bx_D)$ are integers bounded by $n^{O_{C,\ep}(1)}$, and $g_{i}(\Bx_2,\dots,\Bx_D)$ are the generators of $Q_{\Bx_2,\dots,\Bx_D}$.

We show that there are many $(\Bx_2,\dots,\Bx_D)$ that correspond to the same coefficients $c_{i_1i_2}$. 

\begin{claim}\label{claim:dlinear:common} There exists a (``dense'') subset  $G''\subset G'$ such that the following holds

\begin{itemize}
\item  $\P_{\Bx_2,\dots,\Bx_D}((\Bx_2,\dots,\Bx_D)\in G'')\ge \P_{\Bx_2,\dots,\Bx_D}((\Bx_2,\dots,\Bx_D)\in G')/n^{O_{C,\ep}(1)} \ge \rho/n^{O_{C,\ep}(1)};$
\item (common tuples) there exist $r$ tuples  $(c_{11},\dots,c_{1r}),\dots, (c_{r1},\dots c_{rr})$, whose components are integers bounded by $n^{O_{C,\ep,\mu}(1)}$, such that the following hold for all $(\Bx_2,\dots,\Bx_D)\in G''$:
\begin{enumerate}
\item $q_{i}(\Bx_2,\dots,\Bx_D) = c_{i1}g_{1}(\Bx_2,\dots,\Bx_D)+\dots + c_{ir}g_{r}(\Bx_2,\dots,\Bx_D)$, for $i=1,\dots,r$;
\item The vectors  $(c_{11},\dots,c_{1r}),\dots, (c_{r1},\dots c_{rr})$ span $\Z^{\rank(Q_{\Bx_2,\dots,\Bx_D})}$.
\end{enumerate} 
\end{itemize}
\end{claim}

\begin{proof}(of Claim \ref{claim:dlinear:common})
Consider the collection $\mathcal{C}$ of the coefficient-tuples 
$$\mathcal{C}:=\Big\{\Big(\big(c_{11}(\Bx_2,\dots,\Bx_D),\dots,c_{1r}(\Bx_2,\dots,\Bx_D)\big);\dots; \big(c_{r1}(\dots),\dots c_{rr}(\dots)\big)\Big), (\Bx_2,\dots,\Bx_D)\in G'\Big\} .$$ 
Because the number of possibilities these tuples can take is at most $(n^{O_{C,\ep}(1)})^{r^2} =n^{O_{C,\ep}(1)}$, by the pigeon-hole principle there exists a coefficient-tuple, say  $\Big((c_{11},\dots,c_{1r}),\dots, (c_{r1},\dots c_{rr})\Big)\in \mathcal{C}$, such that  
\begin{align*}
&\quad \Big(\big(c_{11}(\Bx_2,\dots,\Bx_D),\dots,c_{1r}(\Bx_2,\dots,\Bx_D)\big);\dots; \big(c_{r1}(\Bx_2,\dots,\Bx_D),\dots c_{rr}(\Bx_2,\dots,\Bx_D)\big)\Big)\\ 
&=\Big((c_{11},\dots,c_{1r}),\dots, (c_{r1},\dots c_{rr})\Big)
\end{align*}
for all $(\Bx_2,\dots,\Bx_D)$ from a subset $G''$ of $G'$ which satisfies 
\begin{equation}
  \P_{\Bx_2,\dots,\Bx_D}((\Bx_2,\dots,\Bx_D)\in G'')\ge \P_{\Bx_2,\dots,\Bx_D}((\Bx_2,\dots,\Bx_D)\in G')/n^{O_{C,\ep}(1)} \ge \rho/n^{O_{C,\ep}(1)}.
\end{equation}
\end{proof}

Now we focus on the elements of $G''$. Because $|I_{\Bx_2,\dots,\Bx_D}|\ge n-n^\ep$ for each $(\Bx_2,\dots,\Bx_D)\in G''$, we obtain the following.

\begin{claim}\label{claim:dlinear:I}
There is a set $I$ of size $n-3n^\ep$ such that $I \cap \{1,\dots,r\} =\emptyset$ and for each $i\in I$ we have 
\begin{equation}\label{eqn:optional}
\P_{\Bx_2,\dots,\Bx_D}(i\in I_{\Bx_2,\dots,\Bx_D}, (\Bx_2,\dots,\Bx_D)\in G'') \ge \P_{\Bx_2,\dots,\Bx_D}((\Bx_2,\dots,\Bx_D)\in G'')/2.
\end{equation}  
\end{claim}

\begin{proof}(of Claim \ref{claim:dlinear:I})
The result easily follows by an elementary averaging argument, 
\end{proof}


{\bf Lemma \ref{lemma:reduction}: proof conclusion.} Now we fix an arbitrary index $i$ from $I$. We concentrate on those $(\Bx_2,\dots,\Bx_D)\in G''$ where the index $i$ belongs to $I_{\Bx_2,\dots,\Bx_D}$. Because $q_{i}(\Bx_2,\dots,\Bx_D) \in Q_{\Bx_2,\dots,\Bx_D}$, we can write 
$$q_{i}(\Bx_2,\dots,\Bx_D)= c_{1}(\Bx_2,\dots,\Bx_D)g_{1}(\Bx_2,\dots,\Bx_D)+\dots c_{r}(\Bx_2,\dots,\Bx_D)g_{r}(\Bx_2,\dots,\Bx_D),$$ 
where $c_{1}(\Bx_2,\dots,\Bx_D),\dots,c_r(\Bx_2,\dots,\Bx_D)$ are integers bounded by $n^{O_{C,\ep}(1)}$.

For short, we denote by $\mathbf{v}_{i,\Bx_2,\dots,\Bx_D}$ the vector $(c_{1}(\Bx_2,\dots,\Bx_D),\dots c_{r}(\Bx_2,\dots,\Bx_D))$, we also use the shorthand $\Bv_j$ for the vectors $(c_{j1},\dots,c_{jr})$ obtained from Claim \ref{claim:dlinear:common}. 

Because $Q_{\Bx_2,\dots,\Bx_D}$ is spanned by $q_{1}(\Bx_2,\dots,\Bx_D),\dots, q_{r}(\Bx_2,\dots,\Bx_D)$, we must have $k:=\det(\mathbf{v}_1,\dots \mathbf{v}_r)\neq 0$ and that
\begin{align*}
&k q_i(\Bx_2,\dots,\Bx_D) + \det(\mathbf{v}_{i,y},\mathbf{v}_2,\dots,\mathbf{v}_r)q_{1}(\Bx_2,\dots,\Bx_D)+\dots \\
&+ \det(\mathbf{v}_{i,y},\mathbf{v}_1,\dots,\mathbf{v}_{r-1})q_{r}(\Bx_2,\dots,\Bx_D)=0. 
\end{align*}

Furthermore, because each coefficient of the identity above is bounded by $n^{O_{C,\ep,\mu}(1)}$, there exists a subset $G_{i}''$ of $G''$ such that all $(\Bx_2,\dots,\Bx_D)\in G_{i}''$ correspond to the same identity, and
\begin{align*}
\P_{(\Bx_2,\dots,\Bx_D)}((\Bx_2,\dots,\Bx_D)\in G_{i}'') &\ge (\P_{\Bx_2,\dots,\Bx_D}((\Bx_2,\dots,\Bx_D)\in G'')/2)/(n^{O_{C,\ep}(1)})^r\\
&\ge \rho/n^{O_{C,\ep}(1)}.
\end{align*}
In other words, there exist integers $k_1,\dots,k_r$, all bounded by $n^{O_{C,\ep}(1)}$, such that 
$$k q_i(\Bx_2,\dots, \Bx_D) + k_1 q_{1}(\Bx_2,\dots,\Bx_D)+ \dots + k_r q_{r}(\Bx_2,\dots,\Bx_D)=0$$
for all $(\Bx_2,\dots, \Bx_D)\in G_{i}''$. 

Note that $k$ is independent of the choice of $i$ and $(\Bx_2,\dots,\Bx_D)$. By passing from $q_i$ to $B_i$ by approximation, we thus complete the proof of Lemma \ref{lemma:reduction}.
\end{proof}

We are now ready to complete the proof of our inverse result.

{\bf Theorem \ref{theorem:dlinear:weak}: proof conclusion}. From Lemma \ref{lemma:reduction}, for any fixed $i\in I$, we consider the following $(D-1)$-multilinear form
\begin{align*}
B_i'(\Bx_2,\dots,\Bx_D):&= k B_i(\Bx_2,\dots,\Bx_D)+ \sum_{i_0\in I_0} k_{ii_0} B_{i_0}(\Bx_2,\dots,\Bx_D)\\
&=(k a_{ii_2\dots i_D} + \sum_{i_0} k_{ii_0} a_{i_0i_2 \dots i_D}) x_{2i_2}\dots x_{Di_D}\\
&:=\sum_{i_2,\dots, i_D} b_{i_2\dots i_D}' x_{2i_2}\dots x_{Di_D}.
\end{align*}
By the conclusion of Lemma \ref{lemma:reduction}, we have $\sup_a\P_{\Bx_2,\dots,\Bx_D}(B_i'(\Bx_2,\dots,\Bx_D)\in B(a,\beta))\ge \rho/n^{O_{C,\ep}(1)}$. Thus Lemma \ref{lemma:reduction} is applicable again for this new $(D-1)$-multilinear form. By iterating the process $D-1$ times, we obtain the conclusion of Theorem \ref{theorem:dlinear:weak}.

\section{Singularity of block matrices: the approach to prove Theorem \ref{thm:least-sing-value}} \label{section:singularity:approach}

As the singular values do not change under row and column permutations,  for the sake of convenience, we will restrict our analysis to matrices of the form $\BM_n= \BX_n+\BN_n$, where $\BN_n$ is any deterministic matrix of polynomially bounded norm and $\BX_n$ is a $dn \times dn$ matrix whose $ij$-th block takes the form 
$$ \begin{pmatrix} x_{11;ij} & x_{12;ij} & \dots & x_{1d;ij} \\  \dots & \dots & \dots & \dots  \\ x_{d1;ij} & x_{d2;ij} & \dots & x_{dd;ij} \end{pmatrix}, $$ 
where $(x_{11;ij},\dots,x_{dd;ij}), 1\le i,j\le n$, are iid copies of $(\xi_{11},\dots,\xi_{dd})$ which satisfy the following conditions from Definition \ref{cond:covariance} and Theorem \ref{thm:least-sing-value}: 
\begin{align}\label{eqn:covariance}
&\E \xi_{st}=0, \E |\xi_{st}|^2=1, \E|\xi_{st}|^{2+\eta}<\infty \mbox{ for some } \eta>0 \nonumber \\
&\E\left[ \xi_{st} \overline{\xi_{uv}}\right] = 0 \mbox{ for all } (s,t) \neq (u,v).
\end{align}

We now restate Theorem \ref{thm:least-sing-value} as follows.

\begin{theorem}\label{theorem:least-sing}
For any $B>0$, there exists $A>0$ depending on $B$ and $\alpha$ such that 
$$\P(\sigma_{dn}(\BM_n)\le n^{-A})\le n^{-B}.$$
\end{theorem}

In the sequel we sketch the proof of Theorem \ref{theorem:least-sing}. In general, our approach will resemble that of \cite{NgO,Ng,RV,TVcirc,Ver} where the main ingredient is an inverse-type argument. However, as our matrix now consists of large blocks of correlated entries, we need to elaborate more on the algebraic and technical side. For the sake of simplicity, we will prove our result under the following condition. 


\begin{assumption}\label{condition:bound}
With probability one, $|x_{11;ij}| \le n^{B+1} \wedge \dots \wedge |x_{dd;ij}| \le n^{B+1}$ for all $i,j$.
\end{assumption}


In what follows we assume that $\BM_{n}$ has full rank. This is the main case to consider as most random matrices are non-singular with very high probability. The case that $\BM_n$ is singular can be deduced by a standard argument (see for instance \cite[Appendix A]{NgV-det}).

Assume that $\sigma_{nd}(\BM_n)\le n^{-A}$. Thus $\BM_n\Bx=\By$ for some $\|\Bx\|=1$ and $\|\By\| \le n^{-A}$. Let $\BC=(c_{i,j}(\BM_n))$, $1\le i,j\le dn$, be the matrix of cofactors of $\BM_n$. By definition, $\BC\By = \det(\BM_n) \cdot \Bx$, and thus we have $\|\BC\By\| = |\det(\BM_n)|$. 

By paying a factor of $dn$ in probability, without loss of generality we can assume that the first component of $\BC\By$ is greater than $\det(\BM_n)/(dn)^{1/2}$, 
\begin{equation}\label{eqn:firstrow}
|c_{1,1}(\BM_n)y_1+\dots c_{1,dn}(\BM_n)y_{dn}|\ge |\det(\BM_n)|/(dn)^{1/2}.
\end{equation}

\begin{claim}\label{claim:singularity:cofactordeduction} Let $\BM_{n-1}$ be the matrix obtained from $\BM_n$ by removing its first $d$ rows, and $ c_{i_1i_2\dots i_d}(\BM_{n-1}), 1\le i_1,\dots,i_d\le nd$ be the sign determinant of the minor obtained from $\BM_{n-1}$ by removing its $i_1,\dots,i_d$-th columns. We have
\begin{eqnarray}\label{eqn:intro:0}
\sum_{1\le i_1,i_2,\dots,i_d \le dn} |c_{i_1i_2\dots i_d}(\BM_{n-1})|^2 \ge n^{2A-O(B+\alpha)}|\det(\BM_n)|^2
\end{eqnarray}
\end{claim}

\begin{proof}(of Claim \ref{claim:singularity:cofactordeduction})
As $\|\By\| \le n^{-A}$, it follows from \eqref{eqn:firstrow} that
\begin{equation}\label{eqn:intro:1}
\sum_{i_1=1}^{dn} |c_{1,i_1}(\BM_n)|^2 \ge n^{2A-2} |\det(\BM_n)|^2.
\end{equation}
Next, as each cofactor $c_{1i_1}(\BM_n)$, as a sign determinant of a $(dn-1)\times (dn-1)$ block of $\BM_n$, can be expressed as
$$c_{1,i_1}(\BM_n)=(-1)^{(1+i_1)\dots+(d+i_d)}\sum_{1\le i_2,\dots,i_d \le dn} m_{2i_2}\dots m_{di_d} c_{i_1i_2\dots i_d}(\BM_{n-1}).$$

The claim then follows by applying Cauchy-Schwarz inequality together with Condition \ref{condition:bound} and the upper bound $n^\alpha$ on the entries of $\BN_n$.
\end{proof}



By Claim \ref{claim:singularity:cofactordeduction}, in order to prove Theorem \ref{theorem:least-sing} it suffices to justify the following result.


\begin{theorem}\label{theorem:singularvalue''}
For any $B>0$, there exists $A>0$ such that 
$$\P \big((\sum_{1\le i_1,i_2,\dots,i_d \le dn} |c_{i_1i_2\dots i_d}(\BM_{n-1})|^2)^{1/2} \ge n^A|\det(\BM_n)| \big)\le n^{-B}.$$
\end{theorem}


Next, express $\det(\BM_n)$ as a $d$-multilinear form of its first $d$ rows 
$$\det(\BM_n) = \sum_{1\le i_1,i_2,\dots,i_d \le dn} c_{i_1i_2\dots i_d}(\BM_{n-1}) m_{1i_1}\dots m_{di_d}.$$
With $c:=(\sum_{1\le i_1,i_2,\dots,i_d \le dn} |c_{i_1i_2\dots i_d}(\BM_{n-1})|^2)^{1/2}$ and $a_{i_1\dots i_d}:=c_{i_1\dots i_d}(\BM_{n-1})/c$,
\begin{equation}\label{eqn:bilinear}
\frac{1}{c}\det(\BM_n) =\sum_{1\le i_1,i_2,\dots,i_d \le dn} a_{i_1i_2\dots i_d}(\BM_{n-1}) m_{1i_1}\dots m_{di_d}.
\end{equation}
 
Heuristically, conditioning on $\BM_{n-1}$, the $d$-multilinear form in the RHS of \eqref{eqn:bilinear} is comparable to 1 in absolute value with probability extremely close to one. Thus the assumption $\P(|\det(M_n)|/c\le n^{-A})\ge n^{-B}$ of Theorem \ref{theorem:singularvalue''}, with  an appropriately large value $A$, must yield a high cancelation of the multilinear form. Based on this observation, our rough approach will consist of two main steps.
  
\begin{itemize}
\item  {\it Step 1}. Assume that for an appropriately large value $A>0$ we have 
$$\P_{x_{11;11},\dots, x_{dd;1n}}\big( |\sum_{1\le i_1,i_2,\dots,i_d \le dn} a_{i_1i_2\dots i_d}(\BM_{n-1}) m_{1i_1}\dots m_{di_d}|\le n^{-A}|\BM_{n-1}\big)\ge n^{-B}.$$ 
Then the normalized cofactors $a_{i_1\dots i_d}$ of $\BM_{n-1}$ must satisfy a very special property.
\item {\it Step 2}. The probability, with respect to $\BM_{n-1}$, that the $a_{i_1\dots i_d}$ satisfy this special property is negligible.
\end{itemize}

Although the setting of Step 1 is identical to our inverse problem discussed in Section \ref{section:anti-concentration}, the dependencies of the entries make the problem substantially harder. We will remove these dependencies using a series of decoupling tricks to arrive at a conclusion as useful as Theorem \ref{theorem:dlinear:weak}.



\begin{theorem}[Step 1]\label{theorem:step1}
Let $0<\ep<1$ be given constant. Assume that
$$\sup_a\P_{x_{11;11},\dots, x_{dd;1n}}\big(|\sum_{1\le i_1,i_2,\dots,i_d \le dn} a_{i_1i_2\dots i_d}(\BM_{n-1}) m_{1i_1}\dots m_{di_d}-a |\le n^{-A}\big)\ge n^{-B}.$$ 
for some sufficiently large integer $A$, where $a_{i,j}=c_{i,j}(\BM_{n-1})/c$. Then there exists $k=O(d)$ indices $i_1<\dots<i_k$ and a complex vector $\Bu=(u_1,\dots,u_{nd})$ which satisfies the following properties.

\begin{itemize}
\item \textup{(}orthogonality \textup{)} $\|\Bu\|_2\asymp 1$ and   $|\langle \Bu_1,\row_i^{(1)}(\BM_{n-1})\rangle + \langle \Bu_2,\row_i^{(2)}(\BM_{n-1})\rangle| \le n^{-A/2+O_{B,\ep}(1)}$ for $n-O_{B,\ep}(1)$ rows $\row_i$ of $\BM_{n-1}$, where $\Bu_1$ and $\row_i^{(1)}$ are the subvectors corresponding to the components indexed by $i_1,\dots,i_k$ of $\Bu$ and $\row_i$ respectively, and $\Bu_2$ and $\row_i^{(2)}$ are the subvectors corresponding to the remaining components  of $\Bu$ and $\row_i$ respectively;
\item \textup{(}additive structure\textup{)} there exists a generalized arithmetic progression $Q$ of rank $O_{B,\ep}(1)$ and size $n^{O_{B,\ep}(1)}$ that contains at least $dn-2n^\ep$ components $u_i$;
\item \textup{(}controlled form\textup{)} all the components $u_i$, and all the generators of the generalized arithmetic progression are rational complex numbers of the form $\frac{p}{q}+ \sqrt{-1} \frac{p'}{q'} $, where $|p|,|q|,|p'|,|q'| \le n^{A/2+O_{B,\ep}(1)}$.
\end{itemize}

\end{theorem}


In the second step, we  show that the probability that $\BM_{n-1}$ has the above properties is negligible.


\begin{theorem}[Step 2]\label{theorem:step2}
With respect to $\BM_{n-1}$, the probability that there exists a vector $\Bu$ as in Theorem \ref{theorem:step1} is $\exp(-\Omega(n))$.
\end{theorem}





\section{Singularity of block matrices: proof Theorem \ref{theorem:step1}} \label{section:step1}

Recall that in the inverse step, Theorem \ref{theorem:step1}, we assumed a high concentration of a multilinear form on a small ball of radius $n^{-A}$. As the entries in each block are dependent, we are not able to apply Theorem \ref{theorem:dlinear:weak} yet. In what follows we present two main steps to remove these dependencies. 


\subsection{Dependency removal I: general linear forms}\label{subsection:linear:new}

First, it will be useful  to study the concentration of the linear form 
$$\sum_{1\le i\le n} (a_{11;i}x_{11;i}+\dots + a_{dd;i}x_{dd;i}),$$ 
where $(x_{11;i},\dots,x_{dd;i})$ are iid copies of $(x_{11},\dots,x_{dd})$ satisfying \eqref{eqn:covariance}. Intuitively, as the covariance of $(x_{11},\dots,x_{dd})$ is non-singular, the random variables $(x_{11;i},\dots,x_{dd;i})$ are not totally dependent on each other. (See Appendix \ref{section:ILOlinear:new:proof} for a more precise statement.) This fact may suggest a way to apply Theorem \ref{theorem:linear} with respect to $(x_{11;1},\dots,x_{11;n})$ while holding $x_{12;1},\dots,x_{dd;n}$ fixed and vice versa. In what follows $(x_{1;i},\dots,x_{D;i})$ plays the role of $(x_{11;i},\dots,x_{dd;i})$. 

\begin{theorem}[Inverse Littlewood-Offord theorem for mixing linear forms]\label{theorem:ILOlinear:new} Let $0<\ep <1, B>0$ be given, and $D$ be a positive integer. Let 
$ \beta >0$ be an arbitrary real number that may depend on $n$. Suppose that $a_{1;i},\dots, a_{D;i}\in \C$ such that $\sum_{i=1}^n \sum_{1\le j\le D} |a_{j;i}|^2=1$ and 
$$\sup_a\P_{x_{1;1},\dots,x_{D;n}} \left(\left|\sum_{i=1}^n (a_{1;i}x_{1;i}+\dots + a_{D;i}x_{D;i}) -a\right| \le \beta\right)=\gamma \ge n^{-B},$$ 
where $(x_{1;i},\dots,x_{D;i}),1\le i\le n$ are iid copies of $(x_{1},\dots,x_{D})$ from \eqref{eqn:covariance}. Then there exist positive constants $\alpha,c_0,C_0$ depending only on the distribution of $(x_{1},\dots,x_{D})$ and $D$ tuples $(\eta_{k1},\eta_{k2},\dots,\eta_{kD}),1\le k\le D$ of complex numbers such that 

\begin{itemize}
\item $|\eta_{ij}|$ are bounded from below and above by $c_0$ and $C_0$ respectively,
\item  The least singular value of the matrix $(\eta_{ij})$ is at least $\alpha$,
\item  for any number $n'$ between $n^\ep$ and $n$, there exists a proper symmetric GAP $Q=\{\sum_{i=1}^r k_ig_i : k_i\in \Z, |k_i|\le L_i \} \subset \C$ whose parameters satisfy (i) and (ii) of Theorem \ref{theorem:linear} and for at least $n-n'$ indices $i$, we have $\eta_{k1}a_{1;i}+\dots+\eta_{kD}c_ia_{D;i}, 1\le k\le D$ are $\beta$-close to $Q$.
\end{itemize}
\end{theorem}

As Theorem \ref{theorem:ILOlinear:new} can be shown by using the method of \cite{NgV}, we skip its proof and refer the reader to Appendix \ref{section:ILOlinear:new:proof} for a proof of a somewhat more general result (Theorem \ref{theorem:linear:operator} below). We now introduce a useful corollary. 

\begin{corollary}\label{cor:ILOlinear:new}
Assume as in Theorem \ref{theorem:ILOlinear:new}. Then there exists a proper symmetric GAP $Q=\{\sum_{i=1}^r k_ig_i : k_i\in \Z, |k_i|\le L_i \} \subset \C$ whose parameters satisfy (i) and (ii) of Theorem \ref{theorem:ILOlinear:new} and an index set $I$ of size at least $n-n'$ such that with $(\gamma_{ij})$ being the inverse matrix of $(\eta_{ij})$: the sequence $a_{k;i},i\in I,1\le k\le D$ are $O(\beta)$-close to the GAP $P=P_1+\dots+P_D$, where $P_k=\gamma_{k1}\cdot Q + \gamma_{k2}\cdot Q+\dots+\gamma_{kD} \cdot Q$.
\end{corollary}

\subsection{Dependency removal II: decoupling} We now work with the multilinear form appearing in Theorem \ref{theorem:step1}. Our goal is to show the following.

\begin{theorem}\label{theorem:decoupling:weak} 
Let $0 <\ep < 1$ and $B>0$. Let $ \beta >0$ be a parameter that may depend on $n$. Assume that  
$$\P_{x_{11;11},\dots, x_{dd;1n}}\Big(|\sum_{1\le i_1,i_2,\dots,i_d \le dn} a_{i_1i_2\dots i_d}(\BM_{n-1}) m_{1i_1}\dots m_{di_d}|\le n^{-A}\vert\BM_{n-1}\Big)\ge n^{-B}.$$ 
Then there exist index sets $I_1,I_1^0$ with $|I_1|=dn-n^\ep$ and $|I_1^0|=O_{C,\ep}(1)$ such that for any $i_1\in I_1$, there exist index sets $I_2,I_2^0$ depending on $i_1$ with $|I_2|=dn-n^\ep$ and $|I_2^0|=O_{C,\ep}(1)$ such that \dots there exist index sets $I_{d-1},I_{d-1}^0$ depending on $i_1,\dots,i_{d-2}$ with  $|I_{d-1}|=n-n^\ep$ and $|I_{d-1}^0|=O_{C,\ep}(1)$  such that the following holds: for any $i_{d-1}\in I_{d-1}$, there exist integers $k_{j_1 \dots j_{d-1}}$, where each index $j_k$ with $1\le k\le d-1$ either takes value $i_k$ or belongs to the thin sets $I_k^0$, such that $k_{j_1 \dots j_{d-1}}=n^{O_{C,d,\ep}(1)}$ and $k_{i_1,\dots,i_{d-1}}\neq 0$, as well as 
$$\P_{\Bx_d}\Big(|\sum_{1\le i_d \le n}\sum_{j_1 \dots j_{d-1}}k_{j_1 \dots j_{d-1}}a_{j_1 \dots j_{d-1} i_d}x_{i_d}|\le \beta n^{O_{C,\ep}(1)}\Big)\ge n^{-O_{C,\ep}(1)}.$$

\end{theorem}

Thus Theorem \ref{theorem:decoupling:weak} asserts that as long as the entries in each block are not totally dependent, the conclusion of Theorem \ref{theorem:dlinear:weak} still holds as if the matrix entries were mutually independent.

In what follows we introduce the main supporting lemmas to prove Theorem \ref{theorem:decoupling:weak}.  By definition, we can rewrite this form as 
$$\sum_{i_1,\dots,i_d} a_{i_1 i_2 \dots i_d} \det[\Bc_{i_1},\dots,\Bc_{i_d}],$$ 
where $\det[\Bc_{i_1},\dots,\Bc_{i_d}]$ is the determinant of the $d\times d$ block generated by the $i_1$-th, $\dots i_d$-th columns of the matrix of the first $d$ rows of $M_n$.

Let $\mathcal{U}:=\{U_1,\dots,U_d\}$ be an ordered random partition of $[n]$. These index sets will serve as the collection of blocks (among the $n$ blocks of size $d\times d$ of the matrix generated by the first $d$ rows) to be partitioned. We denote by $B(U_i)$ the collection of indices generated by $U_i$, that is 

\begin{equation}\label{eqn:BU}
B(U_i):=\cup_{l\in U_i} \{(d-1)l+1,\dots,dl\}.
\end{equation}



Given any partition $\mathcal{U}$, we easily obtain the following lemma by a series applications of the Cauchy-Schwarz inequality.

\begin{lemma}[Decoupling lemma]\label{lemma:decoupling} Assume that 
$$\rho=\sup_{a} \P_{x_{11;11},\dots,x_{dd;1n}}  \Big(|\sum_{1\le i_1,\dots,i_d \le dn} a_{i_1 i_2 \dots i_d} \det[\Bc_{i_1},\dots,\Bc_{i_d}] -a|\le \beta\Big)\ge n^{-B},$$
Then,
\begin{equation}\label{eqn:ILObilinear:0}
\P_{x_{11;11}',\dots,x_{dd;1n}'}\Big(|\sum_{ i_1\in B(U_1),\dots,i_d\in B(U_d)} a_{i_1 i_2 \dots i_d}\det[\Bc_{i_1},\dots,\Bc_{i_d}]-a| =O_B(\beta \sqrt{\log n})\Big) =\Omega(\rho^{2d}),
\end{equation}
where $({x_{11;11}}',\dots,{x_{dd;11}}');\dots; ({x_{11;1n}}',\dots,{x_{dd;1n}}')$ are iid copies of the vector $(x_{11}-x_{11}',\dots, x_{dd}-x_{dd}')$, and where $(x_{11}',\dots,x_{dd}')$ is an independent copy of $(x_{11},\dots,x_{dd})$.
\end{lemma}

As the proof of Lemma \ref{lemma:decoupling} is standard, we refer the reader to \cite{CTV,Ng,Ver}. As the columns $\Bc_{i_1},\dots,\Bc_{i_d}$ are independent, we will be able to obtain an analogue of Lemma \ref{lemma:reduction} as follows. 

\begin{lemma}\label{lemma:reduction:A_U}
There exist index sets $I_0(U_1)$ with $|I_0(U_1)|=O(1)$ and $I(U_1)\subset B(U_1)$ with $|I(U_1)|\ge d|U_1|-n^\ep$ and an integer $k\neq 0, k=n^{O_{B,\ep}(1)}$ such that for any $i\in I(U_1)$, there exist integers $k_{ii_0}=n^{O_{B,\ep}(1)}$ such that
\begin{align*} 
&\P_{x_{rs;1j}, j\in B(U_2)\cup \dots \cup B(U_d)} \Big(\big|k \sum_{i_2\in B(U_2),\dots,i_d \in B(U_d)} a_{i,i_2 \dots i_d} \det [\Bc_{i_2}^{\bar{1}},\dots\Bc_{i_d}^{\bar{1}}]\\ 
&-\sum _{i_0\in I_0} k_{ii_0}\sum_{i_2\in B(U_2),\dots,i_d\in B(U_d)} a_{i_0,i_2 \dots i_d}\det [\Bc_{i_2}^{\bar{1}},\dots\Bc_{i_d}^{\bar{1}}]\big|=O(\beta n^{O_{B,\ep}(1)}) \Big) \ge \rho^{2d}/n^{O_{B,\ep}(1)},
\end{align*}
where $\Bc_{j}^{\bar{1}}$ is the $j$-th column $\Bc_j$ without its first component.
\end{lemma}

\begin{proof}(of Lemma \ref{lemma:reduction:A_U}) 
As usual, it suffices to assume $\xi$ to have discrete distribution. For each $l\in U_1$, let $B_l=\{(l-1)d+1,\dots,ld\}$ be the $l$-th block. By the determinant expansion, we have
\begin{align*}
&\sum_{i_1\in B_l,i_2\in B(U_2),\dots,i_d\in B(U_d)}a_{i_1i_2\dots i_d}\det[\Bc_{i_1},\dots,\Bc_{id}]\\
&= \sum_{i=1}^n x_{1i;1l} \sum_{i_2,\dots,i_d} (-1)^{1+i}a_{((l-1)d+i)i_2 \dots i_d}\det [\Bc_{i_2}^{\bar{1}},\dots,\Bc_{i_d}^{\bar{1}}]\\
&+\dots\\
&+ \sum_{i=1}^n x_{di;1l} \sum_{i_2,\dots,i_d} (-1)^{d+i}a_{((l-1)d+i)i_2 \dots i_d}\det [\Bc_{i_2}^{\bar{d}},\dots,\Bc_{i_d}^{\bar{d}}].
\end{align*}

By summing over $l\in U_1$ and by applying Lemma \ref{lemma:decoupling} and Corollary \ref{cor:ILOlinear:new} for the random variables $x_{rs,1j},j\in B(U_1)$, with high probability with respect the the random variables indexing from $\bar{U}_1$ (i.e. $x_{rs;1j}, j\in \bar{U}_1$), most of the coefficients 
$$\sum_{i_2,\dots,i_d} a_{((l-1)d+1)i_2 \dots i_d}\det [\Bc_{i_2}^{\bar{1}},\dots,\Bc_{i_d}^{\bar{1}}], \dots, \sum_{i_2,\dots,i_d} a_{((ld)i_2 \dots i_d}\det [\Bc_{i_2}^{\bar{1}},\dots,\Bc_{i_d}^{\bar{1}}]$$ belong to a GAP of rank $O_{C,\ep}(1)$ and size $n^{O_{C,\ep}(1)}$. From here, to conclude Lemma \ref{lemma:reduction:A_U}, we just follow the proof of Lemma \ref{lemma:reduction} verbatim. 
\end{proof}

\subsection{Randomization} Roughly speaking, by iterating Lemma \ref{lemma:reduction:A_U} to the new $(d-1)$-linear form of the random variables restricted by $\bar{U}_1$ and so on, we will be able to deduce an analogue of Theorem \ref{theorem:dlinear:weak} with the dependence upon $U_1,\dots,U_d$. One  might then try to randomize $U_1,\dots,U_d$ to obtain Theorem \ref{theorem:decoupling:weak}. However, the randomization of $U_1,\dots,U_d$ altogether may pose a highly technical difficulty. To avoid this hurdle we will try to randomize one pair at a time before each iteration of Lemma \ref{lemma:reduction:A_U}.  

Assume that $(U_{12},U_3,\dots,U_d)$ is a $(d-1)$ ordered partition of $[n]$ in which each partition has order $\Theta(n)$. Fixing this partition, we next partition $U_{12}$ into $U_1,U_2$ randomly. For the new $d$ partition $(U_1,\dots,U_d)$ we then apply Lemma \ref{lemma:reduction:A_U}. As a result, the index $i_2$ in the conclusion belongs to $B(U_2)$. We will show that by randomizing $U_1$, one may recover the result for $i_2$ now an element of $B(U_{12})$. Let us first extend Lemma \ref{lemma:reduction:A_U} as follows.   

\begin{lemma}\label{lemma:reduction:A_U12}
There exist subsets $I_0(U_{1})$ and $I(U_{1})$ of $B(U_{12})$ with size $O(1)$ and $d|U_{12}|-n^\ep$ respectively and an integer $k\neq 0, k=n^{O_{B,\ep}(1)}$ such that for any $i\in I(U_1)$, there exist integers $k_{ii_0}=n^{O_{B,\ep}(1)}$such that the following holds for all $i\in I(U_1)$  :
\begin{align*} 
&\P_{x_{11;11},\dots, x_{dd;1n}} \Big(\big|k \sum_{i_2\in B(U_{12}),\dots,i_d \in B(U_d)} a_{ii_2 \dots i_d}(U_1) \det [\Bc_{i_2}^{\bar{1}},\dots\Bc_{i_d}^{\bar{1}}]\\ 
&-\sum _{i_0\in I_0} k_{ii_0}\sum_{i_2\in B(U_{12}),\dots,i_d\in B(U_d)} a_{i_0i_2 \dots i_d}(U_1)\det [\Bc_{i_2}^{\bar{1}},\dots\Bc_{i_d}^{\bar{1}}]\big|=O(\beta n^{O_{B,\ep}(1)}) \Big) \ge \rho^{2d}/n^{O_{B,\ep}(1)},
\end{align*}
where  
$$ a_{ii_2 \dots i_d}(U_1):=
\begin{cases}
 a_{ii_2\dots i_d} & \mbox{ if }  i\in B(U_1), i_2\in B(U_2)\\
 a_{i_2i\dots i_d} & \mbox{ if } i\in B(U_2),i_2\in B(U_1)\\
 0 & \mbox{ otherwise} .
\end{cases} $$

\end{lemma}

In comparison with Lemma \ref{lemma:reduction:A_U},  the probability in Lemma \ref{lemma:reduction:A_U12} is now with respect to all random variables of the first $d$ rows of $M_n$. Also, $i_2$ now runs over all the indices restricted by $U_{12}$. The entries $a_{ii_2\dots i_d}(U_1)$, without the indices $i_3,\dots,i_d$, could be viewed as entries of a symmetric matrix. 

\begin{proof}(of Lemma \ref{lemma:reduction:A_U12}) 
We first fix the random variables restricted by $\bar{U}_{12}$ for which the conclusion of Lemma \ref{lemma:decoupling} holds with respect to the random variables restricted by $U_{12}$. Similarly to the proof of Lemma \ref{lemma:reduction:A_U}, the following holds with high probability with respect to $x_{rs;1i},i\in B(U_2)$: there exist subsets $I_0(U_1)$ and $I(U_1)$ of $B(U_1)$  with size $O(1)$ and $d|U_1|-n^\ep$ respectively such that the following holds for all $i\in I(U_1)$  :
\begin{align}\label{eqn:U_1} 
&\P_{x_{rs;1i_2}, i_2 \in B(U_2)} \Big(\big|k' \sum_{i_2\in B(U_2),\dots,i_d \in B(U_d)} a_{ii_2 \dots i_d} \det [\Bc_{i_2}^{\bar{1}},\dots\Bc_{i_d}^{\bar{1}}]\nonumber \\ 
&-\sum _{i_0\in I_0} {k'}_{ii_0}\sum_{i_2\in B(U_2),\dots,i_d\in B(U_d)} a_{i_0i_2 \dots i_d}\det [\Bc_{i_2}^{\bar{1}},\dots\Bc_{i_d}^{\bar{1}}]\big|=O(\beta n^{O_{B,\ep}(1)})\Big) \nonumber \\ 
&\ge \rho^{2d}/n^{O_{B,\ep}(1)}.
\end{align}
By switching the role of $U_1$ and $U_2$, there also exist subsets $I_0(U_2)$ and $I(U_2)$ of $B(U_2)$  with size $O(1)$ and $d|U_2|-n^\ep$ respectively such that the following holds for all $i\in I(U_2)$  :
\begin{align}\label{eqn:U_2} 
&\P_{x_{rs;1i_1}, i_1\in B(U_1)} \Big(\big|k'' \sum_{i_1\in B(U_1),i_3\in B(U_3),\dots,i_d \in B(U_d)} a_{i_1ii_3 \dots i_d} \det [\Bc_{i_1}^{\bar{1}},\Bc_{i_3}^{\bar{1}}\dots\Bc_{i_d}^{\bar{1}}]\nonumber \\ 
&-\sum _{i_0\in I_0} {k''}_{ii_0}\sum_{i_1\in B(U_1),i_3\in B(U_3),\dots,i_d\in B(U_d)} a_{i_1i_0 i_3 \dots i_d}\det [\Bc_{i_1}^{\bar{1}},\Bc_{i_3}^{\bar{1}},\dots\Bc_{i_d}^{\bar{1}}]\big|=O(\beta n^{O_{B,\ep}(1)})\Big) \nonumber \\
&\ge \rho^{2d}/n^{O_{B,\ep}(1)}.
\end{align}

Now, by the definition of $a_{ii_2\dots i_d}(U_1)$, with $I=I(U_1)\cup I(U_2)$ and $I_0=I_0(U_1)\cup I_0(U_2)$, we can rewrite both of the events in \eqref{eqn:U_1} and \eqref{eqn:U_2} in the following form

\begin{align*}
&k \sum_{i_2\in B(U_{12}),\dots,i_d \in B(U_d)} a_{ii_2 \dots i_d}(U_1) \det [\Bc_{i_2}^{\bar{1}},\dots\Bc_{i_d}^{\bar{1}}]\\
&-\sum _{i_0\in I_0} k_{ii_0}\sum_{i_2\in B(U_2),\dots,i_d\in B(U_d)} a_{i_0i_2 \dots i_d}(U_1)\det [\Bc_{i_2}^{\bar{1}},\dots\Bc_{i_d}^{\bar{1}}] =O(\beta n^{O_{B,\ep}(1)}) .
\end{align*}

The conclusion of Lemma \ref{lemma:reduction:A_U12} then follows from \eqref{eqn:U_1} and \eqref{eqn:U_2}, noting that $\{x_{rs;1i_1}, i_1 \in U_1\}$ and $\{x_{rs;1i_2}, i_2 \in U_2\}$ are independent. 
\end{proof}

Now we randomize $U_1$ to obtain the following main result of the subsection.

\begin{lemma}[Randomization]\label{lemma:random:A_U}
There exist subsets  $I_0(U_{12})$ and  $I(U_{12})$ of $B(U_{12})$ with size $O(1)$ and $d|U_{12}|-n^\ep$ respectively such that the following holds for all $i\in I$:
\begin{align*} 
&\P_{x_{rs;1i},i \in \bar{B}(U_{12});{x_{rs;1i}}',i\in B(U_{12})} \Big(\big|k \sum_{i_2\in B(U_{12}),\dots,i_d \in B(U_d)} a_{i,i_2 \dots i_d} \det [{\Bc_{i_2}^{\bar{1}}}',\Bc_{i_3}^{\bar{1}}, \dots,\Bc_{i_d}^{\bar{1}}]\\ 
&-\sum _{i_0\in I_0} k_{ii_0}\sum_{i_2\in B(U_2),\dots,i_d\in B(U_d)} a_{i_0,i_2 \dots i_d}\det [{\Bc_{i_2}^{\bar{1}}}',\Bc_{i_3}^{\bar{1}} ,\dots,\Bc_{i_d}^{\bar{1}}]\big|=O(\beta n^{O_{B,\ep}(1)}) \Big) \ge \rho^{2d}/n^{O_{B,\ep}(1)},
\end{align*}
where ${x^{(i)}_{rs}}':=\eta_{i}x^{(i)}_{rs}$ with $\eta_i$ iid Bernoulli random variables of parameter $1/2$, and $\Bc_{i_2}':=\eta_{i_2} \Bc_{i_2}$  in the determinants.
\end{lemma}

\begin{proof}(of Theorem \ref{lemma:random:A_U})
Note that Lemma \ref{lemma:reduction:A_U12} holds for any choice $U_1\subset U_{12}$. As $I_0(U_1)\subset [n]^{O_{B,\ep}(1)}$ and $k(U_1)\le n^{O_{B,\ep}(1)}$, there are only $n^{O_{B,\ep}(1)}$ possibilities that the tuple $(I_0(U_1),k(U_1))$ can take. Thus, there exists a tuple $(I_0,k)$ such that 
$I_0(U_1)=I_0$ and $k(U_1)=k$ for $2^{|U_{12}|}/n^{O_{B,\ep}(1)}$ different sets $U_1$. Let us denote this set of $U_1$ by $\mathcal{S}$; we have
$$|\mathcal{S}|\ge 2^{|U_{12}|}/n^{O_{B,\ep}(1)}.$$
Next, let $I$ be the collection of all $i\in B(U_{12})$ which belong to at least $|\mathcal{S}|/2$ index sets $I(U_1)$. Then,
\begin{align*}                         
|I||\mathcal{S}| + (d|U_{12}|-|I|)|\mathcal{S}|/2 & \ge (d|U_{12}|-n^\ep )|\mathcal{S}|\\
|I| &\ge  |U_{12}|-2n^\ep.
\end{align*}

From now on we fix an $i\in I$. Consider the tuples $(k_{ii_0}(U_1), i_0\in I_0)$ over all $U_1$ where $i\in I(U_1)$. Because there are only $n^{O_{B,\ep}(1)}$ possibilities such tuples can take, there must be a tuple, say $(k_{ii_0}, i_0\in I_0)$, such that $(k_{ii_0}(U_1), i_0\in I_0)=(k_{ii_0}, i_0\in I_0)$ for at least $|\mathcal{S}|/2n^{O_{B,\ep}(1)}=2^n/n^{O_{B,\ep}(1)}$ sets $U_1$. Without loss of generality, we assume that $i\in U_1$ for at least half of those sets. Let $\mathcal{U}$ denote the collection of such $U_1$, and for each $U_1\in \mathcal{U}$ we let $\Bu=(u_1,\dots,u_{|B(U_{12})|})\in \R^{|B(U_{12})|}$ be its characteristic vector, i.e. $u_i=1$ if $i\in B(U_1)$, and $u_i=0$ otherwise. 

By the definition of $a_{ii_2\dots i_d}(U_1)$, as $i\in U_1$, we can write 
\begin{align*}
& \quad k \sum_{i_2\in B(U_{12}),\dots,i_d \in B(U_d)} a_{i,i_2 \dots i_d}(U_1) \det [\Bc_{i_2}^{\bar{1}},\dots\Bc_{i_d}^{\bar{1}}]\\
&-\sum _{i_0\in I_0} k_{ii_0}\sum_{i_2\in B(U_{12}),\dots,i_d\in B(U_d)} a_{i_0,i_2 \dots i_d}(U_1)\det [\Bc_{i_2}^{\bar{1}},\dots\Bc_{i_d}^{\bar{1}}]\\
&=  k \sum_{i_2\in B(U_{12}),\dots,i_d \in B(U_d)} a_{i,i_2 \dots i_d} (1-u_{i_2})\det [\Bc_{i_2}^{\bar{1}},\dots\Bc_{i_d}^{\bar{1}}] \\
&- \sum _{i_0\in I_0} k_{ii_0}\sum_{i_2\in B(U_{12}),\dots,i_d\in B(U_d)} a_{i_0i_2 \dots i_d} (1-u_{i_2})\det [\Bc_{i_2}^{\bar{1}},\dots\Bc_{i_d}^{\bar{1}}].
\end{align*}

Recall that $|\mathcal{U}|= 2^{|U_{12}|}/n^{O_{B,\ep}(1)}$. Hence, by Lemma \ref{lemma:reduction:A_U12}, we have the following joint probability
\begin{align*}
\P_{x_{11;11},\dots, x_{dd;1n}} &\P_{U_1} \Big(\big| k \sum_{i_2\in B(U_{12}),\dots,i_d \in B(U_d)} a_{i,i_2 \dots i_d} (1-u_{i_2})\det [\Bc_{i_2}^{\bar{1}},\dots\Bc_{i_d}^{\bar{1}}]\\ 
&- \sum _{i_0\in I_0} k_{ii_0}\sum_{i_2\in B(U_{12}),\dots,i_d\in B(U_d)} a_{i_0i_2 \dots i_d} (1-u_{i_2})\det [\Bc_{i_2}^{\bar{1}},\dots\Bc_{i_d}^{\bar{1}}]\big| = O(\beta n^{O_{B,\ep}(1)})\Big) \\
&=n^{-O_{B,\ep}(1)} .
\end{align*}

By applying the Cauchy-Schwarz inequality, we obtain 
\begin{align}\label{eqn:z}
&n^{-O_{B,\ep}(1)}\le \Big[\E_{x_{rs;1i},1\le i\le n, 1\le r,s\le d} \E_{U_1}(...)\Big]^2 \le  \E_{x_{11;11},\dots, x_{dd;1n}} \Big[\E_{U_1}(...))\Big]^2 = \E_{x_{11;11},\dots, x_{dd;1n}} \Big[\E_{\Bu}(...))\Big]^2 \nonumber \\
&\le \E_{x_{rs;1i},1\le i\le n, 1\le r,s\le d} \E_{\Bu,\Bu'}\Big( k \sum_{i_2\in B(U_{12}),\dots,i_d \in B(U_d))} a_{i,i_2 \dots i_d} (u_{i_2}-u_{i_2}')\det [\Bc_{i_2}^{\bar{1}},\dots\Bc_{i_d}^{\bar{1}}]\nonumber \\  
&-\sum _{i_0\in I_0} k_{ii_0}\sum_{i_2\in B(U_{12}),\dots,i_d\in B(U_d)} a_{i_0 i_2 \dots i_d} (u_{i_2}-u_{i_2}')\det [\Bc_{i_2}^{\bar{1}},\dots\Bc_{i_d}^{\bar{1}}]|= O(\beta n^{O_{B,\ep}(1)})\Big) \nonumber \\
&=\E_{x_{rs;1i}, i\notin B(U_{12}); {x_{rs;1i}}',i\in B(U_{12}), 1\le r,s\le d} \Big( k \sum_{i_2\in B(U_{12}),\dots,i_d \in B(U_d))} a_{i,i_2 \dots i_d} \det [{\Bc_{i_2}^{\bar{1}}}',\dots,\Bc_{i_d}^{\bar{1}}]\nonumber \\ 
&-\sum _{i_0\in I_0} k_{ii_0}\sum_{i_2\in B(U_{12}),\dots,i_d\in B(U_d)} a_{i_0 i_2 \dots i_d}  \det [{\Bc_{i_2}^{\bar{1}}}',\dots,\Bc_{i_d}^{\bar{1}}]|= O(\beta n^{O_{B,\ep}(1)})\Big),  
\end{align}
where  ${x^{(i)}_{rs}}':=(u_i-u_i')x^{(i)}_{rs}$ and in the determinant formulas the column $\Bc_{i_2}'$ stands for $(u_{i_2}-u_{i_2}')\Bc_{i_2}$. Also, in the first estimate we used the elementary property that for any function $f$, 
\begin{align*}
&\quad \E_{\Bu,\Bu'}\Big(\|f(\Bu)\|_2=O(\beta n^{O_{B,\ep}(1)}),\|f(\Bu')\|_2=O(\beta n^{O_{B,\ep}(1)})\Big)\\ 
&\le \E_{\Bu,\Bu'}\Big(\|f(\Bu)-f(\Bu')\|_2=O(\beta n^{O_{B,\ep}(1)})\Big).
\end{align*}
The proof is complete by noting that $k$ and $I_0$ are independent of the choice of $i$. 
\end{proof}

\begin{proof}(of Theorem \ref{theorem:decoupling:weak})
Note that by Lemma \ref{lemma:random:A_U}, we just need to deal with a $(d-1)$-multilinear form of the rows $\row_2,\dots,\row_d$. Our next step is to apply this machinery again when fixing $U_{123}=U_{12}\cup U_3$ and letting $U_{12}$ be chosen as a random subset of $U_{123}$. By iterating the machinery $d-1$ times similarly to the proof of Theorem \ref{theorem:dlinear:weak}, we obtain the result as claimed.
\end{proof}

We now conclude this section by proving the inverse step of Section \ref{section:singularity:approach}.

\subsection{Proof of Theorem \ref{theorem:step1}}

We first apply Theorem \ref{theorem:decoupling:weak} to obtain 
\begin{equation}\label{eqn:step1:1}
\P_{\Bx_d}(\sum_{1\le i_d \le dn}\sum_{j_1 \dots j_{d-1}}k_{j_1 \dots j_{d-1}}a_{j_1 \dots j_{d-1} i_d}x_{i_d}\in B(a,\beta))\ge \rho^{2d}/n^{O_{C,\ep}(1)}.
\end{equation}

Set $K_0,\dots, K_d$ to be a sequence of thresholds  with $K_i:=n^{-A/2+2i d}$.  We consider two cases.

{\bf Case 1.}({\it{degenerate case}}) 

{\it Subcase 1.1.} Assume that for all $i\in I_1$, 
\begin{equation}\label{eqn:step1:2}
\sum_{i_2,\dots,i_d}|ka_{ii_2\dots i_d}-\sum k_{ii_0}a_{i_0i_2\dots i_d}|^2 \le K_0^2=n^{-A}.
\end{equation}

As $\sum_{i_1,i_2,\dots,i_d}|a_{i_1i_2\dots i_d}|^2=1$, there exists $i_2,\dots,i_d$ such that
\begin{equation}\label{eqn:step1:3}
\sum_{i_1}|a_{i_1i_2\dots i_d}|^2\ge 1/(dn)^{d-1}\ge n^{-O_d(1)}. 
\end{equation}
We next fix these indices $i_2,\dots,i_d$. It follows from \eqref{eqn:step1:2} that $|ka_{ii_2\dots i_d}-\sum k_{ii_0}a_{i_0i_2\dots i_d}| \le K_0$ for any $i\in I_1$. Set 
$$v_i:=\frac{1}{k}\sum k_{ii_0}a_{i_0i_2\dots i_d}.$$ 
It is clear that the set of $v_i$\rq{}s  is a GAP of rank $|I_0|=O_{B,\ep}(1)$ and size $n^{O_{B,\ep}(1)}$. Also, by definition, with $\Bv=(v_i,i\in I_1)$
$$\|\Bv-(a_{ii_2\dots i_d})_{1\le i\le dn}\|\le n^{1/2}K_0 =n^{-A/2+1/2}.$$

On the other hand, as the vector $(a_{ii_2\dots i_d})_{1\le i\le dn}$ is orthogonal to any row $\row_j(\BM_{n-1})$ of $\BM_{n-1}$, we have 
$$|\langle \Bv,\row_j(\BM_{n-1}) \rangle|\le n^{-A/2+O_{B,\ep}(1)}.$$
Recall that by the approximation and by \eqref{eqn:step1:2}, $\|\Bv\|\ge n^{-O_d(1)}$. Thus by letting $\Bu:=\Bv/\|\Bv\|$, we have 
$$|\langle \Bu,\row_j(\BM_{n-1}) \rangle|\le n^{-A/2+O_{B,\ep,d}(1)}.$$
It is clear that $\Bu$ satisfies all the conditions of Theorem \ref{theorem:step1}, we are done with this subcase.

{\it Subcase 1.2.} From now on we assume that there exists $i\in I_1$ such that  
\begin{equation}\label{eqn:step1:4}
\sum_{i_2,\dots,i_d}|ka_{ii_2\dots i_d}-\sum k_{ii_0}a_{i_0i_2\dots i_d}|^2 \ge K_1^2=n^{-A}.
\end{equation}

Fixing $i$, we apply Theorem \ref{theorem:decoupling:weak} for the index $i_2$, and reconsidering Subcases 1.1 with the new threshold $K_1:=n^{-A/2-2d}$. By iterating the process for at most $d-1$ steps, we will end up with either Subcase 1.1 (and hence done) or with the following non-degenerate case.

{\bf Case 2.}({\it{non-degenerate case}}). There exist a collection $J_1,\dots, J_{d-1}$ of the indices $j_1,\dots,j_{d-1}$ such that $|J_i|=O_{B,\ep}(1)$ and 
$$\sum_{1\le i \le dn}|\sum_{j_1\in J_1, \dots j_{d-1}\in J_{d-1}}k_{j_1 \dots j_{d-1}}a_{j_1 \dots j_{d-1} i}|^2 \ge K_d^2= n^{-A+2d^2}.$$ 
Notice that for each fixed $i_1,\dots,i_{d-1}$ the vector $(a_{i_1,\dots,i_{d-1},i},1\le i\le dn,i\neq i_1,\dots,i_{d-1})$ is orthogonal to any $\row_j^{(i_1,\dots,i_{d-1})}(\BM_{n-1})$, the $j$-th row of $\BM_{n-1}$ without components $i_1,\dots,i_{d-1}$. By adding zeros to the missing components $i_1,\dots,i_{d-1}$ if needed, we see that the $\R^{nd}$ vector $(a_{i_1,\dots,i_{d-1},i},1\le i\le dn)$ is now orthogonal to $\row_j(\BM_{n-1})$. 

Set $J=J_1\cup \dots \cup J_{d-1}$, we thus have
$$\sum_{i\notin J,1\le i\le dn} a_{j_1,\dots,j_{d-1},i} \row_j(i) = -\sum_{i\in J,1\le i\le dn}a_{j_1,\dots,j_{d-1},i}\row_j(i),$$
where the entries $a_{j_1,\dots,j_{d-1},i}$ are set to  be zero if the indices are not distinct.

Now we set $\Bw_1:=(w_i)_{i\notin J}$ and $\Bw_2:=(w_i)_{i\in J}$, where $w_i:=k_{j_1 \dots j_{d-1}}a_{i_1 \dots i_{d-1} i}$. Then 
$$\langle \Bw_1,\row_j^{(\bar{J})}(\BM_{n-1})\rangle= -\langle \Bw_2, \row_j^{(J)}(\BM_{n-1})\rangle.$$
Set $\Bv:=\Bw/\|\Bw\|$. Theorem \ref{theorem:linear} applied to \eqref{eqn:step1:1} implies that $\Bv$ can be approximated by a vector $\Bu$ as follows.

\begin{itemize}
\item $|u_i-v_i|\le n^{-A/2+O_{B,\ep,d}(1)}$ for all $i$.
\item There exists a GAP of rank $O_{B,\ep}(1)$ and size $n^{O_{B,\ep}(1)}$ that contains at least $dn-n^\ep$ components $u_i$.
\item All the components $u_i$, and all the generators of the GAP are rational complex numbers of the form $\frac{p}{q}+\sqrt{-1}\frac{p'}{q'}$, where $|p|,|q|,|p'|,|q'| \le n^{A/2+O_{B,\ep}(1)}$.
\end{itemize}

Note that, by the approximation above, $\|\Bu\| \asymp 1$ and $|\langle \Bu_1,\row_j^{(\bar{J})}(\BM_{n-1}) \rangle + \langle \Bu_2,\row_j^{(J)}(\BM_{n-1})| \le n^{-A/2+O_{B,\ep}(1)}$ for all row vectors of $\BM_{n-1}$.

\section{Singularity of block matrices: proof sketch for Theorem \ref{theorem:step2}} \label{section:singularity:step2}


Our first ingredient is the following variant of Theorem \ref{theorem:linear} in which random variables are replaced by random matrices and $a_i$ are replaced by vectors.

\begin{theorem}[Inverse Littlewood-Offord for sequence of random operators]\label{theorem:linear:operator}
Let $\{\Bu^{(i)}=(\Bu_1^{(i)},\dots,\Bu_d^{(i)}), 1\le i\le n\}$ be a sequence of $n$ vectors in $\C^d$ such that the following concentration-type holds with high probability 
$$\sup_{\Bu\in \C^d} \P_{X^{(1)},\dots,X^{(n)}} (\sum_{1\le i\le n} X^{(i)}\Bu^{(i)}\in B(\Bu,\beta))=\gamma=n^{-O(1)},$$
where $X^{(1)},\dots,X^{(n)}$ are iid block matrices whose entries are copies of $(x_{11},\dots,x_{dd})$ from \eqref{eqn:covariance}. Then there exist a positive constant $\delta$ and $d^2$ numbers $c_{11},\dots,c_{dd}$ such that the least singular value $\sigma_d$ (largest singular value $\sigma_1$) of the matrix $(c_{ij})_{1\le i,j\le d}$ is at least $\delta$ (at most $\delta^{-1}$) and for any number $n'$ between $n^\ep$ and $n$, there exists a proper symmetric GAP $Q=\{\sum_{i=1}^r k_ig_i : |k_i|\le K_i \}\subset \C^d$ such that

\begin{itemize}

\item At least $n-n'$ elements of $V:=\{(c_{11}\Bu^{(i)}_1+\dots+c_{1d}\Bu^{(i)}_d, \dots, c_{d1}\Bu^{(i)}_1+\dots+c_{dd}\Bu^{(i)}_d), 1\le i\le n\}$ are $\beta$-close to $Q$.


\item $Q$ has small rank, $r=O_{B,\ep}(1)$, and small size

$$|Q| \le \max \{O_{B,\ep}(\frac{\gamma^{-1}}{\sqrt{n'}}),1\}.$$


\item There is a non-zero integer $p=O_{B,\ep}(\sqrt{n'})$ such that all
 steps $g_i$ of $Q$ have the form  $g_i=(g_{i1},\dots,g_{id})$, where $g_{ij}=\beta\frac{p_{ij}} {p} $ with $p_{ij} \in \Z$ and $|p_{ij}|=O_{B,\ep}(\beta^{-1} \sqrt{n'}).$

\end{itemize}

\end{theorem}

In application, $X^{(1)},\dots, X^{(n)}$ will be the $d\times d$ blocks of $\BM_{n-1}$. It is crucial to notice that, as most of the elements of $V$ are $\beta$-close to $Q$, and as the matrix $(c_{ij})$ is far from being degenerate, it follows from Theorem \ref{theorem:linear:operator} that most of the individual components $\Bu^{(i)}_j$ are also close to another GAP of small rank and size (see Corollary \ref{cor:ILOlinear:new}). We will present the proof of Theorem \ref{theorem:linear:operator} in Appendix \ref{section:ILOlinear:new:proof} by following the treatment from \cite{NgV}. For the rest of this section we sketch the proof of Theorem \ref{theorem:step2} following \cite{NgO} and \cite{TVcomp}.

Let $\mathcal{N}$ be the number of such structural vectors $\Bu$ from Theorem \ref{theorem:step1}. Because each GAP is determined by its generators and dimensions, the number of $Q$'s is bounded by 
$$(n^{2A+O_{B,\ep}(1)})^{O_{B,\ep}(1)} (n^{O_{B,\ep}(1)})^{O_{B,\ep}(1)} = n^{O_{A,B,\ep}(1)}.$$ 

Next, for a given $Q$, there are at most $n^{O_{B,\ep}(n)}$ ways to choose the $nd-2n^\ep$ components $u_i$ that $Q$ contains,  and $n^{O_{A,B,\ep}(n^\ep)}$ ways to choose the remaining components from the set $\{\frac{p}{q}+i \frac{p'}{q'}, |p|,|q|,|p'|,|q'|\le n^{A/2+O_{B,\ep}(1)}\}$. Hence, we obtain the key bound 
\begin{equation}\label{eqn:step2:N'}
\mathcal{N}\le n^{O_{A,B,\ep}(1)}   n^{O_{B,\ep}(n)}  n^{O_{A,B,\ep}(n^\ep)} = n^{O_{B,\ep}(n)}.
\end{equation}


From now on, by conditioning on $\Bu_1$ and on the entries of $\BM_{n-1}$ corresponding to the index $i_1,\dots,i_d$ of $\Bu_1$, without loss of generality we assume that $\Bu_1$ vanishes. Set $\beta_0:=n^{-A/2+O_{B,\ep}(1)}$, the bound obtained from the conclusion of Theorem \ref{theorem:step1}. We will denote the blocks of $\BM_{n-1}$ by $X^{(i)}_j$ with $1\le i\le n$ and $1\le j\le n-1$.
 For a given vector $\Bu$, we define $\P_{\beta_0}(\Bu)$ as follows
$$\P_{\beta_0}(\Bu):=\P\Big( \sum_{1\le i\le n} X^{(i)}_j\Bu^{(i)}\in B(0,\beta_0)  \mbox{ for at least } (n-1)-O_{B,\ep}(1) \mbox{ indices } j\Big).$$

If $\Bu$ satisfies the property above, we say that $\Bu$ is {\it $\beta_0$-orthogonal to  almost all blocks of $\BM_{n-1}$}. Because the blocks of $\BM_{n-1}$ are iid, 
$$\P_{\beta_0}(\Bu) \le \big(\P_{X^{(1)},\dots,X^{(n)}} (\sum_{1\le i\le n} X^{(i)}\Bu^{(i)}\in B(0,\beta_0) \big)^{n-O(1)}:=\gamma_{\beta_0}(\Bu)^{n-O(1)},$$ 
where $X^{(i)}$ are iid copies of $(x_{st})_{1\le s,t\le d}$. 

If $\gamma=\gamma_{\beta_0}(\Bu)$ is small, say $n^{-\Omega(1)}$, then $\P_{\beta_0}(\Bu)$ is $n^{-\Omega(n)}$. Thus the contribution of these $\P_{\beta_0}(\Bu)$ in the total sum  $\sum_{\Bu}\P_{\beta_0}(\Bu)$ is negligible as $\mathcal{N}=n^{O(n)}$. 

For the case $\gamma$ is comparably large, $\gamma=n^{-O(1)}$, then by Theorem \ref{theorem:linear:operator}, most of the components $u_i$ are close to a new GAP of rank $O(1)$ and of size $O(\gamma^{-1}/\sqrt{n})$. This would then enable us to approximate $\Bu$ by a new vector $\Bu'$ in such a way that $|\langle \Bu',\row_i(\BM_{n-1})\rangle|$, where we recall that $\row_i(\BM_{n-1})$ is the $i$-th row of $\BM_{n-1}$, is still of order $O(\beta_0)$ and the components of $\Bu'$ are now from the new GAPs (after a linear transformation). The number $\mathcal{N'}$ of these $\Bu'$ can be bounded by $(\gamma^{-1}/n^\ep)^{n}$, while we recall that $\P_{\beta_0}(\Bu')$ is of order $\gamma^{-n}$. Thus, summing over $\Bu'$ we obtain the desired bound
$$\sum_{\Bu'} \P_{\beta_0}(\Bu') \le \#\{ \mbox{ new GAPs } \} (\gamma^{-1}/n^\ep)^{n} \gamma^{-n} = O(n^{-\ep n+O(1)}).$$





To proceed further, we need the following elementary claim. 



\begin{claim}\label{claim:comparison}
Assume that $\BC=(c_{ij})$ is a $d\times d$ matrix such that $\delta \le \sigma_d(\BC)\le \sigma_1(\BC) \le \delta^{-1}$. Let $\Bu'=(\Bu'^{(1)},\dots,\Bu'^{(n)})$, where $\Bu'^{(i)} =\BC \Bu^{(i)}$. Then we have 
$$\gamma_\beta(\Bu)\le \gamma_{\delta^{-1} \beta}(\Bu')\le \gamma_{\delta^{-2}\beta}(\Bu).$$ 
\end{claim}


By paying a factor of $n^{O_{B,\ep}(1)}$ in probability, we may assume that $|\langle \Bu,\row_i(M_{n}) \rangle | \le  \beta_0$ for the first  $d(n-1)-O_{B,\ep}(1)$ rows of $\BM_{n}$. Also, by paying another factor of $n^{n^\ep}$ in probability, we may assume that the first $d n_0$ components of $\Bu$ belong to a GAP $Q$, and the Euclidean norm of $\Bu^{(n_0)}$ is comparable, $\|\Bu^{(n_0)}\|= \Omega(1/n)$ (recall that $\|\Bu\| \asymp 1$), where 
$$n_0:=n-n^\ep.$$ 

We refer to the remaining $u_i$'s as exceptional components. Note that these extra factors do not affect our final bound $\exp(-\Omega(n))$. Because $\|\Bu^{(n_0)}\|= \Omega(1/n)$ and $X^{(n_0)}$ is not degenerate with high probability, there exist positive constants $c_1,c_2$ such that $c_2<1$ and for any $\beta\le c_1/\sqrt{n}$  we have
\begin{eqnarray}\label{eqn:step2:upper}
\gamma_{\beta}(\Bu) &\le& \sup_a \P_{X^{(n_0)}}(|X^{(n_0)}\Bu^{(n_0)}-a|\le \beta) \le 1-c_2.
\end{eqnarray}


\subsection{Classification} Next, let $C$ be a sufficiently large constant depending on $B$ and $\ep$ but not $A$. We classify $\Bu$ into two classes $\mathcal{B}$ and $\mathcal{B}'$, depending on whether $\P_{\beta_0}(\Bu)\ge n^{-Cn}$ or not. Because of \eqref{eqn:step2:N'}, for $C$ large enough, 
\begin{equation}\label{eqn:step2:B'}
\sum_{\Bu\in \mathcal{B}'}\P_{\beta_0}(\Bu)\le n^{O_{B,\ep}(n)}/n^{Cn} \le n^{-n/2}.
\end{equation}

For the set $\mathcal{B}$ of remaining vectors, we divide it into two subfamilies. Set $n':=n^{1-\ep}$. We say that $\Bu\in \mathcal{B}$ is compressible if for any $n'$ components $\Bu^{(i_1)},\dots,\Bu^{(i_{n'})}$ among the $\Bu^{(1)},\dots, \Bu^{(n_0)}$, we have 
\begin{equation}\label{eqn:step2:degenerate}
\sup_a\P_{X_{i_1},\dots,X_{i_{n'}}}(|X_{i_1}\Bu^{(i_1)}+\dots+X_{i_{n'}}\Bu^{(i_{n'})}-a|\le n^{-B-4})\ge (n')^{-1/2+o(1)}.
\end{equation}

Let $\mathcal{B}_1$ and $\mathcal{B}_2$ be the set of compressible and incompressible vectors respectively. We focus on $\mathcal{B}_1$ first.

\subsection{Approximation for compressible vectors} Set $\beta:=n^{-B-4}$. It follows from Theorem \ref{theorem:linear:operator} that, among any $\Bu^{(i_1)},\dots,\Bu^{(i_{n'})}$, there are, say, at least $n'/2+1$ vectors that belong to a ball of radius $\beta$ in $\C^d$ (because our GAP now has only one element after a linear transformation $\BC=(c_{ij})_{1\le i,j\le d}$). A simple covering argument then implies that there is a ball of radius $2\beta$ that contains all but $n'-1$ vectors $\Bu^{(i)}$.
 
Thus there exists a vector $\Bu'=(\Bu'^{(1)},\dots,\Bu'^{(n)})\in (2\beta)\cdot (\Z+ \sqrt{-1} \Z)^{nd}$ such that

\begin{itemize}
\item $|\BC \Bu^{(i)}-\Bu'^{(i)}|\le 4\beta$ for all $i$;
\item $\Bu'^{(i)}$ takes the same vector-value for at least $n_0-n'$ indices $i$.
\end{itemize}

Because of the approximation and Assumption \ref{condition:bound}, whenever $\sum_{1\le i\le n} X^{(i)}\Bu^{(i)}\in B(\Bu,\beta)$, we also have 
$$|\sum_{1\le i\le n} X^{(i)}{\Bu'}^{(i)}-\BC X^{(i)}\Bu^{(i)}|\le n(n^{B+1}+n^\alpha)(4\beta)+\beta_0:=\beta'.$$ 
By definition, $\beta' \le c_1/\sqrt{n}$, and thus by \eqref{eqn:step2:upper}, $\P_{\beta'}(\Bu') \le (1-c_2)^{(1-o(1))n}$. Now we bound the number of $\Bu'$ obtained from the approximation. First, there are $O(n^{n-n_0+n'}) = O(n^{2n^{1-\ep}})$ ways to choose those $\Bu'^{(i)}$ that take the same vector $\Bw\in \C^d$, and there are just $O(\beta^{-d})$ ways to choose $\Bw$. The remaining components belong to the set $(2\beta)^{-d}\cdot (\Z + i\Z)^d$, and thus there are at most $O((\beta^{-d})^{n-n_0+n'})= O(n^{O_{A,B,\ep}(n^{1-\ep})})$ ways to choose them. Hence we obtain the total bound
\begin{align*}
\sum_{\Bu \in \mathcal{B}_1}\P_{\beta_0}(\Bu) \le \sum_{\Bu'}\P_{\beta'}(\Bu') &\le O(n^{2n^{1-\ep}}) O(n^{O_{A,B,\ep}(n^{1-\ep})}) (1-c_2)^{(1-o(1))n} \le (1-c_2)^{(1-o(1))n}.
\end{align*}

\subsection{Approximation for incompressible vectors} The treatment here is similar, we will sketch the main steps, leaving the details for the reader as an exercise.   

First, by exposing the rows of $\BM_{n-1}$ accordingly, and by paying an extra factor $\binom{n_0}{n'}=O(n^{n^{1-\ep}})$ in probability, we can assume that the components $\Bu^{(n_0-n'+1)},\dots,\Bu^{(n_0)}$ satisfy 
$$\sup_{\Ba\in \C^d}\P_{X^{(n_0-n'+1)},\dots,X^{(n_0)}}(|X^{(n_0-n'+1)} \Bu^{(n_0-n'+1)}+\dots+X^{(n_0)}\Bu^{(n_0)}-\Ba|\le n^{-B-4})\le (n')^{-1/2+o(1)}$$
\begin{equation}\label{eqn:step2:non-degenerate}
\le n^{-1/2+\ep/2+o(1)}.
\end{equation}

Next, define a radius sequence $\beta_k, k\ge 0$ where $\beta_0=n^{-A/2+O_{B,\ep}(1)}$ is the bound obtained from the conclusion of Theorem \ref{theorem:step1}, and $\beta_{k+1}:= (n^{B+2}+n^{\alpha+1}+1)^2 \beta_k.$ Also define  
$$\gamma_{\beta_k}(\Bu):= \sup_{\Ba\in \C^d}\P_{X^{(n_0-n'+1)},\dots,X^{(n_0)}}(|X^{(n_0-n'+1)} \Bu^{(n_0-n'+1)}+\dots+X^{(n_0)}\Bu^{(n_0)}-\Ba|\le \beta).$$
Clearly $\P_{\beta_k}(\Bu) \le \gamma^{n-1}_{\beta_k}(\Bu)$. As $\BC$ is non-degenerate, with $\Bu'$ from Theorem \ref{theorem:linear:operator},

\begin{align}\label{eqn:u':discussion}
\gamma_{\beta_k}(\Bu) \le  \gamma_{(n(n^{B+1}+n^\alpha) \beta_k +\beta_k))}(\Bu') \le \gamma_{\beta_{k+1}}(\Bu).
\end{align}

Furthermore, we have freedom to choose $k$ before applying Theorem \ref{theorem:linear:operator} to obtain $\Bu'$. By the pigeon-hole principle, there exists $k=k_0(\Bu)\le C\ep^{-1}$ such that 
\begin{equation}\label{eqn:step2:pigeon-hole}
\pi_{\beta_{k_0+1}}(\Bu) \le n^{\ep n} \pi_{\beta_{k_0}}(\Bu).
\end{equation}

Since $A$ was chosen sufficiently large compared to $O_{B,\ep}(1)$ and $C$, we have $\beta_{k_0+1}\le n^{-B-4}$. With this choice of $k_0$, we apply Theorem \ref{theorem:linear:operator} to obtain an approximation $\Bu'$ of $\BC\Bu$ with the following properties. 

\begin{enumerate}[(i)]
\item $|\BC \Bu^{(i)}-\Bu'^{(i)}|\le \beta_{k_0}$ for all $i$. 
\item The components of ${\Bu'}^{(i)}$ belong to $Q$ for all but $n^{1-2\ep}$ indices $i$, and the generators of $Q$, belong to the set $\beta_{k_0}\cdot \{p/q +\sqrt{-1} p'/q' , |p|,|q|,|p'|,|q'|\le n^{A/2+O_{B,\ep}(1)}\}$. 
\item $Q$ has rank $O_{B,\ep}(1)$ and size $|Q|=O(\gamma_{\beta_{k_0}(\Bu)}^{-1}/n^{1/2-\ep})$.
\end{enumerate}

Let $\mathcal{B'}$ be the collection of such $\Bu'$. By definition,
\begin{equation}\label{eqn:step2:Pu'}
\P_{(n^{B+2}+n^{\alpha+1}+1)\beta_{k_0}}(\Bu')=\gamma^{n-1}_{(n^{B+2}+n^{\alpha+1}+1)\beta_{k_0}}(\Bu')  \le \gamma^{n-1}_{\beta_{k_0+1}}(\Bu) \le n^{\ep n}\gamma^{n-1}_{\beta_{k_0}}(\Bu).
\end{equation}


Arguing similarly to the treatment for $\mathcal{N}$, we can bound the cardinality $\mathcal{N'}$ of $\mathcal{B'}$ by 

\begin{equation}\label{eqn:step2:N}
\mathcal{N}'\le \gamma_{\beta_{k_0}}(\Bu)^{-n}/n^{(1/2-\ep-o(1))n}.
\end{equation}
 
 It follows from \eqref{eqn:step2:Pu'} and \eqref{eqn:step2:N} that
\begin{align*}
\sum_{\Bu'\in \mathcal{B'}}\P_{(n^{B+2}+n^{\alpha+1}+1)\beta_{k_0}}(\Bu') \le  n^{-(1/2-4\ep-o(1))n},
\end{align*}
completing the treatment for incompressible vectors.

\section{Universality of random block matrices: Proof of Theorem \ref{thm:main}} \label{sec:universality}

This section is devoted to Theorem \ref{thm:main}.  We begin by introducing the following notation.  Given a $n \times n$ matrix $\BM$, we let $\mu_{\BM}$ denote the empirical measure built from the eigenvalues of $\BM$ and $\nu_{\BM}$ denote the symmetric empirical measure built from the singular values of $\BM$.  That is,
$$ \mu_{\BM} := \frac{1}{n} \sum_{i =1}^n \delta_{\lambda_i(\BM)} \quad \text{and}\quad \nu_{\BM} := \frac{1}{2n} \sum_{i =1}^n \left( \delta_{\sigma_i(\BM)} + \delta_{-\sigma_i(\BM)} \right), $$ 
where $\lambda_1(\BM), \ldots, \lambda_n(\BM) \in \mathbb{C}$ are the eigenvalues of $\BM$ and $\sigma_1(\BM) \geq \cdots \geq \sigma_n(\BM)$ are the singular values of $\BM$.  Recall that $F^{\BM}$ is the ESD of $\BM$.  In particular, we have
$$ F^{\BM}(x,y) = \int_{-\infty}^x \int_{-\infty}^{y} \mu_{\BM}(z)dt ds, $$
where $z = s + \sqrt{-1}t$.  

Many of the techniques used to study Hermitian matrices fail to work for non-Hermitian matrices \cite[Section 11.1]{BSbook}.  Consider a $n \times n$ non-Hermitian matrix $\BM$.  In \cite{G1,G2}, Girko introduced a natural connection between $\mu_{\BM}$ and the collection of measures $\{\nu_{\BM - z \BI}\}_{z \in \C}$.  Formally, we present this connection as Lemma \ref{lemma:girko} below.  

Lemma \ref{lemma:girko} follows from \cite[Lemma 11.2]{BSbook} and is based on Girko's original observation \cite{G1,G2}.  The lemma has appeared in a number of different forms; for example, see \cite[Lemma 4.3]{BC} and \cite{GK}.  

\begin{lemma}[Lemma 11.2 from \cite{BSbook}] \label{lemma:girko}
Let $\BM$ be a $n \times n$ matrix.  For any $uv \neq 0$, we have
\begin{align*}
	&\iint e^{\sqrt{-1} ux + \sqrt{-1} uy} F^{\BM}(dx, dy) \\
	&\qquad \qquad= \frac{u^2 + v^2}{4 \sqrt{-1} u \pi} \iint \frac{ \partial }{\partial s} \left[ \int_{0}^\infty \ln |x|^2 \nu_{\BM - z\BI}(dx) \right] e^{\sqrt{-1} us + \sqrt{-1} vt} dt ds, 
\end{align*}
where $z = s + \sqrt{-1} t$.  
\end{lemma}

We define the function
\begin{equation} \label{eq:def:gMn}
	g_{\BM}(s,t) := \frac{\partial}{\partial s} \int_{0}^\infty \log |x|^2 \nu_{\BM-z\BI}(dx),
\end{equation}
where $z = s + \sqrt{-1} t$.  We also define
\begin{equation} \label{eq:def:g}
	g(s,t) := \left\{
		\begin{array}{ll}
		\frac{2s}{s^2 + t^2}, & \text{if } s^2 + t^2 >1\\
		2s, & \text{otherwise}
	\end{array}
   	\right. .
\end{equation}

Let $\{\BX_n\}_{n \geq 1}$ be a sequence of matrices that satisfies condition {\bf C0} with parameter $d \geq 2$ and atom variables $(\xi_{st})_{s,t=1}^d$.  We define the $2dn \times 2dn$ Hermitian matrix
$$ \BH_n = \BH_n(z) := \begin{bmatrix} \Bzero & \frac{1}{\sqrt{n}} \BX_n - z \BI \\ \frac{1}{\sqrt{n}} \BX_n^\ast - \bar{z} \BI & \Bzero \end{bmatrix} $$
for $z \in \mathbb{C}$.  It is straight-forward to verify that the eigenvalues of $\BH_n$ are given by 
$$ \pm \sigma_1\left(\frac{1}{\sqrt{n}}\BX_n - z\BI \right), \pm \sigma_2\left(\frac{1}{\sqrt{n}}\BX_n - z\BI\right), \ldots, \pm \sigma_{dn}\left(\frac{1}{\sqrt{n}}\BX_n - z\BI \right). $$  
In other words, $\nu_{\frac{1}{\sqrt{n}} \BX_n - z \BI}$ is the empirical spectral measure of $\BH_n$.  By Lemma \ref{lemma:girko}, the problem of studying $\mu_{\frac{1}{\sqrt{n}} \BX_n}$ reduces to studying the eigenvalue distribution of $\BH_n$.

\subsection{Truncation}

In practice, it will be more convenient to work with a truncated version of $\BH_n$.  That is, we will work with a new matrix $\hat{\BH}_n$ whose entries are truncated versions of the entries of the original matrix $\BH_n$.  This subsection is devoted to the following standard truncation arguments.  

Let $\{\BX_n\}_{n \geq 1}$ be a sequence of matrices that satisfies condition {\bf C0} with parameter $d \geq 2$ and atom variables $(\xi_{st})_{s,t=1}^d$, and assume 
\begin{equation} \label{eq:m2eta}
	m_{2+\eta} := \max_{1 \leq s,t \leq d} \E|\xi_{st}|^{2 + \eta} < \infty, 
\end{equation}
for some $\eta > 0$.  Let $\delta > 0$.  For each $s,t \in \{1,\ldots,d\}$, we define
$$ \tilde{\xi}_{st}^{(n)} := \xi_{st} \indicator{|\xi_{st}| \leq n^{\delta}} - \E \left[\xi_{st} \indicator{|\xi_{st}| \leq n^{\delta}} \right] \quad \text{and} \quad  \hat{\xi}_{st}^{(n)} := \frac{\tilde{\xi}_{st}^{(n)}}{\sqrt{d\var(\tilde{\xi}_{st}^{(n)})}}. $$
Here $\oindicator{E}$ denotes the indicator function of the event $E$.  We present the following standard truncation lemma.  

\begin{lemma} \label{lemma:abcdtruncn}
Let $\{\BX_n\}_{n \geq 1}$ be a sequence of matrices that satisfies condition {\bf C0} with parameter $d \geq 2$ and atom variables $(\xi_{st})_{s,t=1}^d$, and assume \eqref{eq:m2eta} holds for some $\eta > 0$.  For each $\delta > 0$, there exists $n_0 > 0$ such that the following holds for all $n > n_0$.  
\begin{enumerate}[(i)]
\item For each $s,t \in \{1,\ldots,d\}$, $\hat{\xi}_{st}^{(n)}$ has mean zero and variance $1/d$. 
\item a.s. $\max_{1 \leq s,t \leq d} \left|\hat{\xi}_{st}^{(n)}\right| \leq 4 n^{\delta}$. 
\item We have \label{truncn:var}
$$ \max_{1 \leq s,t \leq d} \left| 1/d - \var( \tilde{\xi}_{st}^{(n)} ) \right| \leq 2 \frac{m_{2+\eta}} {n^{\delta \eta}}. $$
\item We have \label{truncn:corr}
\begin{align*}
	\max_{ (s,t) \neq (u,v)} \left|\E \left[\hat{\xi}_{st}^{(n)} \overline{\hat{\xi}_{uv}^{(n)}} \right] \right| \leq 10 \frac{\sqrt{m_{2 + \eta}}} {n^{\delta \eta/2}}.
\end{align*}
\end{enumerate}
\end{lemma}
\begin{proof}(of Lemma \ref{lemma:abcdtruncn})
We first note that
\begin{equation} \label{eq:var14n}
	\var(\tilde{\xi}_{st}^{(n)}) \leq \E|\xi_{st}|^2 \indicator{|\xi_{st}| \leq n^{\delta}} \leq \E|\xi_{st}|^2 = 1/d 
\end{equation}
for all $s,t \in \{1,\ldots,d\}$.  We also note that
\begin{equation} \label{eq:varn}
	\left| 1/d- \var(\tilde{\xi}_{st}^{(n)}) \right| = \left| 1/d - \E \left|\tilde{\xi}_{st}^{(n)}\right|^2 \right| \leq 2 \E|\xi_{st}|^2 \indicator{|\xi_{st}| > n^{\delta} } \leq 2\frac{m_{2+\eta}}{n^{\delta \eta}}.
\end{equation}
Since this holds for all $s,t \in \{1,\ldots,d\}$, we obtain \eqref{truncn:var}.  We now take $n_0$ sufficiently large such that 
\begin{equation} \label{eq:m2en}
	8m_{2 + \eta} \leq n_0^{\delta \eta} 
\end{equation}
and
$$ \min_{1 \leq s,t \leq d} \var(\tilde{\xi}_{st}^{(n)}) > \frac{1}{2d} $$  
for all $n > n_0$; let $n > n_0$.  Then each $\hat{\xi}_{st}^{(n)}$ has mean zero and variance $1/d$ by construction.  Moreover, we have a.s.
$$ \left|\hat{\xi}_{st}^{(n)} \right| \leq \frac{2n^{\delta}}{\sqrt{d\var(\tilde{\xi}_{st}^{(n)})}} \leq 4 n^{\delta} $$
for all $s,t \in \{1,\ldots,d\}$.  We now verify \eqref{truncn:corr}; fix $s,t,u,v \in \{1,\ldots,d\}$ with $(s,t) \neq (u,v)$. We have
\begin{align*}
	\left| \E \left[ \hat{\xi}_{st}^{(n)} \overline{\hat{\xi}_{uv}^{(n)}} \right] - \E \left[ \tilde{\xi}_{st}^{(n)} \overline{\tilde{\xi}_{uv}^{(n)}} \right] \right| &\leq \E \left| \hat{\xi}_{st}^{(n)} \overline{\hat{\xi}_{uv}^{(n)}} \right| \left| 1 - d \sqrt{\var( \tilde{\xi}_{st}^{(n)})} \sqrt{ \var(\tilde{\xi}_{uv}^{(n)})} \right| \\
		&\leq \left| \frac{1}{d} -\sqrt{\var( \tilde{\xi}_{st}^{(n)})} \sqrt{ \var(\tilde{\xi}_{uv}^{(n)})} \right| \\
		&\leq d \left| \frac{1}{d^2} - \var(\tilde{\xi}_{st}^{(n)}) \var(\tilde{\xi}_{uv}^{(n)}) \right| \\
		&\leq \left| \var(\tilde{\xi}_{st}^{(n)}) - \frac{1}{d} \right| + \left| \var(\tilde{\xi}_{uv}^{(n)}) - \frac{1}{d} \right| \\
		&\leq 4 \frac{m_{2+\eta}}{n^{\delta \eta}}
\end{align*}
by \eqref{eq:var14n} and the Cauchy-Schwarz inequality.  Here the last inequality follows from \eqref{eq:varn}.  By another application of Cauchy-Schwarz and \eqref{eq:var14n}, we obtain
\begin{align*}
	\left| \E \left[ \tilde{\xi}_{st}^{(n)} \overline{ \tilde{\xi}_{uv}^{(n)}} \right] \right| &\leq \sqrt{ \E|\xi_{st}|^2 \indicator{|\xi_{st}| > n^{\delta}} } + \sqrt{ \E|\xi_{uv}|^2 \indicator{|\xi_{uv}| > n^{\delta}} } \\
		& \qquad + 2 \E|\xi_{st}|^2 \indicator{|\xi_{st}| > n^{\delta}} + 2 \E|\xi_{uv}|^2 \indicator{|\xi_{uv}| > n^{\delta}} \\
		&\leq 2 \frac{ \sqrt{m_{2+\eta}}}{n^{\delta \eta / 2}} + 4 \frac{m_{2+\eta}}{n^{\delta \eta}}. 
\end{align*}
Combining the bounds above with \eqref{eq:m2en}, we obtain
\begin{equation} \label{eq:hatcovbnd}
	\left|\E \left[ \hat{\xi}_{st}^{(n)} \overline{\hat{\xi}_{uv}^{(n)}} \right] \right| \leq 10 \frac{\sqrt{m_{2+\eta}}}{n^{\delta \eta/2}}. 
\end{equation}
Since \eqref{eq:hatcovbnd} holds for any $(s,t) \neq (u,v)$, the proof of the lemma is complete.
\end{proof}

We will continue to use the notation introduced in Definition \ref{def:C0}.  That is, for any $s,t \in \{1,\ldots,d\}$ and $1 \leq i, j \leq n$, we let $x_{st;ij}$ denote the $(i,j)$-entry of the matrix $\BX_{n,st}$.  For every $s,t \in \{1,\ldots,d\}$, $n \geq 1$, and $1 \leq i,j \leq n$, we define
$$ \tilde{x}_{st;ij}^{(n)} := x_{st;ij} \indicator{|x_{st;ij}| \leq n^{\delta}} - \E \left[ x_{st;ij}  \indicator{|x_{st;ij}| \leq n^{\delta}} \right] $$
and
$$ \hat{x}_{st;ij}^{(n)} := \frac{\tilde{x}_{st;ij}^{(n)}}{ \sqrt{d \var(\tilde{x}_{st;ij}^{(n)} )} }. $$
Set $\tilde{\BX}_{n,st} := \left(\tilde{x}_{st;ij}^{(n)} \right)_{i,j=1}^n$ and $\hat{\BX}_{n,st} := \left(\hat{x}_{st;ij}^{(n)} \right)_{i,j=1}^n$ for every $n \geq 1$ and $s,t \in \{1,\ldots,d\}$.  We also define the $dn \times dn$ random block matrices
$$ \tilde{\BX}_n := \left( \tilde{\BX}_{n,st} \right)_{s,t=1}^d, \quad \hat{\BX}_n := \left( \hat{\BX}_{n,st} \right)_{s,t=1}^d. $$

For $z \in \mathbb{C}$, we define the $2dn \times 2dn$ matrices
$$ \tilde{\BH}_n = \tilde{\BH}_n(z) := \begin{bmatrix} \Bzero & \frac{1}{\sqrt{n}} \tilde{\BX}_n - z \BI \\ \frac{1}{\sqrt{n}} \tilde{\BX}_n^\ast - \bar{z} \BI & \Bzero \end{bmatrix} $$
and
$$ \hat{\BH}_n = \hat{\BH}_n(z) := \begin{bmatrix} \Bzero & \frac{1}{\sqrt{n}} \hat{\BX}_n - z \BI \\ \frac{1}{\sqrt{n}} \hat{\BX}_n^\ast - \bar{z} \BI & \Bzero \end{bmatrix}. $$

We will make use of the following corollary to the law of large numbers.

\begin{lemma}[Law of large numbers] \label{lemma:lln}
Let $\{\BX_n\}_{n \geq 1}$ be a sequence of random matrices that satisfies condition {\bf C0} with parameter $d \geq 2$ and atom variables $(\xi_{st})_{s,t=1}^d$, and assume \eqref{eq:m2eta} holds for some $\eta > 0$.  Let $\delta > 0$.  Then a.s.
\begin{align} \label{eq:llnhs}
	\limsup_{n \rightarrow \infty} \frac{1}{n^2} \| \BX_n \|_2^2 \leq d,\\
	\limsup_{n \rightarrow \infty} \frac{1}{n^2} \| \hat{\BX}_n\|_2^2 \leq 8d \label{eq:llnhshat},
\end{align}
and
\begin{equation} \label{eq:lln2eta}
	\lim_{n \rightarrow \infty} \frac{1}{n^2} \sum_{s,t=1}^d \sum_{i,j=1}^n |x_{st;ij}|^{2 + \eta} \indicator{|x_{st;ij}| > n^{\delta}|} = 0.
\end{equation}
\end{lemma}
\begin{proof}(of Lemma \ref{lemma:lln})
We first prove \eqref{eq:llnhs}.  We begin by noting that
$$ \frac{1}{n^2} \| \BX_n \|_2^2 = \sum_{s,t=1}^d \frac{1}{n^2} \sum_{i,j=1}^n |x_{st;ij}|^2. $$
For any $s,t \in \{1,\ldots,d\}$, we apply the law of large numbers and obtain a.s.
$$ \limsup_{n \rightarrow \infty} \frac{1}{n^2} \sum_{i,j=1}^n |x_{st;ij}|^2 \leq \E|\xi_{st}|^2 = \frac{1}{d}. $$
Since $d$ is fixed, independent of $n$, we conclude that a.s.
$$ \limsup_{n \rightarrow \infty} \sum_{s,t=1}^d \frac{1}{n^2} \sum_{i,j=1}^n |x_{st;ij}|^2 \leq d, $$
and the proof of \eqref{eq:llnhs} is complete.  

For \eqref{eq:llnhshat}, we apply the bounds in Lemma \ref{lemma:abcdtruncn} to obtain
$$ \sum_{s,t=1}^d \sum_{i,j=1}^n \left|\hat{x}^{(n)}_{st;ij} \right|^2 \leq 2 \sum_{s,t=1}^d \sum_{i,j=1}^n \left|\tilde{x}^{(n)}_{st;ij} \right|^2 \leq 4 \sum_{s,t=1}^d \sum_{i,j=1}^n \left( |x_{st;ij}|^2 + \E |x_{st;ij}|^2 \right) $$
for $n$ sufficiently large.  Hence \eqref{eq:llnhshat} follows from \eqref{eq:llnhs}.  

We now prove \eqref{eq:lln2eta}; fix $s,t \in \{1,\ldots,d\}$.  By the law of large numbers, for any $M > 0$, we have a.s. 
$$ \limsup_{n \rightarrow \infty} \frac{1}{n^2} \sum_{i,j=1}^n |x_{st;ij}|^{2 + \eta} \indicator{|x_{st;ij}| > n^{\delta}|} \leq \E|\xi_{st}|^{2+\eta} \indicator{|\xi_{st}| > M}. $$
By the dominated convergence theorem, it follows that 
$$ \lim_{M \rightarrow \infty} \E|\xi_{st}|^{2 + \eta} \indicator{|\xi_{st}| > M} = 0. $$
We conclude that a.s.
$$ \lim_{n \rightarrow \infty} \frac{1}{n^2} \sum_{i,j=1}^n |x_{st;ij}|^{2 + \eta} \indicator{|x_{st;ij}| > n^{\delta}|} = 0. $$
Since $d$ is fixed, independent of $n$, the claim follows.  
\end{proof}

We let $L(F,G)$ denote the Levy distance between two distribution functions $F,G$.  That is,
\begin{equation} \label{eq:def:levy}
	L(F,G) := \inf\{ \eps > 0 : F(x - \eps) - \eps \leq G(x) \leq F(x + \eps) + \eps \text{ for all } x \in \mathbb{R} \}. 
\end{equation}

Convergence in Levy distance implies convergence in distribution \cite[Remark A.40]{BSbook}.  We will compare the ESD of $\hat{\BH}_n$ to the ESD of $\BH_n$ using the Levy metric.  

\begin{lemma} \label{lemma:truncn}
Let $\{\BX_n\}_{n \geq 1}$ be a sequence of random matrices that satisfies condition {\bf C0} with parameter $d \geq 2$ and atom variables $(\xi_{st})_{s,t=1}^d$, and assume \eqref{eq:m2eta} holds for some $\eta > 0$.  Let $\delta > 0$.  Then a.s. 
$$ \sup_{z \in \mathbb{C}} L(F^{\BH_n}, F^{\hat{\BH}_n}) = o(n^{-\delta \eta /3}). $$
\end{lemma}
\begin{proof}(of Lemma \ref{lemma:truncn})
We will apply \cite[Corollary A.41]{BSbook} to bound $L(F^{H_n}, F^{\tilde{H}_n})$ and $L(F^{\tilde{H}_n}, F^{\hat{H}_n})$ separately.  Thus, we have
\begin{align*}
	\sup_{z \in \mathbb{C}} n^{\delta \eta} L^3(F^{\BH_n}, F^{\tilde{\BH}_n}) &\leq \frac{n^{\delta \eta}}{n^2} \left \| \BX_n - \tilde{\BX}_n \right\|_2^2 \\
		&\leq 2 \frac{n^{\delta \eta}}{n^2} \sum_{s,t=1}^d \sum_{i,j=1}^n \left( \left|x_{st,ij}\right|^2 \indicator{|x_{st;ij}| > n^{\delta}} + \E |x_{st;ij}|^2\indicator{|x_{st;ij}| > n^{\delta}} \right) \\
		&\leq \frac{2}{n^2} \sum_{s,t=1}^d \sum_{i,j=1}^n \left( \left|x_{st,ij}\right|^{2+\eta} \indicator{|x_{st;ij}| > n^{\delta}} + \E |x_{st;ij}|^{2+\eta} \indicator{|x_{st;ij}| > n^{\delta}} \right). 
\end{align*}
We note that
$$ \frac{1}{n^2} \sum_{s,t=1}^d \sum_{i,j=1}^n \E |x_{st;ij}|^{2+\eta} \indicator{|x_{st;ij}| > n^{\delta}} = \sum_{s,t=1}^n \E |\xi_{st}|^{2+\eta} \indicator{|\xi_{st}| > n^{\delta}}, $$
and thus, by the dominated convergence theorem, we obtain
$$ \lim_{n \rightarrow \infty} \frac{1}{n^2} \sum_{s,t=1}^d \sum_{i,j=1}^n \E |x_{st;ij}|^{2+\eta} \indicator{|x_{st;ij}| > n^{\delta}} = 0. $$
Therefore, by Lemma \ref{lemma:lln}, we conclude that a.s.
$$ \lim_{n \rightarrow \infty} \sup_{z \in \mathbb{C}} n^{\delta \eta} L^3(F^{\BH_n}, F^{\tilde{\BH}_n}) = 0, $$
and hence a.s.
\begin{equation} \label{eq:supbd1}
	\sup_{z \in \mathbb{C}} L(F^{\BH_n}, F^{\tilde{\BH}_n}) = o(n^{-\delta \eta / 3}). 
\end{equation}

Applying \cite[Corollary A.41]{BSbook} again, we obtain
\begin{align*}
	\sup_{z \in \mathbb{C}} n^{\delta \eta} L^3(F^{\tilde{\BH}_n}, F^{\hat{\BH}_n}) &\leq \frac{n^{\delta \eta}}{n^2} \left\| \tilde{\BX}_n - \hat{\BX}_n \right\|_2^2 \\
		&\leq \frac{n^{\delta \eta}}{n^2} \sum_{s,t=1}^d \sum_{i,j=1}^n \left| \hat{x}_{st;ij}^{(n)} \right|^2 \left| 1 - \sqrt{d \var( \tilde{x}_{st;ij}^{(n)}) } \right|^2 \\
		&\leq \frac{n^{\delta \eta}} {n^2} \sum_{s,t=1}^d \sum_{i,j=1}^n \left| \hat{x}_{st;ij}^{(n)} \right|^2 \left| 1 - d \var(\tilde{x}_{st;ij}^{(n)}) \right|^2 \\
		&\leq d^2 \frac{n^{\delta \eta}} {n^2} \sum_{s,t=1}^d \sum_{i,j=1}^n \left| \hat{x}_{st;ij}^{(n)} \right|^2 \left| \frac{1}{d} - \var(\tilde{\xi}_{st}^{(n)}) \right|^2 \\
		&\leq 4d^2 \frac{m_{2+\eta}^2}{n^{\delta \eta}} \frac{1}{n^2} \left\| \hat{\BX}_n \right\|_2^2.
\end{align*}
Here the last inequality follows from Lemma \ref{lemma:abcdtruncn}.  By Lemma \ref{lemma:lln}, we have a.s.
$$ \lim_{n \rightarrow \infty} \frac{1}{n^{2+\delta \eta}} \left\| \hat{\BX}_n \right\|_2^2 = 0, $$
and we conclude that a.s.
\begin{equation} \label{eq:supbd2}
	\sup_{z \in \mathbb{C}} L(F^{\tilde{\BH}_n}, F^{\hat{\BH}_n}) = o(n^{-\delta \eta / 3}). 
\end{equation}
The claim now follows from \eqref{eq:supbd1}, \eqref{eq:supbd2}, and the triangle inequality for Levy distance.  
\end{proof}

\subsection{Cubic Relation}

We now consider the distribution of eigenvalues of $\BH_n$.  In fact, by Lemma \ref{lemma:truncn}, it will suffice to consider the eigenvalues of $\hat{\BH}_n$.  To this end, we will study the resolvent of $\hat{\BH}_n$ in Theorem \ref{thm:cubic} below.  Indeed, for $w \in \mathbb{C}$ with $\Im(w) > 0$, 
$$ \hat{m}_n(z,w) := \frac{1}{2dn} \tr \left( \hat{\BH}_n(z) - w \BI \right)^{-1} = \int_{\mathbb{R}} \frac{1}{x - w} \nu_{\frac{1}{\sqrt{n}}\hat{\BX}_n - z \BI}(dx) $$
is the Stieltjes transform of the measure $\nu_{\frac{1}{\sqrt{n}}\hat{\BX}_n - z\BI}$.  It follows from standard Stieltjes transform techniques (e.g. \cite[Theorem B.9]{BSbook}) that computing the limiting ESD of $\hat{\BH}_n(z)$ is equivalent to computing the limit of $\hat{m}_n(z,w)$ for all $w \in \C$ with $\Im(w) > 0$.

As is standard in random matrix theory, we will not compute $\hat{m}_n$ explicitly.  Instead we will derive a fixed point equation.  Indeed, we will show
\begin{equation} \label{eq:cubicrelationn}
	\hat{m}_n(z,w) + \frac{ \hat{m}_n(z,w) + w}{ (\hat{m}_n(z,w) + w)^2 - |z|^2} = o(1) 
\end{equation}
for $z,w\in \C$ with $\Im(w) > 0$.  We will then conclude that $m_n(z,w)$ converges to a limit, which we denote by $m(z,w)$.  It follows that $m(z,w)$ satisfies the equation
\begin{equation} \label{eq:cubicrelation}
	m(z,w) + \frac{ m(z,w) + w}{ (m(z,w) + w)^2 - |z|^2} = 0. 
\end{equation}
From \eqref{eq:cubicrelation} we will deduce the limiting ESD of $\hat{\BH}_n$.  Equation \eqref{eq:cubicrelation} has appeared previously in \cite{BC,GTcirc} and in a slightly different form in \cite[Chapter 11]{BSbook}.  We refer to equation \eqref{eq:cubicrelation} as a cubic relation since it can be rewritten as the cubic polynomial equation
$$ m(z,w)^3 + 2w m(z,w)^2 + (w^2 - |z|^2 + 1) m(z,w) + w = 0. $$

In this subsection we will show $\hat{m}_n$ satisfies \eqref{eq:cubicrelationn}.  We begin with the following concentration result for bilinear forms from \cite{OR}.  

\begin{lemma}[Lemma 3.10 of \cite{OR}] \label{lemma:quadp}
Let $(x,y)$ be a random vector in $\mathbb{C}^2$ where $x,y$ both have mean zero, unit variance, and satisfy 
\begin{itemize}
\item $\max\{|x|,|y|\} \leq L$ a.s.,
\item $\E[\bar{x} y] = \rho$.
\end{itemize}
Let $(x_1, y_1), (x_2, y_2), \ldots, (x_n, y_n)$ be iid copies of $(x,y)$, and set $X = (x_1, x_2, \ldots, x_n)^\mathrm{T}$ and $Y=(y_1, y_2, \ldots, y_n)^\mathrm{T}$.  Let $\BB$ be a $n \times n$ random matrix, independent of $X$ and $Y$, which satisfies $\|\BB\| \leq n^{1/4}$ a.s.  Then, for any $p \geq 2$, 
$$ \Prob \left( \left| \frac{1}{n} X^\ast \BB Y - \frac{\rho}{n} \tr \BB \right| > n^{-1/8} \right) =O _p \left(\frac{ L^{2p}} {n^{p/8}} \right). $$
\end{lemma}

We formally establish \eqref{eq:cubicrelationn} in the following theorem.  

\begin{theorem} \label{thm:cubic}
Let $\{\BX_n\}_{n \geq 1}$ be a sequence of random matrices that satisfies condition {\bf C0} with parameter $d \geq 2$ and atom variables $(\xi_{st})_{s,t=1}^d$, and assume \eqref{eq:m2eta} holds for some $\eta > 0$.  
Let $0 < \delta < 1/100$.  Consider the truncated random matrices $\{\hat{\BX}_n\}_{n \geq 1}$ and $\{\hat{\BH}_n(z)\}_{n \geq 1}$.  For $z,w \in \mathbb{C}$ with $\Im(w) > 0$, define
$$ \hat{\BG}_n(z,w) := \left( \hat{\BH}_n(z) - w\BI \right)^{-1} \quad\text{and}\quad \hat{m}_n(z,w) := \frac{1}{2dn} \tr \hat{\BG}_n(z,w). $$
Let $M, \beta > 0$.  Then, for $v_n := \max \left\{ n^{-\eta \delta/100}, n^{-1/100} \right\}$, a.s.
$$ \sup_{|z| \leq M} \sup_{|w| \leq \beta, \Im(w) \geq v_n} \left| \hat{m}_n(z,w) + \frac{ \hat{m}_n(z,w) + w}{ (\hat{m}_n(z,w) + w)^2 - |z|^2} \right| = O_{M,\beta,d,m_{2+\eta}}(v_n^5). $$
\end{theorem}

In order to prove Theorem \ref{thm:cubic}, we will need the following deterministic lemmas.  

\begin{lemma} \label{lemma:Rdet}
Let $\BR$ be the $2n \times 2n$ block matrix given by
$$ \BR = \begin{bmatrix} -w \BI & \BB \\ \BB^\ast & - w \BI \end{bmatrix}^{-1} = \begin{bmatrix} \BR_1 & \BR_2 \\ \BR_3 & \BR_4 \end{bmatrix}, $$
where $\BB, \BR_1,\BR_2,\BR_3,\BR_4$ are $n \times n$ matrices.  Then $\tr \BR_1 = \tr \BR_4$ for any $w \in \mathbb{C}$ with $\Im(w) > 0$.  
\end{lemma}
\begin{proof}(of Lemma \ref{lemma:Rdet})
We first note that 
$$ \begin{bmatrix} -w \BI & \BB \\ \BB^\ast & - w \BI \end{bmatrix} = \begin{bmatrix} \Bzero & \BB \\ \BB^\ast & \Bzero \end{bmatrix} - w \BI $$
is invertible for any $w \in \mathbb{C}$ with $\Im(w) > 0$.  Let $\sigma_1, \sigma_2, \ldots, \sigma_n \geq 0$ denote the singular values of $\BB$.  Then $-w\BI + w^{-1} \BB \BB^\ast$ has eigenvalues $-w + w^{-1}\sigma_i^2$ for $i=1,2,\ldots,n$.  In particular
$$ \Im \left( -w + \frac{\sigma_i^2}{w} \right) = - \Im(w) - \Im(w) \frac{\sigma_i^2}{|w|^2} < 0 $$
for $\Im(w) > 0$.  Thus $-w\BI + w^{-1}\BB \BB^\ast$ is invertible.  Similarly, $-w\BI + w^{-1} \BB^\ast \BB$ has the same eigenvalues and is also invertible.  By the Schur complement \cite[Section 0.7.3]{HJ2}, 
\begin{align*}
	\BR_1 &= \left( -w\BI + w^{-1} \BB \BB^\ast \right)^{-1}, \\
	\BR_4 &= \left( -w\BI + w^{-1} \BB^\ast \BB \right)^{-1}.
\end{align*}
Since $\BR_1$ and $\BR_4$ have the same eigenvalues, $\tr \BR_1 = \tr \BR_4$.  
\end{proof}

We introduce $\eps$-nets as a convenient way to discretize a compact set.  Let $\eps > 0$.  A set $X$ is an $\eps$-net of a set $Y$ if for any $y \in Y$, there exists $x \in X$ such that $\|x-y\| \leq \eps$.  In order to prove Theorem \ref{thm:cubic}, we will need the following well-known estimate for the maximum size of an $\eps$-net.  

\begin{lemma}[Lemma 3.11 of \cite{OR}] \label{lemma:epsnet}
The set $\{w \in \mathbb{C} : |w| \leq \beta, \Im(w) \geq \alpha \}$ admits an $\eps$-net of size at most
$$ \left( 1 + \frac{2\beta}{\eps} \right)^2. $$
\end{lemma}

We will also take advantage of the following facts, which can be found in \cite{HJ,HJ2}.  Let $\BB$ be a $n \times n$ matrix with singular values $\sigma_1(\BB) \geq \cdots \geq \sigma_n(\BB) \geq 0$.  Then the $2n \times 2n$ matrix 
\begin{equation} \label{eq:matrixwbbw}
	\begin{bmatrix} -w \BI & \BB \\ \BB^\ast & -w \BI \end{bmatrix} = \begin{bmatrix} 0 & \BB \\ \BB^\ast & 0 \end{bmatrix} - w \BI 
\end{equation}
has an orthonormal basis of eigenvectors with eigenvalues $\pm \sigma_1(\BB) - w, \ldots, \pm \sigma_n(\BB) - w$.  Thus, if $\Im(w) > 0$, the matrix \eqref{eq:matrixwbbw} is invertible.  Let 
$$ \BG := \begin{bmatrix} -w \BI & \BB \\ \BB^\ast & -w \BI \end{bmatrix}^{-1}. $$
Then, for $\Im(w) > 0$, we have
\begin{equation} \label{eq:Gnormbnd}
	\| \BG \| = \max_{1 \leq i \leq n} \frac{1}{ |\pm \sigma_i(\BB) - w|} \leq \frac{1}{\Im(w)}. 
\end{equation}

We will make use of the following identity: for any invertible $n \times n$ matrices $\BA$ and $\BB$,
\begin{equation} \label{eq:generalresolventid}
	\BA^{-1} - \BB^{-1} = \BA^{-1} (\BB - \BA) \BB^{-1}. 
\end{equation}

A special case of \eqref{eq:generalresolventid} is the resolvent identity (also known as Hilbert's identity): for any Hermitian $n \times n$ matrix $\BA$,
\begin{equation} \label{eq:resolventid}
	(\BA - w \BI)^{-1} - (\BA - w' \BI)^{-1} = (w-w') (\BA - w \BI)^{-1} (\BA - w' \BI)^{-1}
\end{equation}
for all $w,w' \in \C$ with $\Im(w), \Im(w') > 0$.  

We are now ready to prove Theorem \ref{thm:cubic}.  

\begin{proof}(of Theorem \ref{thm:cubic})
Fix $M, \beta > 0$.  For the remainder of the proof, the implicit constants in our asymptotic notation (such as $O,o,\Omega, \ll$) depend only on the constants $M, \beta$, $d$, and $m_{2+\eta}$; for simplicity, we no longer include these subscripts in our notation. 

For notational convenience, we will write $\BX_n$ instead of $\hat{\BX}_n$.  That is, we let $\{\BX_n\}_{n \geq 1}$ denote the sequence of truncated matrices.  Similarly, we write $\BX_{n,st}$ instead of $\hat{\BX}_{n,st}$ for $s,t \in \{1,\ldots,d\}$.  We define the $2dn \times 2dn$ matrix
$$ \BG_n(z,w) := \begin{bmatrix} -w \BI & \frac{1}{\sqrt{n}} \BX_n - z \BI \\ \frac{1}{\sqrt{n}} \BX_n^\ast - \bar{z} \BI & - w \BI \end{bmatrix}^{-1}. $$
We will often drop the dependence on $z,w$ and simply write $\BG_n$ instead of $\BG_n(z,w)$.  We write $\BG_n = (\BG_{n,st})_{s,t=1}^{2d}$ where each $\BG_{n,st}$ is a $n \times n$ matrix.  Then $\BG_{n,st}(i,j)$ denotes the $(i,j)$-entry of $\BG_{n,st}$.  We define $m_{n,st}(z,w) := \frac{1}{n} \tr \BG_{n,st}$ for $s,t \in \{1,\ldots,2d\}$ and $m_n(z,w) := \frac{1}{2dn} \tr \BG_n$.  We will often drop the dependence on $z,w$ and simply write $m_n$ and $m_{n,st}$ instead of $m_n(z,w)$ and $m_{n,st}(z,w)$.  

Let $1 \leq k \leq n$.  We let $\Br_k(\BX_{n,st})$ denote the $k$-th row of $\BX_{n,st}$ with the $k$-th entry removed.  Similarly, we let $\Bc_k(\BX_{n,st})$ denote the $k$-th column of $\BX_{n,st}$ with the $k$-th entry removed.  We let $\BX_{n,st}^{(k)}$ be the $(n-1) \times (n-1)$ matrix constructed from $\BX_{n,st}$ by removing the $k$-th column and $k$-th row.  We let $\BX_{n}^{(k)}$ be the $d(n-1) \times d(n-1)$ block matrix given by $\BX_{n}^{(k)} := \left(\BX_{n,st}^{(k)}\right)_{s,t=1}^d$.  Define the $2d(n-1) \times 2d(n-1)$ matrix
$$ \BG^{(k)}_n(z,w) := \begin{bmatrix} -w \BI & \frac{1}{\sqrt{n}} \BX_{n}^{(k)} - z \BI \\ \frac{1}{\sqrt{n}} {\BX_{n}^{(k)}}^\ast - \bar{z} \BI & -w \BI \end{bmatrix}^{-1}. $$
We will often drop the dependence on $z,w$ and simply write $\BG^{(k)}_n$.  We again write $\BG_n^{(k)} = \left(\BG_{n,st}^{(k)}\right)_{s,t=1}^{2d}$ where each $\BG_{n,st}^{(k)}$ is a $(n-1) \times (n-1)$ matrix.  We let $m_{n,st}^{(k)}(z,w) := \frac{1}{n} \tr \BG_{n,st}^{(k)}$ for $s,t \in \{1,\ldots,2d\}$ and $m_n^{(k)}(z,w) := \frac{1}{2dn} \tr \BG_n^{(k)}$.  We will often drop the dependence on $z,w$ and write $m_n^{(k)}$ and $m_{n,st}^{(k)}$ instead of $m_n^{(k)}(z,w)$ and $m_{n,st}^{(k)}(z,w)$.

From \eqref{eq:Gnormbnd}, for $w \in \mathbb{C}$ with $\Im(w) \geq v_n$, we have the deterministic bounds $\|\BG_n(z,w)\| \leq v_n^{-1}$, $| m_n(z,w) | \leq v_n^{-1}$, and $| m_{n,st}(z,w) | \leq v_n^{-1}$.  By Cauchy's interlacing theorem \cite[Theorem 4.3.8]{HJ2} (or alternatively \cite[(A.1.12)]{BSbook}), we have the deterministic bound
\begin{equation} \label{eq:cauchyinter}
	\sup_{1 \leq k \leq n} \sup_{|z| \leq M} \sup_{|w| \leq \beta, \Im(w) \geq v_n} \left| m_n^{(k)}(z,w) - m_n(z,w) \right| = O \left( \frac{1}{v_n n} \right). 
\end{equation}
By Lemma \ref{lemma:Rdet}, we have
$$ \sum_{s=1}^d m_{n,ss}(z,w) = \sum_{s=d+1}^{2d} m_{n,ss}(z,w) $$
and
$$ \sum_{s=1}^d m_{n,ss}^{(k)}(z,w) = \sum_{s=d+1}^{2d} m_{n,ss}^{(k)}(z,w) $$
for any $1 \leq k \leq n$.  Thus, from \eqref{eq:cauchyinter}, we find
\begin{equation} \label{eq:sup12-e}
	\sup_{1 \leq k \leq n} \sup_{|z| \leq M} \sup_{|w| \leq \beta, \Im(w) \geq v_n} \left| \frac{1}{d} \sum_{s=1}^d m_{n,ss}^{(k)}(z,w) - m_n(z,w) \right| = O\left( \frac{1}{n v_n} \right) 
\end{equation}
and
\begin{equation} \label{eq:sup34-e}
	\sup_{1 \leq k \leq n} \sup_{|z| \leq M} \sup_{|w| \leq \beta, \Im(w) \geq v_n} \left| \frac{1}{d} \sum_{s=d+1}^{2d} m_{n,ss}^{(k)}(z,w) - m_n(z,w) \right| = O \left( \frac{1}{n v_n} \right). 
\end{equation}
Fix $1 \leq k \leq n$ and $z \in \mathbb{C}$ with $|z| \leq M$.  Fix $w \in \mathbb{C}$ with $|w| \leq \beta$ and $\Im(w) \geq v_n$.  Let $\BQ_k$ be the $2d \times 2d$ matrix given by $\BQ_k := ( \BG_{n,st}(k,k) )_{s,t=1}^{2d}$.  By the Schur complement \cite[Section 0.7.3]{HJ2}, 
$$ \BQ_k^{-1} = \begin{bmatrix} -w \BI & \left( \frac{1}{\sqrt{n}} x_{st;kk} - z \delta_{s,t} \right)_{s,t=1}^d \\ \left( \frac{1}{\sqrt{n}} \bar{x}_{st;kk} - \bar{z} \delta_{s,t} \right)_{s,t=1}^d & - w \BI \end{bmatrix} - \frac{1}{n} \BR_k \BG^{(k)}_n \BR_k^\ast ,$$
where $\delta_{s,t}$ is the Kronecker delta and
$$ \BR_k := \begin{bmatrix} \Bzero & \left( \Br_k( \BX_{n,st}) \right)_{s,t=1}^d \\ \left( \Bc_k( \BX_{n,ts})^\ast \right)_{s,t=1}^d & \Bzero \end{bmatrix}. $$
By the truncation assumption and Lemma \ref{lemma:abcdtruncn}, we have a.s.
$$ \max_{1 \leq k \leq n} \max_{1 \leq s,t \leq d} \frac{|x_{st;kk}|}{\sqrt{n}} \ll \frac{n^{\delta}}{\sqrt{n}} \leq \frac{1}{n^{2/5}} $$
for $n$ sufficiently large.  

We observe that $\BR_k$ and $\BG^{(k)}_n$ are independent random matrices.  By expanding the product, we note that the entries of $\BR_k \BG^{(k)}_n \BR_k^\ast$ are linear combinations of bilinear forms.  We will apply Lemma \ref{lemma:quadp} to control each bilinear form.  Applying the bound $\|\BG_n(z,w) \| \leq v_n^{-1}$ and Lemma \ref{lemma:quadp}, we obtain
\begin{equation} \label{eq:schurnormbnd}
	\left\| \frac{1}{n} \BR_k \BG^{(k)}_n \BR_k^\ast - \begin{bmatrix} \left( \frac{1}{d} \sum_{s=d+1}^{2d} m_{n,ss}^{(k)} \right) \BI_d & \Bzero \\ \Bzero & \left( \frac{1}{d} \sum_{s=1}^d m_{n,ss}^{(k)} \right) \BI_d \end{bmatrix}  \right\| \ll n^{-1/8} + \frac{1}{ n^{\delta \eta/2} v_n} 
\end{equation}
with probability $1 - O(n^{-100})$.  Here we obtain the bound on the spectral norm by bounding each entry individually and noting that 
$$ \| \BB \| \leq \| \BB \|_2 \leq 2d \max_{i,j} |\BB_{ij}| $$ 
for any $2d \times 2d$ matrix $\BB$ (recall that the matrices above are $2d \times 2d$).  The bound \eqref{eq:schurnormbnd} holds with probability $1 - O(n^{-100})$ by taking $p$ sufficiently large in Lemma \ref{lemma:quadp}.  The factor $1/d$ appears because the entries of $\BR_k$ have variance $1/d$.  We also used \eqref{truncn:corr} from Lemma \ref{lemma:abcdtruncn} and the deterministic bound $\sup_{s,t \in \{1,\ldots,2d\}} |  m_{n,st}^{(k)} | \leq \| \BG^{(k)}_n \| \leq v_n^{-1}$.  

By \eqref{eq:sup12-e}, \eqref{eq:sup34-e}, and the union bound over $1 \leq k \leq n$, we obtain
\begin{equation} \label{eq:qkmn-e}
	\sup_{1 \leq k \leq n} \| \BQ_k^{-1} - \BM_n  \| \ll n^{-1/8} + \frac{1}{ n^{\delta \eta/2} v_n} + \frac{1}{n v_n} 
\end{equation}
with probability $1 - O(n^{-99})$, where 
$$ \BM_n := \begin{bmatrix} -(m_n(z,w) + w) \BI_d & -z \BI_d \\ -\bar{z} \BI_d & -(m_n(z,w) + w) \BI_d \end{bmatrix}. $$
We note that
\begin{equation} \label{eq:immw-e}
	|m_n(z,w) + w| \geq \Im(m_n(z,w)+w) \geq \Im(w) \geq v_n > 0. 
\end{equation}
It follows that $M_n$ is invertible and the inverse is given (in block form) by
$$ \BM_n^{-1} = \begin{bmatrix} \BM_{n,1} & \BM_{n,2} \\ \BM_{n,3} & \BM_{n,4} \end{bmatrix}, $$
where $\BM_{n,1}, \BM_{n,2}, \BM_{n,3}, \BM_{n,4}$ are $d \times d$ matrices with
\begin{align*}
	\BM_{n,1} = \BM_{n,4} &:= - \frac{ m_n(z,w)+w } { (m_n(z,w)+w)^2 - |z|^2 } \BI, \\
	\BM_{n,2} &:= -\frac{z}{m_n(z,w)+w} \BM_{n,1}, \\
	\BM_{n,3} &:= -\frac{\bar{z}}{m_n(z,w)+w} \BM_{n,1}.
\end{align*}

Using \eqref{eq:immw-e}, we obtain
\begin{align}
	\inf_{\Im(w) \geq v_n} | (m_n(z,w) + w)^2 - |z|^2 | \geq \inf_{\Im(w) \geq v_n} \left|m_n(z,w) + w - |z| \right| |m_n(z,w) + w + |z|| \geq v_n^2  \label{eq:mnzbnd}
\end{align}
and hence
$$ \sup_{|z| \leq M} \sup_{|w| \leq \beta, \Im(w) \geq v_n} \| \BM_n^{-1} \| = O(v_n^{-4}). $$
Since $\sup_{\Im(w) \geq v_n} \|\BG_n(z,w)\| \leq v_n^{-1}$, we obtain 
$$ \sup_{1 \leq k \leq n} \sup_{|w| \leq \beta, \Im(w) \geq v_n} \|\BQ_k\| = O(v_n^{-1}). $$
Therefore, by \eqref{eq:qkmn-e}, we have
\begin{align*}
	\sup_{1 \leq k \leq n} \| \BQ_k - \BM_n^{-1} \| &\leq \sup_{1 \leq k \leq n} \| \BQ_k (\BQ_k^{-1} - \BM_n ) \BM_n^{-1} \| \\
		&\leq \sup_{1 \leq k \leq n} \|\BQ_k \| \|\BQ_k^{-1} - \BM_n \| \|\BM_n^{-1} \| \\
		&\ll \frac{1}{n^{1/8} v_n^6} + \frac{1}{n^{\delta \eta/2} v_n^6} 
\end{align*}
with probability at least $1 - O(n^{-99})$.  Since $m_n(z,w)$ is the normalized sum of the diagonal elements of $\BG_n$, we now consider the diagonal elements of $\BQ_k$ and $\BM_n$; from the above estimate, we conclude that
\begin{equation} \label{eq:fixcubic}
	\left| m_n(z,w) +  \frac{ m_n(z,w)+w } { (m_n(z,w)+w)^2 - |z|^2 } \right| \ll v_n^5 
\end{equation}
with probability $1 - O(n^{-99})$.  Here we used the fact that
$$ \frac{1}{n^{1/8} v_n^6} + \frac{1}{n^{\delta \eta/2} v_n^6} \leq 2 v_n^5 $$
by definition of $v_n$.  

We now use an $\eps$-net argument to extend \eqref{eq:fixcubic} to all $|z| \leq M$ and $|w| \leq \beta$ with $\Im(w) \geq v_n$.  Since $ \sup_{\Im(w) \geq v_n } \| \BG_n(z,w) \| \leq v_n^{-1}$, we apply \eqref{eq:generalresolventid} and the resolvent identity \eqref{eq:resolventid} to obtain the deterministic bounds
$$ |m_n(z,w) - m_n(z,w')| \leq \frac{|w-w'|}{v_n^2} $$
and
$$ |m_n(z,w) - m_n(z',w)| \leq \frac{|z-z'|}{v_n^2} $$
for all $z,z' \in \mathbb{C}$ and $w,w' \in \mathbb{C}$ with $\Im(w), \Im(w') \geq v_n$.  Applying \eqref{eq:mnzbnd} and the triangle inequality, we obtain
\begin{align*}
	&\left| \frac{ m_n(z,w) + w}{(m_n(z,w) + w)^2 - |z|^2} - \frac{m_n(z,w') + w'}{(m_n(z,w') + w')^2 - |z|^2} \right| \\
	& \qquad \ll \frac{1}{v_n^2} \left| m_n(z,w) + w - m_n(z,w') - w' \right| \\
	& \qquad \qquad + \frac{1}{v_n} \left| \frac{1}{ (m_n(z,w) + w)^2 - |z|^2} - \frac{1}{(m_n(z,w') + w')^2 - |z|^2} \right| \\
	& \qquad \ll \frac{|w-w'|}{v_n^4} + \frac{1}{v_n^5} \left| (m_n(z,w) + w)^2 - (m_n(z,w') + w')^2 \right| \\
	& \qquad \ll \frac{|w-w'|}{v_n^8}
\end{align*}
for all $|z| \leq M$ and $|w|, |w'| \leq \beta$ with $\Im(w), \Im(w') \geq v_n$.  Similarly, 
$$ \left| \frac{ m_n(z,w) + w}{(m_n(z,w) + w)^2 - |z|^2} - \frac{m_n(z',w) + w}{(m_n(z',w) + w)^2 - |z'|^2} \right| \ll \frac{|z-z'|}{v_n^8} $$
for all $|z|, |z'| \leq M$ and $|w| \leq \beta$ with $\Im(w) \geq v_n$. 

We now apply an $\eps$-net argument with $\eps = v_n^{13}$ to the sets $\{z \in \mathbb{C} : |z| \leq M \}$ and $\{ w \in \mathbb{C} : |w| \leq \beta, \Im(w) \geq v_n\}$.  Let $\mathcal{N}_1$ and $\mathcal{N}_2$ denote the respective $\eps$-nets of the two sets.  By Lemma \ref{lemma:epsnet}, 
$$ |\mathcal{N}_1| + |\mathcal{N}_2 |\ll v_n^{-26} \leq n^{1/2}. $$
Therefore, by a standard $\eps$-net argument and the union bound, we conclude that
$$ \sup_{|z| \leq M} \sup_{|w| \leq \beta, \Im(w) \geq v_n} \left| m_n(z,w) + \frac{ m_n(z,w) + w}{(m_n(z,w) + w)^2 - |z|^2} \right| = O(v_n^5) $$
with probability (say) $1 - O(n^{-2})$.  The claim now follows from an application of the Borel-Cantelli lemma.  
\end{proof}

\subsection{Proof of Theorem \ref{thm:main}}

This subsection is devoted to the proof of Theorem \ref{thm:main}.  With Lemma \ref{lemma:truncn} and Theorem \ref{thm:cubic} in hand, the proof of Theorem \ref{thm:main} will follow from a standard (and somewhat technical) argument; see \cite[Chapter 11]{BSbook}, \cite{TVcirc}, and references therein.  We detail the argument below.  

Recall the definition of the functions $g_{\BM}$ and $g$ given in \eqref{eq:def:gMn} and \eqref{eq:def:g}.  By \cite[Chapter 11]{BSbook} (see also \cite{BC} and \cite[Section 3]{GTcirc}), for each $z \in \mathbb{C}$, there exists a probability measure $\nu_z$ on the real line such that 
$$ g(s,t) = \frac{\partial}{\partial s} \int_{0}^\infty \log|x|^2 d\nu_z(dx), $$
where $z = s + \sqrt{-1}t$.  

Assume $\{\BX_n\}_{n \geq 1}$ and $\{\BN_n\}_{n \geq 1}$ satisfy the assumptions of Theorem \ref{thm:main}.  By Lemma \ref{lemma:girko} and \cite[Lemma 11.5]{BSbook}, in order to prove Theorem \ref{thm:main}, it suffices to show that a.s.
$$ \iint \left[g_{\frac{1}{\sqrt{n}}(\BX_n + \BN_n)}(s,t) - g(s,t)\right] e^{\sqrt{-1} us + \sqrt{-1} vt} dt ds \longrightarrow 0 $$
as $n \rightarrow \infty$.  

By the triangle inequality and Lemma \ref{lemma:lln}, we have that a.s.
\begin{equation} \label{eq:hsbnd}
	\frac{1}{n} \left \| \frac{1}{\sqrt{n}} (\BX_n + \BN_n) \right\|_2^2 = O_d(1). 
\end{equation}  

Let $A > 0$.  Define
$$ T:= \left\{(s,t) : |s| \leq A, |t| \leq A^3 \right\}. $$
By \cite[Lemma 11.7]{BSbook} and \eqref{eq:hsbnd}, in order to prove Theorem \ref{thm:main} it suffices to show that for each fixed $A>0$ a.s.
$$ \iint_T \left[g_{\frac{1}{\sqrt{n}}(\BX_n + \BN_n)}(s,t) - g(s,t)\right] e^{\sqrt{-1} us + \sqrt{-1} vt} dt ds \longrightarrow 0 $$
as $n \rightarrow \infty$.  

Let $\eps_n := n^{-B}$ for some sufficiently large $B>0$ (independent of $n$) to be chosen later.  Following the integration by parts argument from \cite[Section 11.7]{BSbook}, it suffices to show that a.s.
\begin{equation} \label{eq:showintdiff}
	 \limsup_{n \rightarrow \infty} \iint_{T} \left| \int_{\eps_n}^{\infty} \log |x|^2 \left(\nu_{\frac{1}{\sqrt{n}} (\BX_n + \BN_n) - z\BI}(dx) - \nu_z(dx) \right) \right| dt ds = 0 
\end{equation}
and
\begin{equation} \label{eq:showintsingle}
	\limsup_{n \rightarrow \infty} \iint_{T} \left| \int_{0}^{\eps_n} \log |x|^2 \nu_{\frac{1}{\sqrt{n}} (\BX_n + \BN_n) - z\BI}(dx) \right| dt ds = 0,
\end{equation}
and similarly with the two-dimensional integral on $T$ replaced by one-dimensional integrals on the boundary of $T$.  We shall only estimate the two-dimensional integrals, as the treatment of the one-dimensional integrals are similar.  

We prove \eqref{eq:showintdiff} first.  By \eqref{eq:hsbnd}, it follows that $\nu_{\frac{1}{\sqrt{n}} (\BX_n + \BN_n) - z \BI}$ is supported on $[-n^{50}, n^{50}]$ a.s.  Thus, it suffices to show that a.s.
$$ \limsup_{n \rightarrow \infty} \iint_{T} \left| \int_{\eps_n}^{n^{50}} \log |x|^2 \left(\nu_{\frac{1}{\sqrt{n}} (\BX_n + \BN_n) - z\BI}(dx) - \nu_z(dx) \right) \right| dt ds = 0. $$
By definition of $\eps_n$, it suffices to show that a.s.
\begin{equation} \label{eq:nalphanorm}
	\limsup_{n \rightarrow \infty} (\log n) \sup_{z \in T} \sup_{x \in \mathbb{R}} \left| \nu_{\frac{1}{\sqrt{n}} (\BX_n + \BN_n) - z\BI}( (-\infty, x)) - \nu_z( (-\infty, x)) \right| = 0.
\end{equation}
\eqref{eq:nalphanorm} will follow from Lemma \ref{lemma:rate} below.

We now prove \eqref{eq:showintsingle}.  By Theorem \ref{thm:least-sing-value} (and the Borel-Cantelli lemma), for some sufficiently large $B > 0$, we have the following:
\begin{equation} \label{eq:svzero}
	\text{for a.e. } z \in T, \text{ a.s. }\quad \lim_{n \rightarrow \infty} \int_{0}^{\eps_n} \log |x|^2 \nu_{\frac{1}{\sqrt{n}} (\BX_n + \BN_n) - z\BI}(dx) = 0.
\end{equation}

We now observe that it is possible to switch the quantifiers ``a.e.'' on $z$ and ``a.s.'' on $\omega$ in \eqref{eq:svzero} using the arguments from \cite[Section 4]{BC} and Fubini's theorem, where $\omega$ denotes an element of the sample space.  Thus, we have
\begin{equation} \label{eq:svzeroreverse}
	\text{a.s., } \text{for a.e. } z \in T, \quad \lim_{n \rightarrow \infty} \int_{0}^{\eps_n} \log |x|^2 \nu_{\frac{1}{\sqrt{n}} (\BX_n + \BN_n) - z\BI}(dx) = 0.
\end{equation}

Using the $L^2$-norm argument in \cite[Section 12]{TVcirc}, it follows that a.s.
\begin{equation} \label{eq:unifbnd}
	\left( \iint_{T} \left| \int_{0}^{\eps_n} \log |x|^2 \nu_{\frac{1}{\sqrt{n}} (\BX_n + \BN_n) - z \BI}(dx) \right|^2 dt ds \right)^{1/2} 
\end{equation}
is bounded uniformly in $n$, and hence the sequence of functions $\int_{0}^{\eps_n} \log |x|^2 \nu_{\frac{1}{\sqrt{n}} (\BX_n + \BN_n) - z\BI}(dx)$ is a.s. uniformly integrable on $T$.  Let $L > 1$ be a large parameter and define $T_{L,n}$ to be the set of all $z \in T$ such that $\left| \int_{0} ^{\eps_n} \log |x|^2 \nu_{\frac{1}{\sqrt{n}} (\BX_n + \BN_n) - z \BI}(dx)\right| \leq L$.  By \eqref{eq:svzeroreverse} and the dominated convergence theorem, we have a.s.
$$ \lim_{n \rightarrow \infty} \iint_{T_{L,n}} \left| \int_{0}^{\eps_n} \log |x|^2 \nu_{\frac{1}{\sqrt{n}}(\BX_n + \BN_n) - z\BI}(dx) \right| dt ds = 0. $$
On the other hand, from the uniform boundedness of \eqref{eq:unifbnd}, we obtain a.s.
$$ \limsup_{n \rightarrow \infty} \iint_{T \setminus T_{L,n}} \left| \int_{0}^{\eps_n} \log |x|^2 \nu_{\frac{1}{\sqrt{n}}(\BX_n + \BN_n) - z\BI}(dx) \right| dt ds \ll \frac{1}{L}.  $$
Combining the bounds above and taking $L \rightarrow \infty$ yields \eqref{eq:showintsingle}.  

It remains to establish the following lemma.  

\begin{lemma} \label{lemma:rate}
Let $\{\BX_n\}_{n \geq 1}$ be a sequence of random matrices that satisfies condition {\bf C0} with parameter $d \geq 2$ and atom variables $(\xi_{st})_{s,t=1}^d$, and assume \eqref{eq:m2eta} holds for some $\eta > 0$.  For each $n \geq 1$, let $\BN_n$ is a $dn \times dn$ matrix such that $\rank(\BN_n) = O(n^{1-\eps})$ for some $\eps > 0$.  Then there exists $\alpha > 0$ such that a.s.
$$ \sup_{|z| \leq M} \left\| \nu_{\frac{1}{\sqrt{n}}(\BX_n + \BN_n) - z\BI} - \nu_z \right\| = O_{M,m_{2+\eta},d}(n^{-\alpha}), $$
where $\| \nu - \mu \| := \sup_{x \in \mathbb{R}} | \nu((-\infty, x)) - \mu((-\infty, x))|$ for any two probability measures $\nu, \mu$ on the real line.  
\end{lemma}
\begin{proof}(of Lemma \ref{lemma:rate})
The proof of Lemma \ref{lemma:rate} is based on the arguments from \cite[Lemma 64]{TVuniv}.  By \cite[Theorem A.43]{BSbook},
$$ \sup_{z \in \mathbb{C}} \left\| \nu_{\frac{1}{\sqrt{n}} (\BX_n + \BN_N) - z\BI} - \nu_{\frac{1}{\sqrt{n}} \BX_n - z\BI} \right\| = O(n^{-\eps}). $$
Thus, by the triangle inequality, it suffices to show that a.s.
$$ \sup_{|z| \leq M} \left\| \nu_{\frac{1}{\sqrt{n}} \BX_n - z\BI} - \nu_z \right\| = O_{M,m_{2+\eta}}(n^{-\alpha}) $$
for some $\alpha > 0$.  

From \cite[Remark 3.1]{GTcirc}, it follows that for each $z \in \mathbb{C}$, $\nu_z$ has density $\rho_z$ with 
$$ \sup_{z \in \mathbb{C}} \sup_{x \in \mathbb{R}} |\rho_z(x)| \leq 1. $$
By \cite[Lemma B.18]{BSbook}, it suffices to show that a.s.
$$ \sup_{|z| \leq M} L\left(F^{\BH_n(z)}, F_z \right) = O_{M,m_{2+\eta}}(n^{-\alpha}), $$
where $F_z$ is the cumulative distribution function of $\nu_z$.  We remind the reader that $L(F,G)$ denotes the Levy distance, defined in \eqref{eq:def:levy}, between the distribution functions $F$ and $G$.  

By Lemma \ref{lemma:truncn}, it suffices to show that a.s.
\begin{equation} \label{eq:showsupzmnorm}
	\sup_{|z| \leq M} \left\| \nu_{\frac{1}{\sqrt{n}} \hat{\BX}_n - z \BI} - \nu_z \right\| = O_{M,m_{2+\eta}}(n^{-\alpha}), 
\end{equation}
where $\{\hat{\BX}_n\}_{n \geq 1}$ is the sequence of truncated matrices from Lemma \ref{lemma:truncn} for some $0 < \delta < 1/100$.  Let $m(z,w)$ denote the Stieltjes transform of $\nu_z$.  That is,
$$ m(z,w) := \int \frac{1}{x - w} d \nu_z(dx) = \int \frac{ \rho_z(x) dx}{x - w}. $$
From \cite[Section 3]{GTcirc}, it follows that $m(z,w)$ is a solution of
$$ m(z,w) + \frac{m(z,w) + w}{(m(z,w) + w)^2 - |z|^2} = 0 $$
analytic in the upper-half plane $\{w \in \mathbb{C} : \Im(w) > 0 \}$.  

By \cite[Remark 3.1]{GTcirc}, we choose $\beta > 100$ sufficiently large (depending only on $M$) such that $\rho_z$ is supported inside the interval $[-\beta/2, \beta/2]$ for all $|z| \leq M$.  By Theorem \ref{thm:cubic} and \cite[Lemma 2.4]{GTcirc}, it follows that a.s.
\begin{equation} \label{eq:diststtransf}
	\sup_{|z| \leq M } \sup_{|w| \leq 4\beta, \Im(w) \geq v_n} | \hat{m}_n(z,w) - m(z,w)| = O_{M,m_{2+\eta}}(v_n^4), 
\end{equation}
where $v_n$ is defined in Theorem \ref{thm:cubic}.  

For the remainder of the proof, we fix a realization in which \eqref{eq:diststtransf} holds.   The implicit constants in our asymptotic notation (such as $O,o,\Omega, \ll$) depend only on the constants $M, m_{2+\eta},d$; for simplicity, we no longer include these subscripts in our notation.  By \cite[(3.2)]{GTcirc} and \eqref{eq:diststtransf}, it follows that
$$ \sup_{|z| \leq M} \sup_{|w| \leq 4\beta, \Im(w) \geq v_n} \Im\left( \hat{m}_n(z,w) \right) \ll 1. $$
Write $w = u + \sqrt{-1}v$.  For any interval $I \subset \mathbb{R}$, we define
$$ N_I(z) := \# \left\{ 1 \leq i \leq 2dn : \lambda_i(\hat{\BH}_n(z)) \in I \right\}, $$
where $\lambda_1(\hat{\BH}_n(z)), \lambda_2(\hat{\BH}_n(z)), \ldots, \lambda_{2dn}(\hat{\BH}_n(z))$ are the eigenvalues of $\hat{\BH}_n(z)$.  We remind the reader that the eigenvalues of $\hat{\BH}_n(z)$ are given by 
$$ \pm \sigma_1\left(\frac{1}{\sqrt{n}} \hat{\BX}_n - z\BI \right), \pm \sigma_2\left(\frac{1}{\sqrt{n}} \hat{\BX}_n - z\BI\right), \ldots, \pm \sigma_{dn}\left(\frac{1}{\sqrt{n}} \hat{\BX}_n - z\BI \right). $$  
For an interval $I \subset [-\beta, \beta]$ of length $|I| = v \geq v_n$ centered at $u$, we have
$$ \sup_{|z| \leq M} \frac{N_I(z)}{8dn v} \leq \sup_{|z| \leq M} \frac{1}{2dn} \sum_{i=1}^{2dn} \frac{v}{v^2 + |u - \lambda_i(\hat{\BH}_n(z))|^2} = \sup_{|z| \leq M} \Im(\hat{m}_n(z,u+\sqrt{-1}v)) \ll 1. $$
We conclude that for any interval $I \subset [-\beta, \beta]$ of length $|I| \geq v_n$, 
\begin{equation} \label{eq:initbnd}
	\sup_{|z| \leq M} N_I(z) \ll |I| n. 
\end{equation}
Fix an interval $I \subset [-\beta, \beta]$ with $|I| \geq 10 v_n$.  Define
$$ f(y) := \frac{1}{\pi} \int_{I} \frac{v_n}{v_n^2 + (x-y)^2} dx. $$
We have
$$ \frac{1}{2dn} \sum_{i=1}^{2dn} f( \lambda_i(\hat{\BH}_n(z))) = \frac{1}{\pi} \Im \int_{I} m_n(z, u + \sqrt{-1} v_n) du $$
and
$$ \int_{\mathbb{R}} f(y) \rho_z(y) dy = \frac{1}{\pi} \Im \int_{I} m(z,u + \sqrt{-1} v_n) du $$
for any $z \in \mathbb{C}$.  Therefore, by \eqref{eq:diststtransf}, we obtain
$$ \sup_{|z| \leq M} \left|  \frac{1}{2dn} \sum_{i=1}^{4n} f( \lambda_i(\BH_n(z))) -  \int_{I} f(y) \rho_z(y) dy \right| = O(|I| v_n^4). $$
We note the following pointwise bounds:
$$ f(y) \ll \frac{v_n |I|}{\dist^2(y,I)} $$
when $y \notin I$ and $\dist(y,I) \geq |I|$, and
$$ f(y) \ll \frac{1}{1 + \dist(y,I)/v_n} $$
when $y \notin I$ and $\dist(y,I) < |I|$.  In the case where $y \in I$, we have
$$ f(y) = 1 + O \left( \frac{1}{1 + \dist(y,I^c)/v_n} \right) $$
as $\frac{1}{\pi} \frac{v_n}{v_n^2 + (x-y)^2}$ has total integral $1$.  Using these bounds, we find
$$ \sup_{|z| \leq M } \left| \int_{\mathbb{R}} f(y) \rho_z(y) dy - \int_{I} \rho_z(y)dy  \right| = O \left( v_n \log \frac{|I|}{v_n} \right).  $$
Similarly, by \eqref{eq:initbnd}, Riemann integration, and the trivial bound $N_{J} \leq 2dn$ for any interval $J$ outside of $[-\beta, \beta]$, we have
$$ \sup_{|z| \leq M} \left| \frac{1}{2dn} \sum_{i=1}^{2dn}  f( \lambda_i(\hat{\BH}_n(z))) - \frac{1}{2dn} N_I(z) \right| = O \left( v_n \log \frac{|I|}{v_n} \right). $$
Combining the bounds above, we conclude that for any interval $I \subset [-\beta, \beta]$ with $|I| \geq 10v_n$, we have
$$ \sup_{|z| \leq M} \left| \frac{1}{2dn} N_I(z) - \int_{I} \rho_z(y) dy \right| = O(v_n^4 |I|) + O\left( v_n \log \frac{|I|}{v_n} \right). $$
In particular, since $\rho_z$ is supported inside $[-\beta/2, \beta/2]$, we obtain
$$ \sup_{|z| \leq M} \frac{1}{2dn} N_{[-\beta/2, \beta/2]^{\mathsf{c}}}(z) = O(v_n \log v_n^{-1}), $$
where $[-\beta/2, \beta/2]^\mathsf{c}$ is the complement of the interval $[-\beta/2,\beta/2]$.  Thus, we have
\begin{align*}
	\sup_{|z| \leq M} \left\| \nu_{\frac{1}{\sqrt{n}} \hat{\BX}_n - z\BI} - \nu_z \right\| &\ll v_n \log v_n^{-1} + \sup_{|z| \leq M} \sup_{x \in [-\beta/2,\beta/2]} \left| \frac{1}{2dn} N_{[-\beta, x)}(z) - \int_{-\beta}^x \rho_z(y) dy \right| \\
		&\ll v_n \log v_n^{-1}. 
\end{align*}
Since this bound holds for each fixed realization in which \eqref{eq:diststtransf} holds, we obtain \eqref{eq:showsupzmnorm} a.s.  The proof of the lemma is complete.  
\end{proof}

\appendix

\section{Proof of Theorem \ref{theorem:ILOlinear:new} and Theorem \ref{theorem:linear:operator} }\label{section:ILOlinear:new:proof}
We will mainly focus on Theorem \ref{theorem:linear:operator} as the proof for Theorem \ref{theorem:ILOlinear:new} is similar. Assume that $z_{11},\dots,z_{dd}$ are random variables satisfying \eqref{eqn:covariance}.  Then as $\E|z_i|^{2+\eta}$ is bounded and as the covariance matrix $(\E z_{ij} {\bar{z}}_{i'j'})_{ij,i'j'}$ is the identity matrix (and thus is non-degenerate), we have the following fact (whose proof is left as an exercise).

\begin{claim}\label{lemma:difference} There exists a sufficiently small constant $\delta>0$ such that the following holds.
\begin{enumerate}
\item There exist measurable sets $R_1,\dots, R_{d^2}\subset  B(0,\delta^{-1}) \subset \C^{d^2}$ such that $\P((z_{11},\dots,z_{dd})\in R_i) \ge \delta$ for all $1\le i\le d^2$ and for any collection of $d^2$ vectors $\Bv_1\in R_1,\dots,\Bv_{d^2}\in R_{d^2}$, the least singular value of the matrix formed from $\Bv_1,\dots,\Bv_{d^2}$ is at least $\delta$.

\vskip .1in

\item There exist  measurable sets $R_1,\dots,R_d \subset B(0,\delta^{-1}) \subset \C^d$ such that $\P((z_{11},\dots,z_{1d})\in R_1,\dots, (z_{d1},\dots,z_{dd})\in R_d)\ge \delta$, and for any collection of $d$ vectors $\Bv_1\in R_1,\dots,\Bv_{d}\in R_{d}$, the least singular value of the matrix formed from $\Bv_1,\dots,\Bv_{d}$ is at least $\delta$. \label{difference:d}
\end{enumerate}
 \end{claim}

We remark that (1) and (2) are useful for proving Theorem  \ref{theorem:ILOlinear:new} and Theorem \ref{theorem:linear:operator} respectively. Now we sketch the main steps to prove Theorem \ref{theorem:linear:operator} following \cite{NgV}. 

First, we have
\begin{align*}
\P(\sum_{1\le i\le n} X^{(i)}\Bu^{(i)} \in B(\Bu,\beta)) &= \P\left(\exp(-\pi \|\sum_{i=1}^n X^{(i)}\Bu^{(i)}  -\Bu\|^2  \ge \exp(-\pi \beta^2)\right)\\
&\le \exp(\pi \beta^2)\E \exp(-\pi \|\sum_{i=1}^n X^{(i)}\Bu^{(i)}  -\Bu\|^2)\\
&\le \exp(\pi \beta^2) \int_{\C^d}\E e(\langle \sum_{1\le i\le n} X^{(i)}\Bu^{(i)} ,  \Bt \rangle )e(-\langle \Bu,\Bt\rangle )\exp(-\pi \|\Bt\|^2) d\Bt.
\end	{align*}


By using the independence of $X^{(i)}$ and other elementary estimates such as $|x|\le |x|^2/2+1/2$ and $|\cos(\pi x)|\le \exp(-2\|x\|_{\R/\Z}^2)$, we obtain
$$|\E e(\langle X^{(i)}\Bu^{(i)} ,  \Bt \rangle )| \le \exp\Big(-\E_{Z}\|\Re\big([z_{11}t_1+\dots+z_{d1}t_d ] \Bu^{(i)}_1+\dots+[z_{1d}t_1+\dots+z_{dd}t_d ] \Bu^{(i)}_d \big)\|_{\R/\Z}^2\Big),$$
where $X'$ is an iid copy of $X=(x_{ij})_{1\le i,j\le d}$ and $Z=X-X'$. 
By scaling  the $\Bu^{(i)}$ by a factor of $\beta^{-1}$, it is enough to assume $\beta=1$.  Set $M:= 2A \log n$ where $A$ is large enough. From the fact that $\gamma\ge n^{-O(1)}$ we easily obtain
\begin{align}\label{eqn:continuous:integral}
\frac{\gamma}{2} &\le \int_{\|\Bt\|\le M} \exp\big(-\sum_{i=1}^n\E_{Z}\|\Re([z_{11}t_1+\dots+z_{d1}t_d ] \Bu^{(i)}_1+\dots+\nonumber \\
& [z_{1d}t_1+\dots+z_{dd}t_d ] \Bu^{(i)}_d )\|_{\R/\Z}^2-\pi \|\Bt\|^2\big) d\Bt.
\end{align}

For each integer $0\le m \le M$ we define the level set

$$
S_m:= \left \{\Bt \in \C^d: \sum_{i=1}^n \E_{Z} \| \Re \big([z_{11}t_1+\dots+z_{d1}t_d ] \Bu^{(i)}_1+\dots+[z_{1d}t_1+\dots+z_{dd}t_d ] \Bu^{(i)}_d \big)\|_{\R/\Z}^2\Big) + \|\Bt\|^2 \le m  \right \}.
$$

Then it follows from \eqref{eqn:continuous:integral} that  there exists $m\le M$ such that $\mu(S_m) \ge \gamma\exp(\frac{m}{4}-2)$ and $\mu(T)\ge c\gamma \exp(\frac{m}{4}-2)m^{-2d}$, where


$$T:=\left\{\Bt \in B(0,1), \sum_{i=1}^n  \E_{Z} \| \Re \big([z_{11}t_1+\dots+z_{d1}t_d] \Bu^{(i)}_1+\dots+[z_{1d}t_1+\dots+z_{dd}t_d ] \Bu^{(i)}_d \big)\|_{\R/\Z}^2 \le 4m \right\}.$$

By a proper discretization, with $N$ a sufficiently large prime, we obtain the following discrete analog: there exists a subset $S$ of size at least $cN^{2d}\gamma \exp(\frac{m}{4}-2)m^{-2d}$ of $B_1= \{k_1/N + \sqrt{-1}k_2/N: k_1,k_2\in \Z, -2N \le k_1,k_2 \le 2N\}$ such that the following holds for any $\Bs\in S$
$$\sum_{i=1}^n  \E_{Z} \| \Re \big([z_{11}s_1+\dots+z_{d1}s_d] \Bu^{(i)}_1+\dots+[z_{1d}s_1+\dots+z_{dd}s_d ] \Bu^{(i)}_d \big)\|_{\R/\Z}^2 \le 16m.$$

Next, by the definition of $S$, 
\begin{equation}\label{eqn:double}
\E_{Z} \sum_{s\in S} \sum_{i=1}^n  \| \Re \big([z_{11}s_1+\dots+z_{d1}s_d] \Bu^{(i)}_1+\dots+[z_{1d}s_1+\dots+z_{dd}s_d ] \Bu^{(i)}_d \big)\|_{\R/\Z}^2 \le 16m|S|.
\end{equation}
It then follows that, by \eqref{difference:d} of Claim \ref{lemma:difference}, there exists a matrix $\BC=(c_{11},\dots,c_{dd})$ such that $\delta \le \sigma_d (\BC) \le \sigma_1 \le \delta^{-1}$ and the following holds for some sufficiently large constant $C$
\begin{align*}
\sum_{\Bs\in S} \sum_{i=1}^n  \| \Re \big([c_{11}s_1+\dots+c_{d1}s_d] \Bu^{(i)}_1+\dots+[c_{1d}s_1+\dots+c_{dd}s_d ] \Bu^{(i)}_d \big)\|_{\R/\Z}^2 &\le Cm|S| \\
\sum_{\Bs\in S} \sum_{i=1}^n  \| \langle \Bv^{(i)}, \Bs \rangle\|_{\R/\Z}^2 &\le Cm|S|, 
\end{align*}
where $\Bv^{(i)}:=(c_{11}\Bu^{(i)}_1+\dots + c_{1d}\Bu^{(i)}_d,\dots,c_{d1}\Bu^{(i)}_1+\dots + c_{dd}\Bu^{(i)}_d)$.   




Let $n'$ be any number between $n^{\ep}$ and $n$. We say that an index $1\le i\le n$ is {\it bad}  if $\sum_{\Bs\in S}  \|\langle \Bs,\Bv^{(i)} \rangle \|^2_{\R/\Z} \ge \frac{Cm|S|}{n'}$. Clearly the number of bad indices is at most $n'$. Let $I$ be the set of good indices, and $V$ be the set of vectors $\Bv^{(i)},i\in I$.  Recall that for an arbitrary vector $\Bv\in V$
$$\sum_{\Bs\in S} \|\langle \Bs, \Bv \rangle \|^2_{\R/\Z} \le Cdm|S|/n'.$$
Set $k:=c\sqrt{\frac{n'}{m}}$ for some sufficiently small constant $c$, and let $V_k:=k(V\cup \{0\})$. By the Cauchy-Schwarz inequality, for any $\Bv\in V_k$, we have
$$\sum_{\Bs\in S} 2\pi^2 \|\langle \Bs, \Bv \rangle \|^2_{\R/\Z} \le \frac{|S|}{2}.$$
The last estimate implies that the size of $V_k$ does not grow fast in terms of $k$. The treatment from here is identical to \cite[Section 6]{NgV} by using a Freiman-type inverse result  \cite[Theorem 3.2]{NgV}.

\end{document}